\documentclass[11pt,a4paper, reqno]{amsart}
\usepackage{mathrsfs}

\usepackage{amsmath,amsfonts,verbatim}
\usepackage{latexsym}
\usepackage{amssymb,leftidx}
\usepackage{extarrows}
\usepackage{overpic}
\usepackage{color}
\usepackage{epsfig}
\usepackage{subfigure}
\usepackage{tikz}
\usepackage{bbm}
\usepackage{tikz, caption}
\usepackage[font=small,labelfont=bf]{caption}

\usepackage[colorlinks=true
    ,linkcolor=blue%
    ,citecolor=blue%
    ,urlcolor=blue%
    ,backref=page]{hyperref}
\hypersetup{urlcolor=blue, citecolor=blue}

\usepackage{amssymb}
\usepackage{mathrsfs}
\usepackage{amscd}
\usepackage{cite}

\newcommand\R{\mathbb{R}}
\newcommand\C{\mathbb{C}}
\newcommand\Z{\mathbb{Z}}

\newcommand{\bra}[1]{\langle #1 \rangle}

\newcommand\CC{\mathcal{C}}

\newcommand{\hank}{\mathcal{H}_\nu}

\usepackage{graphicx}
\usepackage{float}

\numberwithin{equation}{section}
\newtheorem{proposition}{Proposition}[section]

\newtheorem{lemma}{Lemma}[section]
\newtheorem{theorem}{Theorem}[section]
\newtheorem{corollary}{Corollary}[section]
\newtheorem{remark}{Remark}[section]

\begin{document}
\title[Dirac in Aharonov--Bohm magnetic fields]{Dispersive estimates for Dirac equations\\ in Aharonov-Bohm magnetic fields: massless case }

\author{Federico Cacciafesta}
\address{Federico Cacciafesta:
Dipartimento di Matematica, Universit\'a degli studi di Padova, Via Trieste, 63, 35131 Padova PD, Italy}
\email{federico.cacciafesta@unipd.it}

\author{Piero D'Ancona}
\address{Piero D'Ancona:
Aldo Moro 5, 00185 Rome, Italy
Department of Mathematics "Guido Castelnuovo", Sapienza University of Rome, Piazzale Aldo Moro 5, 00185 Rome, Italy}
\email{dancona@mat.uniroma1.it}

\author{Zhiqing Yin}
\address{Zhiqing Yin:
Department of Mathematics, Beijing Institute of Technology, Beijing 100081}
\email{zhiqingyin@bit.edu.cn}

\author{Junyong Zhang}
\address{Junyong Zhang:
Department of Mathematics, Beijing Institute of Technology, Beijing 100081}
\email{zhang\_junyong@bit.edu.cn}

\thanks{%
The authors are partially supported by the MIUR PRIN project 2020XB3EFL, ``Hamiltonian and Dispersive PDEs'', and
by the Gruppo Nazionale per l'Analisi Matematica, la Probabilit\`{a} e le loro Applicazioni (GNAMPA), Project CUP E53C23001670001, by the GNAMPA project Teoria perturbativa per l'equazione di Dirac e applicazioni (2024) and by National key R\&D program of China: 2022YFA1005700, National Natural Science Foundation of China(12171031) and Beijing Natural Science Foundation(1242011). The second Author  thanks the Erwin Schr\"{o}dinger Insitute for support during the preparation of this work, Program Nonlinear Waves and Relativity\_FSD\_2024.}%



\begin{abstract}
In this paper we study the dispersive properties of a two dimensional massless Dirac equation perturbed by an Aharonov--Bohm magnetic field. Our main results will be a family of pointwise decay estimates and a full range family Strichartz estimates for the flow. The proof relies on the use of a relativistic Hankel transform, which allows for an explicit representation of the propagator in terms of the generalized eigenfunctions of the operator. 
These results represent the natural continuation of earlier research on evolution equations associated to operators with  magnetic fields with strong singularities (see \cite{DF, FFFP, FZZ} where the Schr\"odinger and the wave  equations were studied). Indeed, we recall the fact that the Aharonov--Bohm field represents a perturbation which is critical with respect to the scaling: this fact, as it is well known, makes the analysis particularly challenging.

 
\end{abstract}

\maketitle


\tableofcontents


\section{Introduction}

 Dispersive equations perturbed with electric and magnetic potentials have attracted considerable attention in the years. In this paper, we intend to investigate the validity of dispersive estimates (pointwise decay and Strichartz) for the massless Dirac equation in space dimension 2, perturbed by an Aharonov--Bohm magnetic field. The Aharonov-Bohm field is a {\em critical} perturbation with respect to the scaling of the massless Dirac operator, making its analysis particularly  challenging.
 We begin by a brief synopsis of the main properties of magnetic operators and a review of the literature.

\subsection{The electromagnetic operators}
An electromagnetic Schr\"odinger Hamiltonian takes the form
\begin{equation}\label{H-AV}
H_{A, V}=-(\nabla+iA(x))^2+V(x)\quad \text{on}\quad L^2(\R^n; \C).
\end{equation}
Here $V: \R^n\to \R$ is the scalar electric potential, while
\begin{equation}
A(x)=(A^1(x),\ldots, A^n(x)): \, \R^n\to \R^n
\end{equation}
is the magnetic vector potential. We shall always work in the Coulomb gauge
\begin{equation}\label{div0}
\mathrm{div}\, A=0.
\end{equation}
In the physical three dimensional case, the magnetic vector potential $A$ produces a magnetic field $B$, given by
\begin{equation}\label{B-3}
B(x)=\mathrm{curl} (A)=\nabla\times A(x).
\end{equation}
In arbitrary dimension $n\geq2$, the natural analogue of $B$ is the matrix--valued field $B:\R^n\to \mathcal{M}_
{n\times n}(\R)$ defined by
\begin{equation}\label{B-n}
B:=DA-DA^t,\quad B_{ij}=\frac{\partial A^i}{\partial x_j}-\frac{\partial A^j}{\partial x_i}.
\end{equation}
Since the interaction of particles and fields is the main object of study in quantum mechanics, Schr\"odinger operators with electromagnetic potentials have always played a fundamental role, and both their spectral and scattering theory have been developed extensively. We refer to Reed--Simon \cite{RS} and Avron--Herbest--Simon \cite{AHS1,AHS2,AHS3} for the fundamentals of this theory and a thorough analysis of the most important physical potentials (in particular the constant magnetic field and the Coulomb electric potential)

Finer aspects of quantum theory, like the inner structure (the \emph{spin}) of particles and relativistic effects, are not captured by the scalar model \eqref{H-AV}. To this end one needs the more involved \emph{electromagnetic Dirac operators} on  $L^2(\R^n; \C^N)$,
where $N=2^{\lceil(n+1)/2 \rceil}$. These operators are defined as
\begin{equation}\label{D-AV}
\begin{split}
\mathcal{D} _{A, V}&=-i\sum_{k=1}^n\alpha_k(\partial_k-iA_k(x))+m\beta+V (x)\\
&=-i{\bf \alpha}\cdot (\nabla-iA(x))+m\beta+V (x),
\end{split}
\end{equation}
where the electric potential $V (x)$ is an $N \times N$ matrix, 
the magnetic potential $A(x)=(A_{1}(x),\dots,A_{n}(x))$
is an $\mathbb{R}^{n}$ valued vector field, and
${\bf\alpha}=(\alpha_1,\alpha_2, \cdots,\alpha_n)$ and $\beta=\alpha_{0}$ are the $N\times N$ \emph{Dirac matrices}, satisfying
the following \emph{anticommutation relations} 
\begin{equation}\label{alpha-r}
  \alpha_j\alpha_k+\alpha_k\alpha_j=2\delta_{jk} I_{N},\quad 0\leq j,k\leq n.
\end{equation}
Here $I_{N}$ denotes the $N\times N$ identity matrix and $\delta_{jk}$ the Kronecker symbol, equal to 1 if $j=k$ and to 0 if $j\neq k$. We call \eqref{D-AV} \emph{massless} if $m=0$, \emph{massive} if $m\neq0$. 

In spatial dimension $n=2$, so that $N=2$,
there are at most three linearly independent anticommuting matrices. A standard choice is given by
$\alpha_1=\sigma_1$, $\alpha_2=\sigma_2$ and $\beta=\sigma_3$, 
where $\sigma_{j}$ are the \emph{Pauli matrices}
\begin{equation}
\sigma_1=\left(\begin{array}{cc}0 & 1 \\1 & 0\end{array}\right),\quad
\sigma_2=\left(\begin{array}{cc}0 &-i \\i & 0\end{array}\right),\quad
\sigma_3=\left(\begin{array}{cc}1 & 0\\0 & -1\end{array}\right).
\end{equation}
With this choice, the magnetic Dirac operator \eqref{D-AV} in
dimension $n=2$ takes the form
\begin{equation}\label{D-AV2}
\begin{split}
\mathcal{D} _{A, V}&=-i\sum_{k=1}^2\sigma_k(\partial_k-iA^k(x))+\sigma_3 m+V (x)\\
  &=-\begin{pmatrix}
    -m & (i\partial_{1}+A^{1}-i(i \partial_{2}+A^{2})\\
    (i\partial_{1}+A^{1}+i(i \partial_{2}+A^{2}) &  m
  \end{pmatrix}+V (x).
\end{split}
\end{equation}

\subsection{The Aharonov--Bohm magnetic field}
The main focus of this paper is the massless Dirac operator on $\mathbb{R}^{2}$ perturbed by the Aharonov--Bohm (AB in the following) magnetic potential. This field is used to model the so called
\emph{AB effect} \cite{AB59}, which plays an important role in quantum theory. 
The AB effect is the interaction of a non-relativistic charged particle with an infinitely long and infinitesimally thin magnetic solenoid field. 
According to classical electrodynamics, the particle should not
be influenced by the solenoid, since its magnetic field is zero outside the singularity. However the particle is scattered, and this is interpreted as an evidence of quantum behaviour of the particle.

In the mathematical model, the {\em Aharonov-Bohm potential}
 $A(x)=(A^1(x),A^2(x))$ is 
\begin{equation}\label{AB}
A:\R^2\setminus\{(0,0)\}\to\R^2,
\quad
A(x)=\frac{\alpha}{|x|}\left(-\frac{x_2}{|x|},\frac{x_1}{|x|}\right),\, \quad
x=(x_1,x_2)
\end{equation}
where the real constant $\alpha$ is called the \emph{magnetic flux}
(for example in \cite{AB59,AT} the choice is
$\alpha\equiv \text{magnetic flux}/(2\pi/e)$;
while $\alpha=e/Q_{\text {Higgs}}$ for the cosmic strings
studied in \cite{AW}). Note that we have obviously,
for all nonzero $x\in \mathbb{R}^{2}$,
\begin{equation}\label{eq:transversal}
  \textstyle
A(x)\cdot\hat{x}=0,
\qquad
\hat{x}=\big(\frac{x_1}{|x|},\frac{x_2}{|x|}\big)\in\mathbb{S}^1.
\end{equation}

With the AB field one can form a second order Hamiltonian,
as in \cite{AT}, or a Dirac Hamiltonian, like in 
\cite{AW} where it was used to model the scattering of fermions by an infinitely thin flux tube.
In the typical cosmic string scenarios observed by Alford and Wilczek \cite{AW}, the fermionic charges can be noninteger multiples of the Higgs charges.
As the flux is quantized with respect to the Higgs charge, this leads to a non-trivial Aharonov-Bohm scattering of these fermions.  
The construction of self-adjoint Schr\"odinger and Dirac operators with the AB field have been studied in detail, see \cite{GTV} and the references therein.

\subsection{The time dependent Dirac equation} 
The evolution equation associated to \eqref{D-AV} is
\begin{equation}
i\partial_t u(t,x)=\mathcal{D} _{A, V}\, u(t,x)
\end{equation}
where $u(t,x)$ is the vector valued wavefunction
\begin{equation}
u(t,x) =\begin{pmatrix}
    u_1(t,x) \\
    \vdots\\
      u_{N}(t,x)
    \end{pmatrix}\in \C^N.
\end{equation}
It is customary to use the term \emph{spinor} to denote elements
of $\mathbb{C}^{N}$ in this context.
If the system at time $t=0$ is in the state $u_0(x)$ and if $\mathcal{D}_{A, V}$ is a self-adjoint operator, then the state at time $t$ can be
represented as a continuous group
$$u(t,x)=e^{-it\mathcal{D}_{A, V}} u_0(x)$$ 
and by Stone's theorem, $u(t,x)$ is the unique strong solution of the Cauchy problem
\begin{equation}\label{equ:DAV}
\begin{cases}
i\partial_t u(t,x)=\mathcal{D}_{A, V}\, u(t,x)\\
u(t,x)|_{t=0}=u_0(x)\in D(\mathcal{D}_{A, V})\subset L^2
\end{cases}
\end{equation}
The purpose of this paper is to study the dispersive behavior of the solution of the Cauchy problem \eqref{equ:DAV} where $A(x)$ is the AB potential \eqref{AB} and $m=V=0$.

The study of decay and Strichartz estimates for dispersive equations has a long history and an extensive literature, due to their central role in harmonic analysis and applications to PDEs. In particular, Schr\"{o}dinger, wave and Klein--Gordon equations with electromagnetic potentials have been investigated in great detail
(see e.g.~\cite{BPSS, BPST, CS,DF, DFVV, EGS1, EGS2, S} and the references therein). The picture is rather clear for wide classes of potentials, however it is still far from complete for potentials of critical decay or singularity. A notable example is given by the Aharonov-Bohm potential, whose study was initiated only recently.
In \cite{FFFP, FFFP1}, Fanelli, Felli, Fontelos, and Primo studied pointwise decay for the Schr\"odinger flow perturbed by a class of singular magnetic potentials, including the Aharonov-Bohm potential.
However, the argument in \cite{FFFP, FFFP1} breaks down for the wave equation due to the lack of pseudoconformal invariance (which was crucial for the Schr\"odinger equation). Very recently, Fanelli, Zheng and the last author \cite{FZZ}
established Strichartz estimates for the wave equation in 2D by constructing the odd sin propagator. 
Gao, Yin, Zheng and the last author \cite{GYZZ} constructed the spectral measure and further proved the time decay and Strichartz estimates of Klein-Gordon equation. 

For the Dirac equation the situation is even harder, due to the rich algebraic structure of the equation: Strichartz estimates for ``small" (i.e. subcritical) potential perturbations were studied in \cite{cacdan, DF,DF2,EG, EGG, EGT} by perturbative techniques, but these approaches break down scaling critical models. To the best of our knowledge, the only available results in this direction are provided in \cite{cacser,cacserzha,danesi} and \cite{CF,CYZ}, in which some local smoothing and Strichartz estimates with angular regularity are proved for, respectively, the Coulomb potential and the AB field.

\subsection{Main results}
In the following we shall focus on the model
\begin{equation}\label{eq:dirac}
\begin{cases}
\displaystyle
 i\partial_tu=\mathcal{D}_Au,\quad u(t,x):\mathbb{R}_t\times\mathbb{R}_x^2\rightarrow\mathbb{C}^{2}\\
u(0,x)=u_0(x)
\end{cases}
\end{equation}
where $\mathcal{D}_A$ is given by \eqref{D-AV2} with $m=V=0$ and $A$ as in \eqref{AB}. Explicitly,
\begin{equation}\label{op:D}
\begin{split}
\mathcal{D}_{A}&=-\begin{pmatrix}
    0 & (i\partial_{1}+A^{1})-i(i \partial_{2}+A^{2})\\
    (i\partial_{1}+A^{1})+i(i \partial_{2}+A^{2}) &  0
  \end{pmatrix}.
\end{split}
\end{equation}

The AB potential \eqref{AB} is homogeneous of degree $-1$, like the unperturbed massless Dirac operator $\mathcal{D}$; thus $A$ is a {\em scaling critical perturbation} for $\mathcal{D}$.
PDEs with scaling critical potentials have attracted much attention recently. Criticality makes perturbative techniques ineffective, and requires the introduction of ``ad hoc" methods. We refer to \cite{CF} and references therein for further details on the model. Notice that in order to have a scaling-invariant structure, the zero mass assumption is crucial. The analysis in the massive case requires different techniques and will be the object of a future work.

The first question to address is {\em self-adjointness}. 
Although this topic is fairly understood (see e.g. \cite{gerb} and \cite{borrelli}), we discuss it briefly since the explicit form of the domain plays a crucial role in our results. As proved in Section \ref{sec:pre}, for any given $\alpha$, the operator $\mathcal{D}_A$ on $D(\mathcal{D}_{A})= \big[C_{c}^{\infty}(\mathbb{R}^{2}\setminus0)\big]^2$ admits a one--parameter family of self-adjoint extensions $\mathcal{D}_{A,\gamma}$ indicized by $\gamma\in[0, 2\pi)$. For the scalar Schr\"odinger operator with Aharonov-Bohm potential, the Friedrichs extension is the natural choice for a {\em distinguished} self-adjoint extension; this choice is not available for Dirac, since the spectrum is unbounded from below. As a consequence, it is not entirely clear how a distinguished self adjoint extension for the operator $\mathcal{D}_A$ should be defined. As we shall see in Theorem \ref{ABsa}, the choice of the parameter $\gamma$ determines the boundary conditions at the origin of the domain of $\mathcal{D}_{A,\gamma}$. More precisely, given $\gamma\in[0, 2\pi)$ the elements of $D(\mathcal{D}_{A})$ can be represented as
\begin{equation*}
  f=
  c
  \begin{pmatrix}
    \cos \gamma \cdot K_{\alpha}(r) \\
    i\sin \gamma \cdot K_{1-\alpha}(r)   e^{i\theta}
  \end{pmatrix}
  + \begin{pmatrix}
  \phi_{0}(r) \\
    e^{i\theta}\ \psi_{0}(r)
  \end{pmatrix}
+
  \sum_{k\in \mathbb{Z}\setminus\{0\}}
  \begin{pmatrix}
    e^{ik \theta}\ \phi_{k}(r) \\
    e^{i(k+1)\theta}\ \psi_{k}(r)
  \end{pmatrix},
\end{equation*}
where $K_{\alpha}$ is the Macdonald (modified Bessel) function,
$c\in \mathbb{C}$ and, for each $k\in\Z$, 
$\phi_{k},\psi_{k}$ are regular at the origin and such that
\begin{equation*}
  \phi_{k},\psi_{k}\in
    \{\phi\in L^{2}(rdr):
    \phi',\phi/r\in L^{2}(rdr),\ 
    \phi\in C([0,+\infty)),\ 
    \phi(0)=0\}.
\end{equation*}
Then, explicit computations show that the action of $\mathcal{D}_{A,\gamma}$ on $f$ is given
 by
\begin{equation*}
  \mathcal{D}_{A,\gamma}f=
  c
  \begin{pmatrix}
    \sin \gamma \cdot K_{\alpha}(r) \\
    -i\cos \gamma\cdot  K_{1-\alpha}(r) e^{i\theta}
  \end{pmatrix}
  +
  \begin{pmatrix}
    i \partial_{1-\alpha}\psi_{0}(r) \\
    ie^{i\theta}\ \partial_{\alpha}\phi_{0}(r)
  \end{pmatrix}+
  \sum_{\{k\in\Z:k\neq 0\}}
  \begin{pmatrix}
    ie^{ik \theta}\ \partial_{k+1-\alpha}\psi_{k}(r) \\
    ie^{i(k+1)\theta}\ \partial_{\alpha-k}\phi_{k}(r)
  \end{pmatrix}
\end{equation*}
where 
$\partial_{\alpha}=\partial_{r}+\frac \alpha r$, with
$\partial_{r}$ being the radial derivative.
In particular, we see that $\mathcal{D}_{A,\gamma}$ acts on the singular component 
$$
f_{sing}(x)=\begin{pmatrix}
  \cos \gamma \cdot K_{\alpha}(r) \\
  i\sin \gamma \cdot K_{1-\alpha}(r)   e^{i\theta}
\end{pmatrix}
\qquad\text{as follows:}\qquad 
\mathcal{D}_{A,\gamma} f_{sing}=
\begin{pmatrix}
    \sin \gamma \cdot K_{\alpha}(r) \\
    -i\cos \gamma\cdot  K_{1-\alpha}(r)   e^{i\theta}
  \end{pmatrix}.
$$
In what follows, we shall pursue the following choice: for any $\alpha\in(0,1)$ we shall take $\gamma\in[0,2\pi)$ such that 
\begin{equation}\label{eq:convention}
  \ \text{either}\ \sin\gamma=0
  \ \text{if}\ \alpha\in (0,1/2],
  \quad\text{or}\quad 
  \cos\gamma=0
  \ \text{if}\ \alpha\in (1/2,1).
\end{equation}
This choice eliminates the most singular part of the bad component
and ensures that the surviving part behaves like
$\lesssim r^{-1/2}$ as $r\to0$. Equivalently, for any value of $\alpha$ we are considering the (unique) self-adjoint extension such that the domain of the operator is contained in $H^{1/2}$. This of course represents a natural choice, it is consistent with existing results (see e.g. \cite{cuesied}, \cite{dolestlos}), and may be interpreted as the choice a \emph{distinguished} self-adjoint extension for Dirac.
We postpone all the details to forthcoming Section \ref{sec:3} (see also Remark \ref{rk:conjecture}). Thus, we shall focus on the problem
\begin{equation}\label{eq:Dirac1}
\begin{cases}
i\partial_t u+\mathcal{D}_{A,\gamma} u=0;\\
u|_{t=0}=f(x) \in D(\mathcal{D}_{A, \gamma})\subset (L^2(\R^2))^2
\end{cases}
\end{equation}
where by convention
$\mathcal{D}_{A,\gamma}$ is the von Neumann selfadjoint extension of the operator defined in 
\eqref{op:D} with $\gamma$ satisfying \eqref{eq:convention}.
 \medskip

In order to state our main results, we introduce the 
orthogonal projections on $L^{2}$
\begin{equation}\label{def-pro}
  P_0:
  L^{2}(\mathbb{R}^{2})^{2}\to   
  L^{2}(rdr)^{2}\otimes h_{0}(\mathbb{S}^{1}),
  \qquad
  P_{\bot}=I-P_{0}.
\end{equation}
It will be also useful to further split
$P_{\bot}=P_{>}+P_{<}$ where
\begin{equation*}
\begin{split}
  P_{>}:
  L^{2}(\mathbb{R}^{2})^{2}&\to
  \bigoplus_{0< k\in \mathbb{Z}}
  L^{2}(rdr)^{2}\otimes h_{k}(\mathbb{S}^{1}),\\
  P_{<}:
  L^{2}(\mathbb{R}^{2})^{2}&\to
  \bigoplus_{0>k\in \mathbb{Z}}
  L^{2}(rdr)^{2}\otimes h_{k}(\mathbb{S}^{1}).
  \end{split}
\end{equation*}
Here the space $h_{k}(\mathbb{S}^{1})$ is the linear
span of 
$\big(\begin{smallmatrix} e^{ik \theta} \\ e^{i(k+1)\theta}
\end{smallmatrix}\big)$
defined in \eqref{eq:hk} below. Then we can prove the following
pointwise dispersive estimate, with a decay rate and regularity
loss similar to the case of the unperturbed equation.
In order to avoid technical issues, we state the result
as a separate estimate for each frequency component
of the solution; 
functions $\varphi(\mathcal{D}_{A,\gamma})$ of the operator are
defined via the spectral theorem.
In the statement we use the matrix of weight functions
\begin{equation}\label{weig-Wj1}
W_j(|x|)=\begin{cases}
\left(\begin{matrix} (1+2^j|x|^{-\alpha})^{-1}& 0 \\0 &  1\end{matrix}\quad \right),\quad \alpha\in(0,\frac12];\\
\left(\begin{matrix} 1& 0 \\0 &  (1+2^j|x|^{\alpha-1})^{-1}\end{matrix}\quad \right),\quad \alpha\in(\frac12,1).
\end{cases}
\end{equation}

 \begin{theorem}[Time decay estimates]\label{thm:disper}  
  Let $\alpha\in(0,1)$, $\varphi\in\mathcal{C}_c^\infty([1,2])$,
  $\tilde\varphi\in\mathcal{C}_c^\infty([1/2,4])$ such that
  $\varphi\tilde \varphi=\varphi$,
  and let $\mathcal{D}_{A,\gamma}$ be the
  self--adjoint extension of $\mathcal{D}_A$ selected in
  \eqref{eq:convention}.
  Then for any $f\in [L^1{(\R^2)}]^2\cap [L^2{(\R^2)}]^2$
  and any $j\in\Z$, there exists a constant $C$ such that
  \begin{equation}\label{est:dis1}
  \begin{split}
  \Big\|e^{it\mathcal{D}_{A,\gamma}}&\varphi(2^{-j}|\mathcal{D}_{A,\gamma}|)P_{\bot}f(x)\Big\|_{[L^\infty(\R^2)]^2}\\&
  \leq C2^{2j}(1+2^{j}|t|)^{-\frac12} 
  \|\tilde\varphi(2^{-j}|\mathcal{D}_{A,\gamma}|)P_{\bot}f\|_{[L^1{(\R^2)]^2}},
  \end{split}
  \end{equation}
  and
  \begin{equation}\label{est:dis0}
  \begin{split}
  \Big\|W_j(|x|) e^{it\mathcal{D}_{A,\gamma}}&\varphi(2^{-j}|\mathcal{D}_{A,\gamma}|)P_{0} W_j(|x|)f(x)\Big\|_{[L^\infty(\R^2)]^2}\\&
  \leq C2^{2j}(1+2^{j}|t|)^{-\frac12} 
  \|\tilde\varphi(2^{-j}|\mathcal{D}_{A,\gamma}|)P_{0}f\|_{[L^1{(\R^2)]^2}}.
  \end{split}
  \end{equation}
  where the weight function matrix $W_{j}$ is given by 
  \eqref{weig-Wj1}.

  Moreover, let $q(\alpha)$ be given by
  \begin{equation}\label{q-alp}
  q(\alpha)=
  \begin{cases} 
    \frac2{\alpha}&
    \text{if}\quad \alpha\in(0,\frac12]\\
    \frac2{1-\alpha}&
    \text{if}\quad \alpha\in(\frac12,1).
  \end{cases}
  \end{equation}
  Then for $2\leq q<q(\alpha)$ one has the weightless estimate
  \begin{equation}\label{est:disq}
    \begin{split}
    \Big\|e^{it\mathcal{D}_{A,\gamma}}&\varphi(2^{-j}|\mathcal{D}_{A,\gamma}|)P_{0} f(x)\Big\|_{[L^q(\R^2)]^2}\\&
    \leq C2^{2j(1-\frac2q)}(1+2^{j}|t|)^{-\frac12(1-\frac2q)} 
    \|\tilde\varphi(2^{-j}|\mathcal{D}_{A,\gamma}|)P_{0}f\|_{[L^{q'}{(\R^2)]^2}}.
    \end{split}
  \end{equation}
\end{theorem}

\begin{remark}\rm
Notice that from \eqref{est:dis1}, \eqref{est:dis0} and the $L^2$-estimates, one can obtain by standard complex interpolation the following additional families of time decay estimates for any $q\geq 2$:
 \begin{equation}\label{est:dis3}
  \begin{split}
  \Big\|e^{it\mathcal{D}_{A,\gamma}}&\varphi(2^{-j}|\mathcal{D}_{A,\gamma}|)P_{\bot}f(x)\Big\|_{[L^q(\R^2)]^2}\\&
  \leq C2^{2j(1-\frac2q)}(1+2^{j}|t|)^{-\frac12(1-\frac2q)} 
  \|\tilde\varphi(2^{-j}|\mathcal{D}_{A,\gamma}|)P_{\bot}f\|_{[L^{q'}{(\R^2)]^2}},
  \end{split}
  \end{equation}
  and
  \begin{equation}\label{est:dis4}
  \begin{split}
  \Big\|W_j(|x|)^{1-\frac2q} e^{it\mathcal{D}_{A,\gamma}}&\varphi(2^{-j}|\mathcal{D}_{A,\gamma}|)P_{0} W_j(|x|)^{1-\frac2q} f(x)\Big\|_{[L^q(\R^2)]^2}\\&
  \leq C2^{2j(1-\frac2q)}(1+2^{j}|t|)^{-\frac12(1-\frac2q)} 
  \|\tilde\varphi(2^{-j}|\mathcal{D}_{A,\gamma}|)P_{0}f\|_{[L^{q'}{(\R^2)]^2}}.
  \end{split}
  \end{equation}

\end{remark}
As a natural application of Theorem \ref{thm:disper},
we can prove Strichartz estimates for \eqref{eq:Dirac1}.
The regularity loss in the estimates will be expressed in terms of suitable Sobolev norms adapted to the operator
$\mathcal{D}_{A}$. For $s=1$ and Schwartz class functions
$\phi$ we define the (semi)norms
\begin{equation}\label{eq:nomrH1A}
  \|\phi\|_{\dot H^{1}_{A}}=
  \|(\partial-iA(x))\phi\|_{L^{2}},
  \qquad
  \|\phi\|_{H^{1}_{A}}^{2}=
  \|\phi\|_{\dot H^{1}_{A}}^{2}+\|\phi\|_{L^{2}}^{2}
\end{equation}
and the corresponding spaces $\dot H^{1}_{A}$, $H^{1}_{A}$
by completion. Next, for all $|s|\le1$ the
spaces
$\dot H^{s}_{A}$ and $H^{s}_{A}$ may be defined by duality
and complex interpolation as usual; equivalently, they can also
be defined using the relativistic Hankel transform,
see Section \ref{sec:relhan} for details. 

For the $P_{\bot}$ component of the flow (the {\em regular component}, which in fact does not contain
singularities), the following result is a standard consequence of \eqref{est:dis1}.

\begin{theorem}[Strichartz estimates for the regular component]
  \label{thm-stri1}
  Let $\alpha\in(0,1)$ and let $\mathcal{D}_{A,\gamma}$ be the
  self--adjoint extension of $\mathcal{D}_A$ selected in
  \eqref{eq:convention}.
  Then for any $f\in [\dot{H}^{s}_A]^2$ and $p>2$, 
  the following Strichartz estimates hold
  \begin{equation}\label{stri-D1} 
    \|e^{it\mathcal{D}_{A,\gamma}}P_{\bot}f\|
    _{[L^p_t(\R; L^q_x(\R^2))]^2}
    \leq C \| f\|_{[\dot{H}^{s}_A]^2},
    \qquad s=1-\frac1p-\frac2q
  \end{equation}
  provided $(p,q)$ satisfy
  \begin{equation}\label{pqrange1}
    (p,q)\in (2,\infty]^2 
    \quad\text{and}\quad 
    \frac2p+\frac1q\leq\frac12, 
    \qquad \text{or}\quad 
    (p,q)=(\infty, 2).
  \end{equation}
\end{theorem}

On the other hand, Strichartz estimates with initial condition $P_0f$ (the {\em singular component}) are more delicate, due to the singularity at the origin.
For a restricted range of indices $(p,q)$ we can prove the
following result;
note that the component $k=0$ is a radial component and we
expect a wider range of indices as in \eqref{pqrange0}, 
as it is usual for radial Strichartz estimates.

\begin{theorem}[Strichartz estimates for the singular component]
  \label{thm-stri3}
  Let $\alpha\in(0, 1/2)\cup (1/2,1)$ and let $\mathcal{D}_{A,\gamma}$ be the
  self--adjoint extension of $\mathcal{D}_A$ selected in
  \eqref{eq:convention}. Assume $(p,q)$ satisfy
  \begin{equation}\label{pqrange0}
  (p,q)\in (2,\infty]^2 \quad{\rm and}\quad \frac1p+\frac1q<\frac12, \qquad \text{or}\quad (p,q)=(\infty, 2).
  \end{equation}
  Assume in addition that $q<q(\alpha)$, where $q(\alpha)$
  is given by \eqref{q-alp}.
  Then for any $f\in [\dot{H}^{s}_A]^2$ the following 
  Strichartz estimates hold:
  \begin{equation}\label{stri-D3}
    \|e^{it\mathcal{D}_{A,\gamma}}P_{0}f\|
    _{[L^p_t(\R; L^q_x(\R^2))]^2}
    \leq C \|P_{0} f\|_{[\dot{H}^{s}_A]^2},
    \qquad
    s=1-\frac 1p-\frac 2q.
  \end{equation}
\end{theorem}

\begin{remark}
In Theorem above the case $\alpha=1/2$ is excluded. The reason for this restriction is purely technical: as we shall see, the couple $(p,q)$ needs to satisfy some condition, \eqref{condfinal}, that can not be satisfied if $\alpha=1/2$. Let us mention that this case could be dealt with by following the argument developed in the subsequent work \cite{cosmic}, in which the Dirac equation in a cosmic string spacetime was studied, but this would require additional amount of technicalities that we prefer to omit.
\end{remark}

As a consequence of Theorems \ref{thm-stri1} and \ref{thm-stri3}, we conclude:

\begin{corollary}[Strichartz estimates for the complete flow]
  \label{cor-stri}
  Let $\alpha\in(0, 1/2)\cup (1/2,1)$, $p,q$ and $s$ be as in Theorem \ref{thm-stri1} and $q<q(\alpha)$,
  and let $\mathcal{D}_{A,\gamma}$ be the
  self--adjoint extension of $\mathcal{D}_A$ selected in
  \eqref{eq:convention}. Then the following Strichartz estimates hold
  \begin{equation}\label{stri}
  \|e^{it\mathcal{D}_{A,\gamma}}f\|_{[L^p_t(\R; L^q_x(\R^2))]^2}\leq
  C
  \| f\|_{[\dot{H}^{s}_A]^2}.
  \end{equation}
  The restriction $q<q(\alpha)$ is necessary in the sense that \eqref{stri} fails if  $q\geq q(\alpha)$.
\end{corollary}

In fact, it is possible to recover the full range of 
admissible pairs for the singular component if
we introduce a suitable weight in the estimates, like in 
Theorem \ref{thm:disper}: 
\begin{equation}\label{weig-Wj}
  W(|x|)=\begin{cases}
    \begin{pmatrix} 
      (1+|x|^{-\alpha-\epsilon})^{-1}& 0 \\
      0 &  1
    \end{pmatrix}
    & \alpha\in(0,\frac12]\\
    \begin{pmatrix} 
      1& 0 \\
      0 &  (1+|x|^{\alpha-1-\epsilon})^{-1}
    \end{pmatrix}
    & \alpha\in(\frac12,1)
\end{cases}
\end{equation}
where $0<\epsilon\ll 1$ is a fixed small parameter.
We obtain the following result:

\begin{theorem}[Weighted Strichartz estimates for the singular component]\label{thm-stri2}
  Let $\alpha\in (0,1)$ and let $\mathcal{D}_{A,\gamma}$ be the
  self--adjoint extension of $\mathcal{D}_A$ selected in
  \eqref{eq:convention}. Assume $(p,q)$ satisfy \eqref{pqrange0}
  and $q(\alpha)\le q<\infty$, where
  $q(\alpha)$ is given by \eqref{q-alp}.
  Denote by $W(|x|)$ the weight \eqref{weig-Wj}.
  Then for any $f\in [{H}^{s}_A]^2$ 
  and any $\theta>1-\frac{q(\alpha)}{q}$ the following weighted
  Strichartz estimates hold:
  \begin{equation}\label{stri-D2b}
    \|W(|x|)^{\theta}e^{it\mathcal{D}_{A,\gamma}}P_{0}f\|
    _{[L^p_t(\R; L^q_x(\R^2))]^2}
    \leq C_{\epsilon,\theta} \|P_{0} f\|_{[{H}^{s}_A]^2},
    \qquad s=1-\frac1p-\frac2q.
  \end{equation}
\end{theorem}


The conditions on $(p,q)$ are represented in the figure.
Estimates \eqref{stri-D1} hold for $(p,q)$ in the triangle $AOC$, the weighted Strichartz estimates \eqref{stri-D2b} hold for  $(p,q)$ in the triangle $AOB$.
On the other hand,
estimates \eqref{stri-D3} hold for $(p,q)$ in the triangle $ADF$
and the classical estimates \eqref{stri} for $(p,q)$ in the triangle $ADE$.
 
 \begin{center}
  \begin{tikzpicture}[scale=1]
 \draw[->] (4,0) -- (8,0) node[anchor=north] {$\frac{1}{p}$};
 \draw[->] (4,0) -- (4,4)  node[anchor=east] {$\frac{1}{q}$};

 \draw  (4.1, -0.1) node[anchor=east] {$O$};
 \draw  (7, 0) node[anchor=north] {$\frac12$};
 \draw  (4, 3) node[anchor=east] {$\frac12$};
 \draw  (5.5, 0) node[anchor=north] {$\frac14$};

 \draw[thick] (4,3) -- (7,0);  
 \draw[red,thick] (4,3) -- (5.5,0);
 \draw[red, dashed,thick] (4,3) -- (7,0); 

 \draw (3.9,3.15) node[anchor=west] {$A$};
 \draw (6.9,0.2) node[anchor=west] {$B$};
 \draw (5.5,-0.2) node[anchor=west] {$C$};
 \draw (6,2.6) node[anchor=west] {$\frac{1}{p}+\frac{1}{q}=\frac12$};

 \draw (6,1.6) node[anchor=west] {$\frac{2}{p}+\frac{1}{q}=\frac12$};

 \draw (7,0) circle (0.06);

 \filldraw[fill=gray!30](4,3)--(5.51,0)--(7,0); 
 \filldraw[fill=gray!50](4,3)--(5.49,0)--(4,0); 


 \draw[<-] (5,2.1) -- (6,2.6) node[anchor=south]{$~$};
 \draw[<-] (4.86,1.3) -- (6,1.6) node[anchor=south]{$~$};

 \draw[red, dashed, thick] (4,1) -- (6,1); \draw (4,1) circle (0.06); \draw (6,1) circle (0.06); \draw (5,1) circle (0.06);
 \draw  (3.6,1.6 ) node[anchor=north] {$\frac1{q(\alpha)}$};
 \draw (3.8, 1.2) node[anchor=west] {D};
 \draw (4.8, 1.2) node[anchor=west] {E};
 \draw (5.9, 1.2) node[anchor=west] {F};

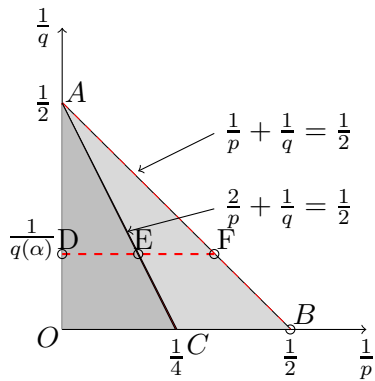
 \path (6,-1.5) node(caption){Diagrammatic picture of the admissible range of $(p,q)$.};  
 \end{tikzpicture}
 \end{center}
 

\begin{remark}\label{rem:square}\rm
  A common method to study dispersive properties of Dirac
  equations is to square them and reduce the equation to
  a system of wave equations, with principal part in diagonal
  form (see e.g.~ \cite{DAnconaOkamoto17-a}).
  This works well for smooth potential perturbations.
  However the method is ineffective for the Dirac AB
  equation, since the square of $\mathcal{D}_{A,\gamma}$
  for every value of $\gamma$ gives the same selfadjoint 
  operator
  $L=(\begin{smallmatrix} H & 0 \\ 0 & H \end{smallmatrix})$,
  where $H$ is the Friedrichs extension of the Schr\"{o}dinger
  operator with AB potential; see the beginning of
  Section \ref{sec:proofdisp}. This destroys all information
  on the singular part of $\mathcal{D}_{A,\gamma}$.
  Thus we are forced to investigate dispersive properties of 
  the Dirac propagator directly.
\end{remark}

\begin{remark}\rm
  Note that the range of admissible couples $(p,q)$ in \eqref{pqrange0} is wider than \eqref{pqrange1}. This is not surprising, since the singular component is one dimensional and the range is the typical one for radial initial data. Let us stress the fact that the proof of \eqref{pqrange0} contains new ideas since we need to prove Littlewood-Paley square function estimates for the singular component. On the other hand, estimate 
  \eqref{stri-D2b} is weaker than \eqref{stri-D1}, since the homogeneous Sobolev norm at the right hand side is replaced by the inhomogeneous norm. This is due to the loss of regularity in low frequencies caused by the weight.
\end{remark}

\begin{remark}\label{rk:conjecture}\rm
  Corollary \ref{cor-stri} and its proof suggest the following general conjecture: singularities of the generalized eigenfunctions are an obstruction for the validity of Strichartz estimates, namely, the stronger is the singularity of the generalized eigenfunctions, the smaller is the range (from the above) for the admissible Strichartz exponents. As a consequence, this would imply that our choice of the self-adjoint extension for $\mathcal{D}_{A,\gamma}$ provides the widest range of admissible exponents. This topic will be the object of further investigations. 
\end{remark}

\medskip

{\bf Plan of the paper.}
The paper is organized as follows. Section \ref{sec:pre} is devoted to an overview of the spectral theory of the Dirac operator in the Aharonov-Bohm magnetic field. In particular, in Theorem \ref{ABsa} we shall provide all the self-adjoint extensions for the operator $\mathcal{D}_A$. In Section \ref{sec:3} we shall derive the generalized eigenfunctions of the operator $\mathcal{D}_A$ and introduce the relativistic Hankel transform, which will represent the stepping stone for proving our dispersive estimate, as indeed it allows for an explicit representation of the propagator, which will be developed in Section \ref{sec:prop}. In Section \ref{sec:proofdisp} we shall then present the proof of Theorem \ref{thm:disper}, while in Section \ref{sec:stri} we collect the proof of all of our Strichartz estimates (Theorems \ref{thm-stri1}, \ref{thm-stri3}, \ref{thm-stri2} and Corollary \ref{cor-stri}).

\section{Spectral theory of the Dirac operator}\label{sec:pre}

In this section we construct the selfadjoint extensions of the Dirac AB Hamiltonian. The following results are partly known, but we reshape them in a form suitable for our applications.

We  recall the definition of \emph{von Neumann extensions}
of a closed symmetric operator $T$ on a
Hilbert space $\mathcal{H}$, with dense domain $D(T)$.
The basic result of von Neumann's theory states that the domain
of the adjoint operator $T^{*}$ splits in a direct sum
\begin{equation*}
  D(T^{*})=D(T)\oplus N_{+}\oplus N_{-}
  \quad\text{where}\quad N_{\pm}=\ker(T^{*}\mp i)
\end{equation*}
and the action of $T^{*}$ can be described as follows:
\begin{equation*}
  T^{*}(f+g_{+}+g_{-})=
  Tf+ig_{+}-ig_{-}
  \qquad
  \forall f\in D(T),\ g_{\pm}\in N_{\pm}.
\end{equation*}
Then
$n_{\pm}=\dim N_{\pm}\in \mathbb{N}_{0}\cup\{+\infty\}$ 
are called the \emph{deficiency indices} of $T$, and satisfy:
\begin{itemize}
  \item 
  $n_{+}=n_{-}=0$ iff $T$ is selfadjoint
  \item 
  $n_{+}=n_{-}$ iff $T$ has a selfadjoint extension
  \item 
  $n_{+}\neq n_{-}$ iff $T$ has no selfadjoint extension.
\end{itemize}
The case of interest is $n_{+}=n_{-}$.
Then all selfadjoint extensions
$T_{V}$ of $T$ are in 1--1 correspondence with
isometries $V$ from $N_{+}$ onto $N_{-}$, 
and we can write explicitly
\begin{equation*}
  D(T_{V})=\{f+g+Vg:f\in D(T),\ g\in N_{+}\}
\end{equation*}
and
\begin{equation*}
  T_{V}(f+g+Vg)=Tf+ig-iVg.
\end{equation*}
In particular when $n_{+}=n_{-}<\infty$ we obtain a family of 
selfadjoint extensions with $n_{+}^{2}$ parameters.

We want to apply the previous notions to the massless 
Dirac AB operator
\begin{equation}\label{DA}
  \mathcal{D}_{A}=
  \begin{pmatrix}
    0 & (i\partial_{1}+A_{1}-i(i \partial_{2}+A_{2})\\
    (i\partial_{1}+A_{1}+i(i \partial_{2}+A_{2}) &  0
  \end{pmatrix}
\end{equation}
initially
defined on $D(\mathcal{D}_{A})= \big[C_{c}^{\infty}(\mathbb{R}^{2}\setminus0)\big]^2$,
where, for a fixed $\alpha\in(0,1)$,
\begin{equation*}
  A(x_{1},x_{2})=(A_1(x_1,x_2), A_2(x_1,x_2))=\frac{\alpha}{|x|^{2}}(-x_{2},x_{1}).
\end{equation*}
Some notations will help simplify computations.
We shall write
\begin{equation}\label{notation-der}
  \partial_{\sigma}=\partial_{r}+\frac{\sigma}{r},
  \qquad
  \partial_{r}=\frac{d}{dr},
  \qquad
  \sigma\in \mathbb{R},
\end{equation}
and 
$f'(r)=\frac{d}{dr}f(r)=\partial_{r}f(r)$.
For $\alpha\in(0,1)$ we also write
\begin{equation}\label{dk}
  d_{k}:=
  \begin{pmatrix}
    0 & i\partial_{k+1-\alpha} \\
    i\partial_{\alpha-k} & 0 
  \end{pmatrix}=
  \begin{pmatrix}
    0 & i \partial_{r}+i\frac{k+1-\alpha}r \\
    i\partial_{r}+i\frac{\alpha-k} r& 0 
  \end{pmatrix}.
\end{equation}

Any $f\in L^{2}(\mathbb{R}^{2})$ can be expanded in spherical
harmonics as
\begin{equation*}
  f(x)=\sum_{k\in \mathbb{Z}}f_{k}(r)e^{ik \theta},
  \qquad
  \theta\in [0,2\pi),\quad r>0
\end{equation*}
where $f_{k}\in L^{2}(rdr)=L^{2}((0,+\infty),rdr)$.
Equivalently, we may decompose $L^{2}(\mathbb{R}^{2})$ as
\begin{equation}\label{eq:spschr}
  L^{2}(\mathbb{R}^{2})=
  \bigoplus_{k\in \mathbb{Z}}
  L^{2}(rdr)\otimes [e^{ik \theta}],
\end{equation}
where $[e^{ik \theta}]$ denotes the one dimensional span 
generated by $e^{ik \theta}$.
For the Dirac operator, we need
a similar decomposition of the spinor space
$\big[L^{2}(\mathbb{R}^{2})\big]^{2}$, which takes the form
\begin{equation}\label{eq:spDir}
  L^{2}(\mathbb{R}^{2})^{2}=
  \bigoplus_{k\in \mathbb{Z}}
  L^{2}(rdr)^{2}\otimes h_{k}(\mathbb{S}^{1}),
\end{equation}
where $h_{k}(\mathbb{S}^{1})$ is the one dimensional space
\begin{equation}\label{eq:hk}
  h_{k}=
  h_{k}(\mathbb{S}^{1})=
  \left[
    \begin{pmatrix}
      e^{ik \theta} \\
      e^{i(k+1)\theta} 
    \end{pmatrix}
  \right]=
  \left\{
    c\begin{pmatrix}
      e^{ik \theta} \\
      e^{i(k+1)\theta} 
    \end{pmatrix}:
    c\in \mathbb{C}
  \right\}.
\end{equation}
This means that we will expand spinors
$(\begin{smallmatrix} \phi \\ \psi \end{smallmatrix})
  \in \big[L^{2}(\mathbb{R}^{2})\big]^{2}$
in the form
\begin{equation*}
  \phi=\sum_{k\in \mathbb{Z}}\phi _{k}(r)e^{ik \theta},
  \qquad
  \psi=\sum_{k\in \mathbb{Z}}\psi _{k}(r)e^{i(k+1)\theta}.
\end{equation*}

\textsc{Decomposition of the operator}.
Using \eqref{eq:spDir} we may decompose $\mathcal{D}_{A}$ as
\begin{equation*}
  \mathcal{D}_{A}=
  \bigoplus_{k\in \mathbb{Z}}d_{k}\otimes I_{2}
  \qquad
  d_{k}:=
  \begin{pmatrix}
    0 & i\partial_{k+1-\alpha} \\
    i\partial_{\alpha-k} & 0 
  \end{pmatrix}, 
    \qquad
  I_2:=
  \begin{pmatrix}
    1 & 0 \\
  0 & 1 
  \end{pmatrix}.
\end{equation*}
Equivalently we can write
\begin{equation*}
  \mathcal{D}_{A}
  \sum_{k}
  \begin{pmatrix}
    e^{ik \theta}\ \phi_{k}(r) \\
    e^{i(k+1)\theta}\ \psi_{k}(r)
  \end{pmatrix}
  =
  \sum_{k}
  \begin{pmatrix}
    ie^{ik \theta}\ \partial_{k+1-\alpha}\psi_{k}(r) \\
    ie^{i(k+1)\theta}\ \partial_{\alpha-k}\phi_{k}(r)
  \end{pmatrix}.
\end{equation*}
We note the following facts:
\begin{enumerate}
  \item 
  $d_{k}$ with domain $D(d_{k})=\big[C_{c}^{\infty}((0,+\infty))\big]^{2}$
  is symmetric and densely defined;
  \item 
  Since the decomposition \eqref{eq:spDir} is orthogonal, we have
  for the adjoint operator and the closure
  \begin{equation*}
    \mathcal{D}_{A}^{*}=
    \bigoplus_{k\in \mathbb{Z}}d_{k}^{*}\otimes I_{2},
    \qquad
    \overline{\mathcal{D}}_{A}=
    \bigoplus_{k\in \mathbb{Z}}\overline{d}_{k}\otimes I_{2}
  \end{equation*}
  and similarly, to determine the selfadjoint extensions of
  $\mathcal{D}_{A}$, it is sufficient to determine the extensions
  of $d_{k}$ for each $k$ and combine them in all possible
  ways.
\end{enumerate}
Hence we are reduced to study the closure $\overline{d}_{k}$ and the adjoint $d^{*}_{k}$.\vspace{0.1cm}

\textsc{The closure $\overline{d}_{k}$}. If
$\phi\in C_{c}^{\infty}((0,+\infty))$, expanding the square and 
integrating by parts we have
\begin{equation*}
  \textstyle
  \int_{0}^{\infty}|\partial_{\sigma}\phi|^{2}rdr=
  \int_{0}^{\infty}
  (|\partial_{r}\phi|^{2}+\frac{\sigma^{2}}{r^{2}}|\phi|^{2})
  rdr.
\end{equation*}
For 
$f=(\begin{smallmatrix}\phi \\ \psi\end{smallmatrix})\in \big[C_{c}^{\infty}((0,+\infty))\big]^{2}$, this implies 
\begin{equation*}
  \textstyle
  \|d_{k}f\|_{L^{2}(rdr)}^{2}=
  \|f'\|_{L^{2}(rdr)}^{2}+
  \|\frac{k+1-\alpha}{r}\psi\|_{L^{2}(rdr)}^{2}+
  \|\frac{k-\alpha}{r}\phi\|_{L^{2}(rdr)}^{2}.
\end{equation*}
The latter norm controls $\|f\|_{L^{\infty}}$: indeed, we have
\begin{equation*}
  \textstyle
  |\phi(b)|^{2}-|\phi(a)|^{2}=
  \int_{a}^{b}\partial_{r}|\phi(r)|^{2}dr\le
  \int_{a}^{b}(|\phi'|^{2}+\frac{|\phi|^{2}}{r^{2}})rdr
\end{equation*}
and hence if
$\phi\in C_{c}^{\infty}((0,+\infty))$
\begin{equation*}
  \|\phi\|_{L^{\infty}}\le\|\phi'\|_{L^{2}(rdr)}+
  \|\phi/r\|_{L^{2}(rdr)}.
\end{equation*}
Now let $f_{n}\in D(d_{k})$ such that both $f_{n}$ and
$d_{k}f_{n}$ converge in $L^{2}(rdr)$; since 
$f_{n}\in \big[C_{c}^{\infty}\big]^2$ it follows that $f=\lim_{n\to\infty} f_{n}$
is in $\big[C([0,+\infty))\big]^2$ and $f(0)=0$.
This implies that the closure $\overline{d}_{k}$ of $d_{k}$
has domain
\begin{equation*}
  Y=
  D(\overline{d}_{k})=X^{2}=X \times X
\end{equation*}
where
\begin{equation}\label{eq:domain}
  X=
  \{\phi\in L^{2}(rdr):
  \phi',\phi/r\in L^{2}(rdr),\ 
  \phi\in C([0,+\infty)),\ 
  \phi(0)=0\}
\end{equation}
and of course for all 
$f=(\begin{smallmatrix}\phi \\ \psi\end{smallmatrix})
  \in D(\overline{d}_{k})$
\begin{equation}\label{closuredirac}
  \overline{d}_{k}f=
  \begin{pmatrix}
     i\partial_{k+1-\alpha}\psi  \\
     i\partial_{\alpha-k}\phi
  \end{pmatrix}.
\end{equation}
Note that functions in $Y$ are absolutely continuous, hence
the distributional derivative coincides with the classical
derivative a.e.\vspace{0.1cm}

\textsc{The adjoint $d^{*}_{k}$}.
By definition
$f=(\begin{smallmatrix}\phi \\ \psi\end{smallmatrix})
  \in D(d^{*}_{k})$
iff $\exists g\in L^{2}(rdr)^{2}$ such that
\begin{equation*}
  \langle f,d_{k}u\rangle=
  \langle g,u\rangle
  \qquad
  \forall u\in (C_{c}^{\infty})^{2}.
\end{equation*}
Clearly this is equivalent to
\begin{equation*}
  f\in L^{2}(rdr)^{2}
  \quad\text{and}\quad 
  d_{k}f\in L^{2}(rdr)^{2},
\end{equation*}
where $d_{k}f$ is computed in distribution sense. We conclude
\begin{equation*}
  D(d^{*}_{k})=
  \{f\in L^{2}(rdr)^{2}:d_{k}f\in L^{2}(rdr)^{2}\}
\end{equation*}
and the expression of $d^{*}_{k}$ is the same as $d_{k}$.\vspace{0.1cm}

\textsc{Deficiency indices of $\overline{d}_{k}$}.
Since $\overline{d}_{k}$ is closed, densely defined and symmetric,
von Neumann's theory applies.
We compute $\ker(d_{k}^{*}-i)$; we must solve the system
\begin{equation}\label{eq:syseq}
  i \partial_{\sigma+1}\psi=i \phi,
  \qquad
  i \partial_{-\sigma}\phi=i \psi,
  \qquad
  \sigma=k-\alpha
\end{equation}
with $\phi,\psi\in L^{2}(rdr)$. This implies
\begin{equation*}
  \textstyle
  (\partial_{r}^{2}+\frac 1r \partial_{r}
    -\frac{\sigma^{2}}{r^{2}})\phi=\phi,
  \qquad
  \phi\in L^{2}(rdr),
\end{equation*}
\begin{equation*}
  \textstyle
  (\partial_{r}^{2}+\frac 1r \partial_{r}
    -\frac{(\sigma+1)^{2}}{r^{2}})\psi=\psi,
  \qquad
  \psi\in L^{2}(rdr).
\end{equation*}
These are modified Bessel equations.
Consider the first one: 
a system of solutions
is given by a linear combination of modified Bessel functions $\{I_{\sigma},K_{\sigma}\}$.
From \cite[10.25.3; 10.30.4]{DLMF}, one knows that
$I_{\sigma}$ grows at infinity for all $\sigma=k-\alpha$ and is not in $L^{2}(rdr)$, 
so we drop it.
On the other hand, from \cite[10.25.3; 10.30.2]{DLMF}, $K_{\sigma}$ decays exponentially as $r\to+\infty$,
while near $r \sim 0$ we have
$K_{\sigma}(r)\sim r^{-|\sigma|}$ 
(indeed $K_{-\sigma}=K_{\sigma}$) so that
\begin{equation*}
  K_{\sigma}\in L^{2}(rdr)
  \qquad\iff\qquad
  |\sigma|<1.
\end{equation*}
Recall \eqref{notation-der} and also that 
\begin{equation*}
  \partial_{-\sigma}K_{\sigma}=-K_{\sigma+1}
\end{equation*}
which gives $\psi$.
We conclude that all solutions of \eqref{eq:syseq} must
be scalar multiples of
\begin{equation*}
  \begin{pmatrix}
    \phi\\
     \psi
  \end{pmatrix}=
  \begin{pmatrix}
     K_{\sigma}  \\
     -K_{\sigma+1}
  \end{pmatrix}.
\end{equation*}
This solution is in the domain of $d^{*}_{k}$ iff
$K_{\sigma}\in L^{2}(rdr)$ and
$\partial_{-\sigma}\phi=\psi=-K_{\sigma+1}\in L^{2}(rdr)$.
Thus we must have $K_{k-\alpha},K_{k+1-\alpha}\in L^{2}(rdr)$
which implies
\begin{equation*}
  |k-\alpha|<1,\quad |k+1-\alpha|<1
  \quad\Leftrightarrow\quad
  k=0.
\end{equation*}
Summing up,
\begin{itemize}
  \item if $k\neq0$ then 
  $\ker(d^{*}_{k}-i)=\{0\}$
  \item 
  $\ker(d^{*}_{0}-i)=
    [(\begin{smallmatrix} K_{\alpha}  \\ -K_{1-\alpha} 
    \end{smallmatrix})]$,
  with dimension $n_{+}=1$.
\end{itemize}
The analysis for $\ker(d^{*}_{k}+i)$ is very similar.
System \eqref{eq:syseq} is replaced by
\begin{equation}\label{eq:syseq2}
  i \partial_{\sigma+1}\psi=-i \phi,
  \qquad
  i \partial_{-\sigma}\phi=-i \psi,
  \qquad
  \sigma=k-\alpha
\end{equation}
which gives the same second order equations for $\phi,\psi$
but now the possible solutions are the multiples of
\begin{equation*}
  \begin{pmatrix}
    \phi\\
     \psi
  \end{pmatrix}=
  \begin{pmatrix}
     K_{\sigma}  \\
     K_{\sigma+1}
  \end{pmatrix}.
\end{equation*}
By similar computations we have
\begin{itemize}
  \item if $k\neq0$ then 
  $\ker(d^{*}_{k}+i)=\{0\}$
  \item 
  $\ker(d^{*}_{0}+i)=
    [(\begin{smallmatrix} K_{\alpha}  \\ K_{1-\alpha} 
    \end{smallmatrix})]$,
  with dimension $n_{-}=1$.
\end{itemize}
By the von Neumann theory we thus obtain:
\begin{enumerate}
  \item If $k\neq0$ then
  $\overline{d}_{k}=d_{k}^{*}$ is selfadjoint, with domain 
  $Y$ defined in \eqref{eq:domain}
  \item For $k=0$
  \begin{equation*}
    D(d^{*}_{0})=
    Y
    \oplus
    \left[g_{+} \right]
    \oplus
    \left[g_{-}\right]
    \quad\text{where}\quad 
    g_{+}=
    \begin{pmatrix} K_{\alpha}  \\ -K_{1-\alpha} \end{pmatrix},
    \quad
    g_{-}=
    \begin{pmatrix} K_{\alpha} \\ K_{1-\alpha} \end{pmatrix}
  \end{equation*}
  and for any $f\in D(\overline{d}_{0})$, 
  $c_{1},c_{2}\in \mathbb{C}$,
  \begin{equation*}
    d^{*}_{0}(f+c_{1}g_{+}+c_{2}g_{-})=
    \overline{d}_{0}f+ic_{1}g_{+}-ic_{2}g_{-}
  \end{equation*}
  \item 
  Selfadjoint extensions $d_{0}^{\gamma}$ of $d_{0}$ are in
  1--1 correspondence with
  unitary operators from $N_{+}$ onto $N_{-}$,
  i.e.~multiplications by $e^{2i \gamma}$
  with $\gamma$ real.  We have
  \begin{equation*}
    D(d^{\gamma}_{0})=
    \{f+c g_{+}+ce^{2i \gamma}g_{-}:f\in Y,\ 
      c\in \mathbb{C}\}=
      Y \oplus[g_{+}+e^{2 i \gamma}g_{-}]
  \end{equation*}
  and for all $f\in Y$, $c\in \mathbb{C}$
  \begin{equation*}
    d_{0}^{\gamma}(f+c g_{+}+ce^{2i \gamma}g_{-})=
    \overline{d}_{0}f+ic g_{+}-ice^{2i \gamma}g_{-}.
  \end{equation*}
  \item 
  We rewrite the last two formulas in an equivalent way:
  \begin{equation*}
    c g_{+}+ce^{2i \gamma}g_{-}=
    ce^{i \gamma}
    \begin{pmatrix} 
      e^{-i \gamma} K_{\alpha}+e^{i \gamma}K_{\alpha}
      \\ 
      -e^{-i \gamma}K_{1-\alpha} +e^{i \gamma}K_{1-\alpha}
    \end{pmatrix}=
    2ce^{i \gamma}
    \begin{pmatrix} 
      \cos \gamma\cdot K_{\alpha}
      \\ 
      i\sin \gamma\cdot K_{1-\alpha}
    \end{pmatrix},
  \end{equation*}
  \begin{equation*}
    ic g_{+}-ice^{2i \gamma}g_{-}=
    ic e^{i \gamma}
    \begin{pmatrix} 
      e^{-i \gamma} K_{\alpha}-e^{i \gamma}K_{\alpha}
      \\ 
      -e^{-i \gamma}K_{1-\alpha} -e^{i \gamma}K_{1-\alpha}
    \end{pmatrix}=
    2ce^{i \gamma}
    \begin{pmatrix} 
      \sin \gamma\cdot K_{\alpha}
      \\ 
      -i\cos \gamma\cdot K_{1-\alpha}
    \end{pmatrix}.
  \end{equation*}
  Thus we have
  \begin{equation*}
    D(d^{\gamma}_{0})=
    \{f+c 
    \left(
    \begin{smallmatrix}
      \cos \gamma \cdot K_{\alpha} \\
      i\sin \gamma \cdot K_{1-\alpha}
    \end{smallmatrix}
    \right)
    : f\in Y,\ c\in \mathbb{C}\}=
    Y \oplus
    \left[\left(
    \begin{smallmatrix}
      \cos \gamma \cdot K_{\alpha} \\
      i\sin \gamma \cdot K_{1-\alpha}
    \end{smallmatrix}
    \right)\right]
  \end{equation*}
  and
  \begin{equation}\label{d0gam}
    d_{0}^{\gamma}
    \left(
    f+c 
    \left(
    \begin{smallmatrix}
      \cos \gamma \cdot K_{\alpha} \\
      i\sin \gamma \cdot K_{1-\alpha}
    \end{smallmatrix}
    \right)
    \right)=
  \overline{d}_{0}f+c 
    \left(
    \begin{smallmatrix}
      \sin \gamma \cdot K_{\alpha} \\
      -i\cos \gamma \cdot K_{1-\alpha}
    \end{smallmatrix}
    \right).
  \end{equation}
  For uniformity we write also, for $k\neq0$,
  \begin{equation}\label{eq:dkgam}
    d_{k}^{\gamma}=\overline{d}_{k},
    \qquad
    D(d_{k}^{\gamma})=Y.
  \end{equation}
\end{enumerate}

We can now characterize all extensions of $\mathcal{D}_{A}$,
initially defined on $C_{c}^{\infty}(\mathbb{R}^{2}\setminus0)$.
Recall that $Y$ defined in \eqref{eq:domain} is the domain
of the closures $\overline{d}_{k}$.
For any $\gamma\in \mathbb{R}$ we set
\begin{equation*}
  Y^{\gamma}=  \left[\left(
    \begin{smallmatrix}
      \phi \\
      \psi
    \end{smallmatrix}
    \right)\right] \oplus
    \left[\left(
    \begin{smallmatrix}
      \cos \gamma \cdot K_{\alpha} \\
      i\sin \gamma \cdot K_{1-\alpha}
    \end{smallmatrix}
    \right)\right] ,\quad \left(
    \begin{smallmatrix}
      \phi \\
      \psi 
    \end{smallmatrix}
    \right)\in Y .
\end{equation*}
Note that $\partial_{\alpha}K_{\alpha}=-K_{1-\alpha}$,
$\partial_{1-\alpha}K_{1-\alpha}=-K_{\alpha}$. 
Thus the action of $d^{\gamma}_{0}$
can be written simply as follows:
\begin{equation*}
  d_{0}^{\gamma} 
  \begin{pmatrix}
     \phi_{0}  \\
     \psi_{0}  
  \end{pmatrix}
  =
  \begin{pmatrix}
    i \partial_{1-\alpha}\psi_{0}(r) \\
    i\partial_{\alpha}\phi_{0}(r)
  \end{pmatrix}
  \quad\text{for all}\quad 
  \begin{pmatrix}
    \phi_{0}\\\psi_{0}\end{pmatrix}
      \in Y^{\gamma}.
\end{equation*}
We have proved the following result:

\begin{theorem}[Selfadjoint extension for Dirac AB]\label{the:sadjextDAB}\label{ABsa}
  Any selfadjoint extension of $\mathcal{D}_{A}$ is of the
  form $\mathcal{D}_{A, \gamma}$, where $\gamma\in [0,2\pi)$,
  \begin{equation*}
    D(\mathcal{D}_{A,\gamma})=
    (Y^{\gamma}\otimes h_{0})
    \bigoplus_{k\neq0}(Y \otimes h_{k})
  \end{equation*}
  and
  \begin{equation*}
    \mathcal{D}_{A,\gamma}=
    (d_{0}^{\gamma}\otimes I_{2})
    \bigoplus_{k\neq0}(\overline{d}_{k}\otimes I_{2}).
  \end{equation*}
  Any $f\in D(\mathcal{D}_{A,\gamma})$ can be written as
  \begin{equation*}
    f=
    \sum_{k\in \mathbb{Z}}
    \begin{pmatrix}
      e^{ik \theta}\ \phi_{k}(r) \\
      e^{i(k+1)\theta}\ \psi_{k}(r)
    \end{pmatrix}
  \end{equation*}
  with 
  $(\begin{smallmatrix}\phi_{k}\\\psi_{k}\end{smallmatrix})\in Y$
  for $k\neq0$ while
  $(\begin{smallmatrix}\phi_{0}\\\psi_{0}\end{smallmatrix})
    \in Y^{\gamma}$
  for $k=0$, i.e., for some $c\in \mathbb{C}$ and
  $(\begin{smallmatrix}\tilde\phi_{0}\\\tilde\psi_{0}
    \end{smallmatrix}) \in Y$,
  \begin{equation*}
    \begin{pmatrix}
      \phi_{0} \\
      \psi_{0}
    \end{pmatrix}=
    c
    \begin{pmatrix}
      \cos \gamma \cdot K_{\alpha}(r) \\
      i\sin \gamma \cdot K_{1-\alpha}(r) 
    \end{pmatrix}
    +\begin{pmatrix}
       \tilde\phi_{0}(r) \\
       \tilde\psi_{0}(r)
    \end{pmatrix}.
  \end{equation*}
  The action of $\mathcal{D}_{A,\gamma}$ on such
  $f$ is given by
  \begin{equation}\label{DAf}
    \mathcal{D}_{A,\gamma}f=
    \sum_{k\in\Z}
    \begin{pmatrix}
      ie^{ik \theta}\ \partial_{k+1-\alpha}\psi_{k}(r) \\
      ie^{i(k+1)\theta}\ \partial_{\alpha-k}\phi_{k}(r)
    \end{pmatrix}.
    \end{equation}
\end{theorem}

\begin{remark}\rm
  The singular term, that is the part associated with the modified Bessel functions, does not vanish for any choice of the parameter $\gamma$ in the self-adjoint extension. This is in sharp contrast with what happens for the second order AB Hamiltonian: in that case the Friedrichs extension does not contain any singular function in the domain. We remark that this fact was overlooked in the previous papers \cite{CF, CYZ}.
\end{remark}

\section{Generalized eigenfunctions and relativistic Hankel transform}\label{sec:3}

The first step in order to study the dynamics of system \eqref{eq:dirac} will be to provide an explicit representation formula for the solution. 
We begin by introducing the definition of the {\em relativistic Hankel transform}.

\subsection{The generalized eigenfunctions} 
The generalized eigenfunctions of the self-adjoint operator $d_{k}^\gamma$ are the weak solutions of
\begin{equation}\label{eq:geneg}
d_{k}^\gamma \Psi_k(\rho r)=d_{k}^\gamma \left(\begin{matrix} \psi_k^1(\rho r)\\ \psi_k^2(\rho r)\end{matrix}\right)=\rho  \left(\begin{matrix} \psi_k^1(\rho r)\\ \psi_k^2(\rho r)\end{matrix}\right)
\end{equation}
for $\rho\in\mathbb{R}$. In the following we will assume $\rho>0$,
since the case of negative energies is similar.
Moreover, in view of the scaling properties of $d^{\gamma}_{k}$,
it is sufficient to compute the egienfunctions for $\rho=1$
and then rescale them

Recalling the definition of $d_{k}$, equation \eqref{eq:geneg} 
for $\rho=1$ reads
\begin{equation}\label{sys-bessel}
\begin{cases} i ( \frac{d}{dr} \psi^2_k(r)+\frac{k+1-\alpha}r \psi^2_k(r))= \psi^1_k(r)\\
i (\frac{d}{dr} \psi^1_k(r)-\frac{k-\alpha}r \psi^1_k(r))= \psi^2_k(r)
\end{cases}
\end{equation}
which is equivalent to the Bessel equation
\begin{equation}\label{sys-bessel'}
  \frac{d^2}{dr^2} \psi^1_k(r)+\frac1r  \frac{d}{dr} \psi^1_k(r)+(1-\frac{(k-\alpha)^2}{r^2})\psi^1_k(r)=0
\end{equation}
whose solutions are
\begin{equation}\label{correctsol}\psi^1_k(r)= J_{\pm(k-\alpha)}(r).
\end{equation}
Thus the solutions to \eqref{sys-bessel} are given by
(any scalar  multiple of)
\begin{equation}\label{Jsol}
  \begin{pmatrix}
 \psi^1_k(r)\\
  \psi^2_k(r)
  \end{pmatrix}=\begin{pmatrix}
   J_{k-\alpha}(r)\\
   - iJ_{k+1-\alpha}(r)
  \end{pmatrix} \quad \text{or}\,\quad \begin{pmatrix}
   J_{-(k-\alpha)}(r)\\
    iJ_{-(k+1-\alpha)}(r)
  \end{pmatrix}.
\end{equation}
Recalling the well known asymptotic behaviour of Bessel functions
as $x\to0$
$$
J_\nu(x) = \frac{x^\nu}{2^\nu \Gamma(1+\nu)}+O(x^{\nu+2}),
$$
square integrability at the origin fixes the sign above, as long as $(k-\alpha) \not\in(-1,0)$, that is for $k\neq0$: we thus take
\begin{equation}\label{gen-eig'}
\begin{cases}
\psi^1_k(r)= J_{|k-\alpha|}(r),\quad \psi^2_k(r)=-iJ_{|k+1-\alpha|}(r)\quad {\rm if}\, k\geq 1, \\
\psi^1_k(r)= J_{|k-\alpha|}(r),\quad \psi^2_k(r)=iJ_{|k+1-\alpha|}(r)\qquad {\rm if} \, k\leq -1.
\end{cases}
\end{equation}
In the remaining case $k=0$, both choices in \eqref{Jsol} lead to solutions that are square integrable near the origin, though one component is more singular than the other. Thus the choice of the generalized eigenfunction in the case $k=0$ will be dictated by
the choice of the selfadjoint extension,
according to Theorem \ref{ABsa}.
Of course, choosing an eigenfunction reduces to the choice of 
$c_{1},c_{2}$ in the linear combination
\begin{equation}\label{gen-eig''}
\psi^1_0(r)= c_1 J_{-\alpha}(r)+c_2J_{\alpha}(r),\quad \psi^2_0(r)=-c_1 iJ_{1-\alpha}(r)+c_2 iJ_{-(1-\alpha)}(r).
\end{equation}

Let us fix our choices.
The radial Hamiltonian $$d^{\gamma}_{0}=\left(\begin{matrix} 0& i (\frac{d}{dr}+\frac{1-\alpha}r) \\ i (\frac{d}{dr}+\frac{\alpha}r) & 0\end{matrix}\right) $$ is symmetric, that is, 
for arbitrary spinors $\varphi,\chi\in D(d^{\gamma}_{0})$
 \begin{align}\label{sym}
 \langle  d^{\gamma}_{0} \chi(r), 
 \varphi(r)\rangle_{L^2(rdr)}=\langle\chi(r), d^{\gamma}_{0}\varphi(r)\rangle_{L^2(rdr)}.
 \end{align}
In components $\varphi(r)=\left(
\begin{smallmatrix} \varphi_1(r) \\ \varphi_2(r) 
\end{smallmatrix}
\right)$ 
and 
$\chi(r)=\left(\begin{smallmatrix} \chi_1(r) \\ \chi_2(r)
\end{smallmatrix} \right)$
we have
\begin{align*}
 & \langle  d^{\gamma}_{0} \chi(r), \varphi(r)\rangle_{L^2(rdr)}=\int^\infty_0  (\overline{\varphi_1}(r), \overline{\varphi_2}(r)) \left(\begin{matrix} 0& i (\frac{d}{dr}+\frac{1-\alpha}r) \\ i (\frac{d}{dr}+\frac{\alpha}r) & 0\end{matrix}\right) \left(
    \begin{smallmatrix}
      \chi_1(r) \\
          \chi_2(r)
    \end{smallmatrix}
    \right)  r \,dr\\
   & =\int^\infty_0  \Big[\overline{ \left(\begin{matrix} 0& i (\frac{d}{dr}+\frac{1-\alpha}r) \\ i (\frac{d}{dr}+\frac{\alpha}r) & 0\end{matrix}\right)\left(
    \begin{smallmatrix}
      \varphi_1(r) \\
          \varphi_2(r)
    \end{smallmatrix}
    \right) }\Big]^T\left(
    \begin{smallmatrix}
      \chi_1(r) \\
          \chi_2(r)
    \end{smallmatrix}
    \right)  r \,dr+i\left(\overline{\varphi_1}\chi_2r+\overline{\varphi_2}\chi_1r\right)\Big|^\infty_0\\
 &=\langle\chi(r), d^{\gamma}_{0}\varphi(r)\rangle_{L^2(rdr)}+i\left(\overline{\varphi_1}\chi_2r+\overline{\varphi_2}\chi_1r\right)\Big|^\infty_0.
 \end{align*}
Thus \eqref{sym} will be satisfied if and only 
if the spinors satisfy 
 \begin{align*}
ir\left(\overline{\varphi_1}\chi_2+\overline{\varphi_2}\chi_1\right)\Big|^\infty_0=0
 \end{align*}
that is to say
\begin{equation}\label{bdy}
  \lim_{r\to 0}r\left(\overline{\varphi_1}\chi_2+\overline{\varphi_2}\chi_1\right)=
  \lim_{r\to \infty}r\left(\overline{\varphi_1}\chi_2+\overline{\varphi_2}\chi_1\right)=0.
\end{equation}
Now, in order to choose $c_1, c_2$ in \eqref{gen-eig''}, let us consider the spinors 
$\varphi(r)=\left(
\begin{smallmatrix}
  \cos \gamma \cdot K_{\alpha} \\
  i\sin \gamma \cdot K_{1-\alpha}
\end{smallmatrix}
\right)$
and 
$\chi(r)=\left(
\begin{smallmatrix}
  \psi^1_0(r) \\
     \psi^2_0(r)
\end{smallmatrix}
\right)$, 
which belong to $D(d^\gamma_0)$ as we know. 
Since $K_{\sigma}$ is exponentially decaying and $J_{\nu}$ decays as $r\to\infty$, the condition at infinity is
obviously satisfied.
On the other hand, since
$K_{\sigma}(r)\sim \frac12 \Gamma(\sigma)(\frac12r)^{-|\sigma|}$ as $r \to 0$,  
then \begin{align}
\chi(r)=\left(
    \begin{smallmatrix}
      \psi^1_0(r) \\
         \psi^2_0(r)
    \end{smallmatrix}
    \right)\sim\left(
    \begin{smallmatrix}
     \cos\gamma\, r^{-\alpha} \\
     i\sin\gamma \, r^{-(1-\alpha)}
    \end{smallmatrix}
    \right)\quad \text{as}\, r\to0.
\end{align}
Recalling \eqref{gen-eig''}, 
and using that $J_{\nu}(r)\sim r^{\nu}$ near $r=0$, 
the boundary condition at $r=0$ implies
$$c_1=\tilde{c_1}\cos\gamma,\quad c_2=\tilde{c_2}\sin\gamma$$
and we must choose the eigenfunctions
\begin{equation*}
\begin{split}
\psi^1_0(r)&= \tilde{c_1}\cos\gamma J_{-\alpha}(r)+\tilde{c_2}\sin\gamma J_{\alpha}(r),\\ \psi^2_0(r)&=-\tilde{c_1}\cos\gamma iJ_{1-\alpha}(r)+\tilde{c_2}\sin\gamma iJ_{-(1-\alpha)}(r).
\end{split}
\end{equation*}
 This explicitly connects the choice of the generalized eigenfunction for $k=0$ to the choice of the self-adjoint extensions of the operator, that is the choice of the parameter $\gamma$ in Theorem \ref{ABsa}. Thus in the case $k=0$ we may choose any
 \begin{equation}\label{geneige0}
   \Psi_{0}(r)=
   \begin{pmatrix} 
   \psi^1_0(r)\\ \psi^2_0(r)
   \end{pmatrix}
   =c
   \begin{pmatrix}
    \cos \gamma J_{-\alpha}(r)+\sin \gamma J_{\alpha}(r)  \\
    -i\cos \gamma J_{1-\alpha}(r)+i\sin \gamma J_{\alpha-1}(r)
   \end{pmatrix}.
\end{equation}

\begin{remark}\rm
  For negative energies the computation is completely
  analogous. We obtain eigenfunctions of the form
  \begin{equation}\label{negen}
  \left(\begin{matrix}\phi_k^1(\rho r)\\ \phi_k^2(\rho r)\end{matrix}\right)=
  \left(\begin{matrix}\psi_k^1(|\rho| r)\\ -\psi_k^2(|\rho| r)\end{matrix}\right),\qquad \rho<0.
  \end{equation}
\end{remark}

\subsection{The relativistic Hankel transform}\label{sec:relhan}

After the previous discussion, we are now ready to define
the \emph{relativistic Hankel transform}, which will be an essential
tool in the rest of the paper. This transform in the Dirac Aharonov-Bohm case was first 
introduced in \cite{CF}, and produces an explicit representation of the propagator via the generalized eigenfunctions. We propose here a version slightly different from the one in \cite{CF}, which allows for a much simpler representation.

The eigenfunctions are normalized as follows:
given $\alpha\in(0,1)$, we choose
\begin{itemize}
  \item 
  for $k\ge1$,
  \begin{equation}\label{eq:psika}
    \Psi_{k}(r)=
    \frac{1}{\sqrt{2}}
    \begin{pmatrix}
      J_{k-\alpha}(r)  \\
      -iJ_{k+1-\alpha}(r)  
    \end{pmatrix}
  \end{equation}
  \item 
  for $k\le-1$,
  \begin{equation}\label{eq:psikb}
    \Psi_{k}(r)=
    \frac{1}{\sqrt{2}}
    \begin{pmatrix}
      J_{\alpha-k}(r)  \\
      iJ_{\alpha-k-1}(r)  
    \end{pmatrix}
  \end{equation}
  \item 
  for $k=0$,
  \begin{equation}\label{eq:psik0}
    \Psi_{0}(r)=
    \frac{1}{\sqrt{2}}
    \begin{pmatrix}
     \cos \gamma J_{-\alpha}(r)+\sin \gamma J_{\alpha}(r)  \\
     -i\cos \gamma J_{1-\alpha}(r)+i\sin \gamma J_{\alpha-1}(r)
    \end{pmatrix}.
  \end{equation}
\end{itemize}
Recall that for negative energies 
we must change the sign of the second component, which is
the same as taking the conjugate of $\Psi_{k}$
(see \eqref{negen}).

Then we define the \emph{relativistic Hankel transform} as a
sequence of operators, $k\in \mathbb{Z}$
\begin{equation}\label{eq:Pkspa}
  \mathcal{P}_{k}:[L^{2}(\mathbb{R}^{+},rdr)]^{2}
    \to L^{2}(\mathbb{R},|\rho|d\rho)
\end{equation}
acting on spinors
$f(r)=\big(\begin{smallmatrix}
   f_1 (r) \\ f_2 (r)\end{smallmatrix}\big)$, $r>0$,
as
\begin{equation}\label{hankel}
  \mathcal{P}_{k}f(\rho)=
  \begin{cases}
    \int_{0}^{+\infty}\Psi_{k}(r \rho)^{T}\cdot f(r)rdr
     &\text{if $ \rho>0 $,}\\
     \int_{0}^{+\infty}\overline{\Psi_{k}}(r |\rho|)^{T}
       \cdot f(r)rdr
     &\text{if $ \rho<0 $.}
  \end{cases}
\end{equation}
Note that $\mathcal{P}_{k}$ takes spinors into scalar
functions. 
$\mathcal{P}_{k}$ can be written as a combination
of Hankel transforms $\mathcal{H}_{\nu}$ of suitable order
$\nu$, defined as
\begin{equation*}
  \mathcal{H}_{\nu}\phi(\rho)=
  \int_{0}^{+\infty}J_{\nu}(r \rho)\phi(r)rdr, \quad\rho>0.
\end{equation*}
Thus we can write for $k\ge1$, with $\nu=k-\alpha$
\begin{equation}\label{eq:hankmat1}
  \mathcal{P}_{k} f(\rho)=
  2^{-\frac 12}
  \begin{cases}
    \mathcal{H}_{\nu}f_{1}(\rho)-i\mathcal{H}_{\nu+1}f_{2}(\rho)
      &\text{if $ \rho>0 $,}\\
    \mathcal{H}_{\nu}f_{1}(|\rho|)+i\mathcal{H}_{\nu+1}f_{2}(|\rho|)
     &\text{if $ \rho<0 $}
  \end{cases}
\end{equation}
while for $k\le-1$, with $\nu=\alpha-k$, we have
\begin{equation}\label{eq:hankmat2}
  \mathcal{P}_{k} f(\rho)=
  2^{-\frac 12}
  \begin{cases}
    \mathcal{H}_{\nu}f_{1}(\rho)+i\mathcal{H}_{\nu-1}f_{2}(\rho)
      &\text{if $ \rho>0 $,}\\
    \mathcal{H}_{\nu}f_{1}(|\rho|)-i\mathcal{H}_{\nu-1}f_{2}(|\rho|)
     &\text{if $ \rho<0 $}
  \end{cases}
\end{equation}
and finally for $k=0$
\begin{equation}\label{eq:hankmat0}
  \mathcal{P}_{0}f(\rho)=
  2^{-\frac 12}
  \begin{cases}
    \cos \gamma
    (\mathcal{H}_{-\alpha}f_{1}
    -i \mathcal{H}_{1-\alpha}f_{2})
    +\sin \gamma
    (\mathcal{H}_{\alpha}f_{1}+
    i \mathcal{H}_{\alpha-1}f_{2})
     &\text{$ \rho>0 $,}\\
     \cos \gamma
     (\mathcal{H}_{-\alpha}f_{1}+i \mathcal{H}_{1-\alpha}f_{2})
     +\sin \gamma
     (\mathcal{H}_{\alpha}f_{1}-i \mathcal{H}_{\alpha-1}f_{2})
     &\text{$ \rho<0 $.}
  \end{cases}
\end{equation}

As it is well known, the Hankel
transform $\mathcal{H}_{\nu}$ is a unitary operator on 
$L^{2}(\mathbb{R}^{+},rdr)$ satisfying
\begin{equation}\label{eq:propHn}
  \mathcal{H}_{\nu}^{2}=I,
  \qquad
  \mathcal{H}_{\nu}\partial_{\nu+1}=
  r\mathcal{H}_{\nu+1},
  \qquad
  \mathcal{H}_{\nu+1}\partial_{-\nu}=
  -r \mathcal{H}_{\nu}.
\end{equation}
The last two properties hold 
on the domain of $\partial_{\nu}=\partial_{r}+\frac{\nu}{r}$
(see \eqref{eq:domain}), and $r$ denotes the multiplication by
the independent variable. Thus by \eqref{eq:hankmat1}, 
\eqref{eq:hankmat2} we see that $\mathcal{P}_k$ is 
a well defined operator on the spaces \eqref{eq:Pkspa} for any $k\in\mathbb{Z}$.

Let us now show that the operator $\mathcal{P}_{k}$ diagonalizes $d_{k}$.
For $k\ge \alpha-\frac 12$ 
we have, denoting with $\nu=k-\alpha$,
\begin{equation*}
  \mathcal{P}_{k}d_{k}
  \begin{pmatrix}
    \phi  \\
     \psi
  \end{pmatrix}=
  \mathcal{P}_{k}
  \begin{pmatrix}
    i \partial_{\nu+1}\psi  \\
    i \partial_{-\nu} \phi
  \end{pmatrix}=2^{-\frac 12}
  \begin{cases}
    \mathcal{H}_{\nu}(i\partial_{\nu+1}\psi)
    -i \mathcal{H}_{\nu+1}
    (i\partial_{-\nu}\phi) 
    &\text{$ \rho>0 $,}\\
    \mathcal{H}_{\nu}(i\partial_{\nu+1}\psi)
    +i  \mathcal{H}_{\nu+1}
    (i\partial_{-\nu}\phi) 
     &\text{$ \rho<0, $}
  \end{cases}
\end{equation*}
and recalling \eqref{eq:propHn}
\begin{equation*}
  \left.
  =2^{-\frac 12}
  \begin{cases}
    -\rho(\mathcal{H}_{\nu}\phi-i \mathcal{H}_{\nu+1}\psi)
    &\text{$ \rho>0$,}\\
    |\rho| (\mathcal{H}_{\nu}\phi+i \mathcal{H}_{\nu+1}\psi )
     &\text{$ \rho<0,$}
  \end{cases}
  \right\}=
  -\rho\mathcal{P}_{k}
  \begin{pmatrix}
    \phi \\
    \psi 
  \end{pmatrix}.
\end{equation*}
Taking closures, this implies for all $k\ge \alpha-\frac 12$ and
$(\begin{smallmatrix} \phi  \\ \psi \end{smallmatrix})\in Y$
\begin{equation*}
  \mathcal{P}_{k}\overline{d}_{k}
  \begin{pmatrix}
    \phi  \\
     \psi
  \end{pmatrix}=
  -\rho\mathcal{P}_{k}
  \begin{pmatrix}
    \phi \\
    \psi 
  \end{pmatrix}.
\end{equation*}
The computation for $k<\alpha-\frac 12$ is similar.
The case $k=0$ requires an additional argument
since the domain $Y^{\gamma}$ of $d_{0}^{\gamma}$
contains also the linear span of the vector
\begin{equation*}
  v_{\gamma}=
    \begin{pmatrix}
      \cos \gamma \cdot K_{\alpha} \\
      i\sin \gamma \cdot K_{1-\alpha}
    \end{pmatrix}
\end{equation*}
which is mapped to
\begin{equation*}
  d_{0}^{\gamma}v_{\gamma}=
  w_{\gamma}:=
    \begin{pmatrix}
      \sin \gamma \cdot K_{\alpha} \\
      -i\cos \gamma \cdot K_{1-\alpha}
    \end{pmatrix}.
\end{equation*}
We can compute $\mathcal{P}_{0}v_{\gamma}$, 
$\mathcal{P}_{0}w_{\gamma}$ explicitly,
using the identity
\begin{equation*}
  \mathcal{H}_{\nu}K_{\nu}(\rho)=
  \frac{\rho^{\nu}}{1+\rho^{2}},
  \qquad
  \nu>-1
\end{equation*}
(see Section 6.5.2 of \cite{GradshteynRyzhik15-y}).
This gives
\begin{equation*}
  \mathcal{P}_{0}v_{\gamma}(\rho)=
  \frac{2^{-\frac 12}}{1+\rho^{2}}
  \begin{cases}
    \cos \gamma
    (\rho^{-\alpha}\cos \gamma+\rho^{1-\alpha} \sin \gamma)
    +\sin \gamma
    (\rho^{\alpha}\cos \gamma -\rho^{\alpha-1}\sin \gamma )
    &\text{$ \rho>0 $,}\\
    \cos \gamma
    (|\rho|^{-\alpha}\cos \gamma-|\rho|^{1-\alpha} \sin \gamma)+
    \sin \gamma
    (|\rho|^{\alpha}\cos \gamma +|\rho|^{\alpha-1}\sin \gamma )
     &\text{$ \rho<0 $}
  \end{cases}
\end{equation*}
\begin{equation*}
  \mathcal{P}_{0}w_{\gamma}(\rho)=
  \frac{2^{-\frac 12}}{1+\rho^{2}}
  \begin{cases}
    \cos \gamma
    (\rho^{-\alpha}\sin \gamma-\rho^{1-\alpha}\cos \gamma)+
    \sin \gamma
    (\rho^{\alpha}\sin \gamma+\rho^{\alpha-1}\cos \gamma)
    &\text{$ \rho>0 $,}\\
    \cos \gamma
    (|\rho|^{-\alpha}\sin \gamma+|\rho|^{1-\alpha}\cos \gamma )+
    \sin \gamma
    (|\rho|^{\alpha}\sin \gamma-|\rho|^{\alpha-1}\cos \gamma)
     &\text{$ \rho<0 $}
  \end{cases}
\end{equation*}
and summing up
\begin{equation*}
  \mathcal{P}_{0}d_{0}^{\gamma}v_{\gamma}(\rho)+
  \rho\mathcal{P}_{0}v_{\gamma}(\rho)=
  2^{-\frac 12}\sin \gamma\cos \gamma
  \begin{cases}
    (\rho^{-\alpha}+\rho^{\alpha-1}) &\text{if $ \rho>0 $,}\\
    |\rho|^{-\alpha}-|\rho|^{\alpha-1} &\text{if $ \rho<0 $.}
  \end{cases}
\end{equation*}
We conclude that
\begin{equation*}
    \mathcal{P}_{0}d_{0}^{\gamma}=
  -\rho\mathcal{P}_{0}
  \qquad\text{holds on $Y^{\gamma}$ iff}\qquad 
  \sin \gamma\cos \gamma=0.
\end{equation*}
From now on we assume that $\sin \gamma\cos \gamma=0$.
Note that if $\sin \gamma=0$ then $\mathcal{P}_{0}$
takes the form \eqref{eq:hankmat1} with
$\nu=-\alpha$, while if $\cos \gamma=0$ then
$\mathcal{P}_{0}$ has the form \eqref{eq:hankmat2}
with $\nu=\alpha$.

Finally, we prove that $\mathcal{P}_{k}$ is invertible.
Given $\phi\in L^{2}(\mathbb{R},|\rho|d\rho)$,
write $\phi_{+}(\rho)=\phi(\rho)$ and $\phi_{-}(\rho)=\phi(-\rho)$ 
for $\rho>0$. Then we can write
\begin{equation}\label{eq:Pm1a}
  \mathcal{P}_{k}^{-1}\phi =
  2^{-\frac 12}
  \begin{pmatrix}
     \mathcal{H}_{\nu}&
       \mathcal{H}_{\nu}  \\
     i\mathcal{H}_{\nu+1}&
       -i\mathcal{H}_{\nu+1}
  \end{pmatrix}
  \begin{pmatrix}
    \phi_{+} \\
     \phi_{-} 
  \end{pmatrix},
  \qquad
  \nu=k-\alpha\ge-\frac 12
\end{equation}
\begin{equation}\label{eq:Pm1b}
  \mathcal{P}_{k}^{-1}\phi =
  2^{-\frac 12}
  \begin{pmatrix}
     \mathcal{H}_{\nu}&
       \mathcal{H}_{\nu}  \\
     -i\mathcal{H}_{\nu-1}&
       i\mathcal{H}_{\nu-1}
  \end{pmatrix}
  \begin{pmatrix}
    \phi_{+} \\
     \phi_{-} 
  \end{pmatrix},
  \qquad
  \nu=\alpha-k>-\frac 12.
\end{equation}
and using $\mathcal{H}_{\nu}^{2}=I$ one verifies 
that this operator actually inverts $\mathcal{P}_{k}$.
It is also clear that 
$\mathcal{P}_{k}^{-1}=\mathcal{P}_{k}^{*}$
so that $\mathcal{P}_{k}$ is unitary.

In the case $k=0$ we restricted $\gamma$ 
so that $\sin \gamma\cos \gamma=0$.
By the same computations as above, one checks that
if $\sin \gamma=0$ then $\mathcal{P}_{0}^{-1}$
is given by formula \eqref{eq:Pm1a} with $\nu=-\alpha$,
while if $\cos \gamma=0$ then $\mathcal{P}_{0}^{-1}$
is given by formula \eqref{eq:Pm1b} with $\nu=\alpha$ (we use the well known fact that the Hankel inversion formula holds as long as $\nu \geq-1/2$).

Summing up, we have proved:

\begin{proposition}\label{pro:hankel}
  Let $\alpha\in(0,1)$ and $\gamma\in[0,2\pi)$ such that
  \begin{equation}\label{eq:sincos}
    \sin \gamma\cos \gamma=0.
  \end{equation}
  Consider the selfadjoint operators $d^{\gamma}_{k}$ 
  defined in \eqref{d0gam}, \eqref{eq:dkgam},
  with domains $Y$ for $k\neq0$ and $Y^{\gamma}$ for $k=0$.
  Define the relativistic Hankel transform
  $\mathcal{P}_{k}$ as in \eqref{hankel},
  with $\Psi_{k}$ given by \eqref{eq:psika}, \eqref{eq:psikb},
  \eqref{eq:psik0}.

  Then $\mathcal{P}_{k}$ can be written in the equivalent forms
  \eqref{eq:hankmat1}, \eqref{eq:hankmat2},
  \eqref{eq:hankmat0}. Moreover,
  \begin{enumerate}
    \item 
    $\mathcal{P}_{k}:
      L^{2}((0,+\infty),rdr)^{2}\to
      L^{2}(\mathbb{R},|\rho|d\rho)$ 
    is a bounded bijection.

    \item 
    $\mathcal{P}_{k}$ is unitary.
    $\mathcal{P}_{k}^{-1}=\mathcal{P}^{*}_{k}$
    is given by \eqref{eq:Pm1a} for $k\ge1$ or
    $k=0$ and $\sin \gamma=0$, and by
    \eqref{eq:Pm1b} if $k\le-1$ or $k=0$ and
    $\cos \gamma=0$.

    \item 
    We have
    $\mathcal{P}_{k}d_{k}^{\gamma}=-\rho\mathcal{P}_{k}$
    on $Y^\gamma$ for all $k\in \mathbb{Z}$.
  \end{enumerate}
\end{proposition}

In the following, we shall select $\gamma$ according to
\eqref{eq:convention}. In particular, 
condition \eqref{eq:sincos} will be always satisfied.

\section{Construction of the Dirac propagator}\label{sec:prop}
We provide next an explicit representation for the solution of the Dirac equation
\begin{equation}\label{eq:Dirac}
\begin{cases}
i\partial_t u+\mathcal{D}_{A,\gamma} u=0,\\
u|_{t=0}=f(x)=(\phi, \psi)^T \in D(\mathcal{D}_{A, \gamma})\subset (L^2(\R^2))^2.
\end{cases}
\end{equation}
As already discussed, we can write any $f$ in
the domain of $\mathcal{D}_{A,\gamma} $ as follows
\begin{equation}\label{eq:decf}
  f=
  \sum_{k\in \mathbb{Z}}
  \begin{pmatrix}
    e^{ik \theta}\ \phi_{k}(r) \\
    e^{i(k+1)\theta}\ \psi_{k}(r)
  \end{pmatrix}
\end{equation}
with
\begin{equation*}
  \begin{pmatrix}
    \phi_{0}(r) \\
      e^{i\theta}\ \psi_{0}(r)
    \end{pmatrix}
  =
  c_0
  \begin{pmatrix}
    \cos \gamma \cdot K_{\alpha}(r) \\
    i\sin \gamma \cdot K_{1-\alpha}(r)   e^{i\theta}
  \end{pmatrix}
  + \begin{pmatrix}
  \tilde\phi_{0}(r) \\
    e^{i\theta}\ \tilde\psi_{0}(r)
  \end{pmatrix}
\end{equation*}
where
$\begin{pmatrix}
  \tilde\phi_{0}(r) \\
    \tilde\psi_{0}(r)
\end{pmatrix}\in Y$ 
and
$\begin{pmatrix}
  \phi_{k}(r) \\
    \psi_{k}(r)
\end{pmatrix}\in Y$ 
for $k\neq0$ are regular at the origin.

We shall need to treat separately
the case $k\neq0$, which we call
the {\em non radial part}, and can be handled by relatively
standard tools. 
To separate the nonradial part we apply
the orthogonal projector $P_{\bot}$  defined in \eqref{def-pro},
which cuts away the spherical harmonic spinors of order 0:
\begin{equation*}
  f_{\bot}=P_{\bot}f=  \sum_{k\in \mathbb{Z}\setminus\{0\}}
  \begin{pmatrix}
    e^{ik \theta}\ \phi_{k}(r) \\
    e^{i(k+1)\theta}\ \psi_{k}(r)
  \end{pmatrix}
\end{equation*} 
Later on we shall consider the case
$k=0$ (the {\em radial part}) which is more 
delicate due to the singularity at 0.
This corresponds to the choice of data
\begin{equation*}
  f_0=P_{0}f=
  \begin{pmatrix}
  \phi_{0}(r) \\
    e^{i\theta}\ \psi_{0}(r)
  \end{pmatrix}=
  c_0
  \begin{pmatrix}
    \cos \gamma \cdot K_{\alpha}(r) \\
    i\sin \gamma \cdot K_{1-\alpha}(r)   e^{i\theta}
  \end{pmatrix}
  +\begin{pmatrix}
  \tilde\phi_{0}(r) \\
    e^{i\theta}\ \tilde\psi_{0}(r)
  \end{pmatrix}.
\end{equation*}
By orthogonality, the full flow is the sum
of the nonradial and the radial components:
 \begin{equation*}
e^{it\mathcal{D}_{A,\gamma}} f=e^{it\mathcal{D}_{A,\gamma}}P_0 f+e^{it\mathcal{D}_{A,\gamma}}P_{\bot} f=e^{it\mathcal{D}_{A,\gamma}} f_0+e^{it\mathcal{D}_{A,\gamma}} f_{\bot}.
  \end{equation*}

We start by providing an explicit representation of the propagator $e^{it\mathcal{D}_{A,\gamma}}$ via the generalized eigenfunctions.
Note that, as
discussed in the previous section, 
the presence of a Bessel function of negative order in the term
$k=0$ introduces some additional technical difficulties,
as we shall see in the next section.
In the following statement we denote polar coordinates on
$\mathbb{R}^{2}$ by $(r,\theta)$, like in
$x=r(\cos\theta,\sin\theta)$ and 
$y=r_2(\cos\theta_2,\sin\theta_2)$.

\begin{proposition}\label{prop:Dp} 
Let $\alpha\in(0, 1)$ and $\gamma$ as in \eqref{eq:convention}.
Let $\Phi_{k}(\theta)$ be the harmonics
\begin{equation}\label{eq:defPhik}
  \Phi_k(\theta)=\frac{1}{\sqrt{2\pi}}
  \left(\begin{matrix} e^{ik\theta}&0\\ 
  0&e^{i(k+1)\theta}\end{matrix}\right).
\end{equation}
Then the solution $u(t,x)$ of
\eqref{eq:Dirac} can be represented as 
\begin{equation}\label{D-propa}
u(t, x)=e^{it\mathcal{D}_{A,\gamma}}f=\int_{0}^\infty \int_{0}^{2\pi}{\bf K}(t,r,\theta,r_2,\theta_2) f(r_2,\theta_2)  d\theta_2 \;r_2dr_2,
\end{equation}
\begin{equation}\label{D-kernel}
{\bf K}(t,r,\theta,r_2,\theta_2) =\sum_{k\in\Z} 
\Phi_k(\theta) {\bf K}_{k}(t,r,r_2)\overline{\Phi_k(\theta_2)},
\end{equation}
where for $k\neq0$
\begin{equation}\label{kernel'}
{\bf K}_{k}(t, r,r_2)=
\begin{pmatrix}
  m_{|k-\alpha|} & 0 \\
  0 &  m_{|k-\alpha+1|}
\end{pmatrix}
\end{equation}
while
\begin{equation}\label{kernel'0}
  {\bf K}_{0}(t, r,r_2)=
  \begin{pmatrix}
    m_{-\alpha} & 0 \\
    0 &  m_{1-\alpha}
  \end{pmatrix}
  \quad\text{if $\alpha\le\frac12$,}\qquad 
  \begin{pmatrix}
    m_{\alpha} & 0 \\
    0 &  m_{\alpha-1}
  \end{pmatrix}
  \quad\text{if $\alpha>\frac12$}\qquad 
\end{equation}
and
\begin{equation*}
  m_{\nu}(t,r,r_{2})=
  \int_{0}^{\infty}e^{-it \rho}
  J_{\nu}(r \rho)J_{\nu}(r_{2}\rho)\rho d \rho.
\end{equation*}
\end{proposition}


\begin{proof}
According to decomposition \eqref{eq:spDir} we write the initial condition as
$$f(x)=(\phi,\psi)^{T}=\sum_{k\in\Z}\Phi_k(\theta) f_k(r),$$
where $f_k(r)=(\phi_k(r), \psi_k(r))^T$ and 
\begin{equation}\label{eq:fk}
 \phi_k(r)=\int_{0}^{2\pi} \phi(x) e^{ik\theta}\, d\theta, \quad \psi_k(r)=\int_{0}^{2\pi} \psi(x) e^{i(k+1)\theta}\, d\theta.
\end{equation}
Then $u(t,x)$ takes the form
\begin{equation}
u(t, x)=e^{it\mathcal{D}_{A,\gamma}}f=\sum_{k\in\Z}\Phi_k(\theta) u_k(t,r),
\end{equation}
where, for each $k\in\Z$,
 $u_k(t, r)$ satisfies 
\begin{equation}\label{eq:uka}
\begin{cases}
i\partial_t u_k(t, r)+d^{\gamma}_{k} u_k(t,r)=0, \\
u_k|_{t=0}=f_k(r)=(\phi_{k}(r),\psi_{k}(r))^T \in  D(d^{\gamma}_{k})
\subset [L^2((0,\infty);rdr)]^2.
\end{cases}
\end{equation}
We recall that  the operator $d^{\gamma}_{k}$ is given by
$$
d^{\gamma}_{k}=
\begin{cases}
d^{\gamma}_{0} \quad {\rm if}\: k=0,\\
\overline{d}_{k}\quad {\rm if}\: k\neq 0
\end{cases}
$$
with $\gamma$ as in the statement.
It is straightforward to represent the solution of 
\eqref{eq:uka} using the relativistic Hankel transform
studied in Proposition \ref{pro:hankel}. If we write
\begin{equation*}
  v_{k}(t,\rho)=\mathcal{P}_{k}u_{k},
\end{equation*}
by the diagonalization property (3)
the function $v_{k}$ satisfies
\begin{equation*}
  \begin{cases}
    i \partial_{t}v_{k}(t,\rho)-\rho v_{k}(t,\rho)=0\\
    v_{k}\vert_{t=0}=\mathcal{P}_{k}f_{k}(\rho)
  \end{cases}
\end{equation*}
which gives immediately
\begin{equation*}
  v_{k}(t,\rho)=e^{-it\rho}\mathcal{P}_{k}f_{k}(\rho)
\end{equation*}
and
\begin{equation*}
  u_{k}(t,x)=\mathcal{P}_{k}^{-1}
  e^{-it\rho}\mathcal{P}_{k}f_{k}(\rho).
\end{equation*}
Let us write explicitly the kernel $\mathbf{K}_{k}(t,r,r')$
of the operator 
$\mathcal{P}_{k}^{-1}e^{-it\rho}\mathcal{P}_{k}$.
Recalling the expressions \eqref{eq:hankmat1} and
\eqref{eq:Pm1a} for  $k\ge1$, or \eqref{eq:hankmat0} for $k=0$,
$\alpha\le \frac 12$ we obtain
\begin{equation*}
  \mathcal{P}_{k}^{-1}e^{-it \rho}\mathcal{P}_{k}=
  \begin{pmatrix}
    \mathcal{H}_{\nu}e^{-it |\rho|}\mathcal{H}_{\nu} & 0 \\
    0 &  \mathcal{H}_{\nu+1}e^{-it |\rho|}\mathcal{H}_{\nu+1}
  \end{pmatrix},
  \qquad
  \nu=
  \begin{cases}
    k-\alpha &\text{if $ k\ge1 $,}\\
    -\alpha &\text{if $ k=0 $.}
  \end{cases}
\end{equation*}

On the other hand, for $k\le-1$
or $k=0$, $\alpha>\frac 12$, using \eqref{eq:hankmat2}, \eqref{eq:Pm1b} and \eqref{eq:hankmat0}, we obtain
\begin{equation*}
  \mathcal{P}_{k}^{-1}e^{-it \rho}\mathcal{P}_{k}=
  \begin{pmatrix}
    \mathcal{H}_{\nu}e^{-it |\rho|}\mathcal{H}_{\nu} & 0 \\
    0 &  \mathcal{H}_{\nu-1}e^{-it |\rho|}\mathcal{H}_{\nu-1}
  \end{pmatrix},
  \qquad
  \nu=
  \begin{cases}
    \alpha-k &\text{if $ k\le-1 $,}\\
    \alpha &\text{if $ k=0 $.}
  \end{cases}
\end{equation*}
Introducing the scalar operator
$\mathcal{M}_{\nu}$
\begin{equation*}
  \mathcal{M}_{\nu}\phi(r)=
  \mathcal{H}_{\nu}e^{-it |\rho|}\mathcal{H}_{\nu}\phi(r)
\end{equation*}
we can write for $k\neq0$
\begin{equation*}
  \mathcal{P}_{k}^{-1}e^{-it \rho}\mathcal{P}_{k}=
  \begin{pmatrix}
    \mathcal{M}_{|k-\alpha|} & 0 \\
    0 &  \mathcal{M}_{|k-\alpha+1|}
  \end{pmatrix}
\end{equation*}
while for $k=0$ we have
\begin{equation*}
  \mathcal{P}_{0}^{-1}e^{-it \rho}\mathcal{P}_{0}=
  \begin{pmatrix}
    \mathcal{M}_{-\alpha} & 0 \\
    0 &  \mathcal{M}_{1-\alpha}
  \end{pmatrix}
  \quad\text{when $\alpha\le\frac12$,}\qquad 
  =
  \begin{pmatrix}
    \mathcal{M}_{\alpha} & 0 \\
    0 &  \mathcal{M}_{\alpha-1}
  \end{pmatrix}
  \quad\text{when $\alpha>\frac12$.}\qquad 
\end{equation*}
In integral form, $\mathcal{M}_{\nu}$ can be written as
\begin{equation*}
  \mathcal{M}_{\nu}\phi(r)=
  \int_{0}^{\infty}
  m_{\nu}(t,r,r_{2})\phi(r_{2})r_{2}dr_{2}
\end{equation*}
with kernel
\begin{equation*}
  m_{\nu}(t,r,r_{2})=
  \int_{0}^{\infty}e^{-it \rho}
  J_{\nu}(r \rho)J_{\nu}(r_{2}\rho)\rho d \rho,
\end{equation*}
and this concludes the proof.
\end{proof}

\section{Localized dispersive estimates: proof of Theorem 
\ref{thm:disper}}\label{sec:proofdisp}

In \eqref{D-propa} the term $k=0$ is singular, since it 
contains functions unbounded near 0.
To overcome this problem we shall
modify the previous construction at $k=0$ to put ourselves in position to exploit some well-established results.

We first recall the \emph{AB Schr\"{o}dinger operator}
on $L^{2}(\mathbb{R}^{2})$
\begin{equation*}
  H=-(\nabla+iA(x))^{2},
  \qquad
  A=(A_{1},A_{2})=\frac{\alpha}{|x|^{2}}(-x_{2},x_{1}),
  \qquad
  \alpha\in(0,1)
\end{equation*}
initially defined on 
$C_{c}^{\infty}(\mathbb{R}^{2}\setminus\{0\})$,
which admits a family of selfadjoint extensions
indicized by unitary operators 
$U:\mathbb{C}^{2}\to \mathbb{C}^{2}$
(see \cite{AT}). All the extensions have singular
functions in their domain but
a \emph{distinguished} one, the Friedrichs extension,
which coincides with the closure of $H$. 
We denote this extension again by
$H$; it is non negative, with natural domain
\begin{equation*}
  D(H)=\{u\in L^{2}:(\nabla+iA(x))^{2}u\in L^{2}\}.
\end{equation*}
In the usual harmonic decomposition
$L^{2}(\mathbb{R}^{2})=
  \bigoplus_{k\in \mathbb{Z}}L^{2}(rdr)\otimes[e^{ik \theta}]$
we have 
\begin{equation*}
    H=\bigoplus_{k\in \mathbb{Z}}H_{k-\alpha}\otimes I_2
\end{equation*}
where $H_{\nu}$ is the modified Bessel operator
\begin{equation*}
  H_{\nu}=H_{-\nu}=
  -\frac{d^{2}}{dr^{2}}-\frac{d}{dr}+\frac{\nu^{2}}{r^{2}}.
\end{equation*}
If $\nu\neq0$, this is a nonnegative selfadjoint
operator on $L^{2}(\mathbb{R}^{+};rdr)$ with domain
\begin{equation*}
  Z=\{f\in L^{2}(rdr)^{2}:
  f',f'',f'/r,f/r^{2}\in L^{2}(rdr)^{2},\ 
  f\in C^{1}([0,+\infty)),\ 
  f(0)=f'(0)=0
  \}.
\end{equation*}
The proof is similar to the proof of
$D(\overline{d}_{k})=Y$ in Section \ref{sec:pre}.
We have
\begin{equation}\label{eq:ZtoX}
  Z^{2}=\{f\in Y:\overline{d}_{k}f\in Y\}.
\end{equation}
By the properties of the Hankel transform $\mathcal{H}_{\nu}$
we can also write
\begin{equation}\label{eq:hanHk}
  H_{\nu}=\mathcal{H}_{\nu}\rho^{2}\mathcal{H}_{\nu}.
\end{equation}
We notice that 
\begin{equation*}
  (\overline{d}_{k})^{2}=
  \begin{pmatrix}
    0 & i\partial_{k+1-\alpha} \\
    i\partial_{\alpha-k} & 0 
  \end{pmatrix}^{2}=
  \begin{pmatrix}
    -\partial_{k+1-\alpha} \partial_{\alpha-k}& 0 \\
    0 &  - \partial_{\alpha-k}\partial_{k+1-\alpha}
  \end{pmatrix}=
  \begin{pmatrix}
    H_{k-\alpha} & 0 \\
    0 &  H_{k-\alpha+1}
  \end{pmatrix},
\end{equation*}
and by a direct computation we have also
\begin{equation*}
  (d_{0}^{\gamma})^{2}=
  (\overline{d}_{0})^{2}=
  \begin{pmatrix}
    H_{\alpha} & 0 \\
    0 &  H_{1-\alpha}
  \end{pmatrix}
\end{equation*}
independently of the value of $\gamma$.
We call these operators $L_{k-\alpha}$:
\begin{equation*}
  L_{k-\alpha}=
  \begin{pmatrix}
    H_{k-\alpha} & 0 \\
    0 &  H_{k-\alpha+1}
  \end{pmatrix},
  \qquad
  k\in \mathbb{Z}.
\end{equation*}
Obviously, $L_{k-\alpha}$ is selfadjoint with domain $Z^{2}$,
non negative, and satisfies
\begin{equation*}
  L_{k}=(\overline{d}_{0})^{2},
  \qquad
  L_{0}=
  (d_{0}^{\gamma})^{2}=
  (\overline{d}_{0})^{2}.
\end{equation*}
If we collect the operators $L_{k-\alpha}$ using
the decomposition \eqref{eq:spDir} of 
$[L^{2}(\mathbb{R}^{2})]^{2}$, we get
\begin{equation}\label{eq:defL}
  \bigoplus_{k}L_{k-\alpha}\otimes I_{2}=
  \begin{pmatrix}
    H & 0 \\
    0 & H 
  \end{pmatrix}
  =:L.
\end{equation}
The operator $L= I_{2}\ H$ clearly has properties
similar to $H$.

Now, in order to avoid the singular term $k=0$ in 
\eqref{D-propa} (we plan to deal with it separately), we consider a modified operator
$\widetilde{\mathcal{P}}_{0}:
  L^{2}((0,+\infty),rdr)^{2}\to
  L^{2}(\mathbb{R},|\rho|d\rho)$ 
for $k=0$ defined as
\begin{equation*}
  \widetilde{\mathcal{P}}_{0}=
  \begin{pmatrix}
    \mathcal{H}_{\alpha} & -i\mathcal{H}_{1-\alpha} \\
    \mathcal{H}_{\alpha} & i\mathcal{H}_{1-\alpha} 
  \end{pmatrix}
  \quad\text{with inverse}\quad 
  \widetilde{\mathcal{P}}_{0}^{-1}=
  \widetilde{\mathcal{P}}_{0}^{*}=
  \begin{pmatrix}
    \mathcal{H}_{\alpha} & \mathcal{H}_{\alpha} \\
    i\mathcal{H}_{1-\alpha} & -i\mathcal{H}_{1-\alpha} 
  \end{pmatrix}
\end{equation*}
(where we split as before $f\in L^{2}(|\rho|d\rho)$ into
$f=(\begin{smallmatrix} f_{+} \\ f_{-} \end{smallmatrix})$
with $f_{\pm}(\rho)=f(\pm \rho)$). For all the other
values of $\mathbb{Z}$ we do not change $\mathcal{P}_{k}$:
\begin{equation*}
  \widetilde{\mathcal{P}}_{k}=\mathcal{P}_{k}
  \quad\text{for}\quad k\in \mathbb{Z}\setminus\{0\}
\end{equation*}
so that we can write simply for all $k\in \mathbb{Z}$
\begin{equation*}
  \widetilde{\mathcal{P}}_{k}=
  \begin{pmatrix}
    \mathcal{H}_{|k-\alpha|} & -i\mathcal{H}_{|k-\alpha+1|} \\
    \mathcal{H}_{|k-\alpha|} & i\mathcal{H}_{|k-\alpha+1|} 
  \end{pmatrix}.
\end{equation*}
We see that for all $k\in \mathbb{Z}$ we have
\begin{equation*}
  \widetilde{\mathcal{P}}_{k}^{-1}
  \rho^{2}
  \widetilde{\mathcal{P}}_{k}=
  \begin{pmatrix}
    \mathcal{H}_{|k-\alpha|}\rho^{2}\mathcal{H}_{|k-\alpha|} 
    & 0 \\
    0 &  
    \mathcal{H}_{|k-\alpha+1|}\rho^{2}\mathcal{H}_{|k-\alpha+1|} 
  \end{pmatrix}=
  \begin{pmatrix}
    H_{k-\alpha} & 0 \\
    0 & H_{{k-\alpha+1}}
  \end{pmatrix}
\end{equation*}
that is to say
\begin{equation}\label{eq:Lka}
  L_{k-\alpha}=  \widetilde{\mathcal{P}}_{k}^{-1}
  \rho^{2}
  \widetilde{\mathcal{P}}_{k}.
\end{equation}

We shall use the propagator $e^{-it\sqrt{L}}$,
which is essentially $e^{-it\sqrt{H}}$.
It admits the usual decomposition
\begin{equation*}
  e^{-it\sqrt{L}}=
  \bigoplus_{k\in \mathbb{Z}}
  e^{-it \sqrt{L_{k-\alpha}}}
  \otimes I_{2},
  \qquad
  e^{-it \sqrt{L_{k-\alpha}}}=
  \widetilde{\mathcal{P}}_{k}^{-1} e^{-it|\rho|}
  \widetilde{\mathcal{P}}_{k}.
\end{equation*}
We write the propagator in integral form
\begin{equation}\label{D-propa'}
  e^{-it \sqrt{L}}f=
  \int_{0}^\infty \int_{0}^{2\pi}{\tilde{\bf K}}(t,r,\theta,r_2,\theta_2) f(r_2,\theta_2)  d\theta_2 \;r_2dr_2.
\end{equation}
From the previous discussion, we see that
$\tilde{\bf K}(t,r,\theta,r_2,\theta_2)$ 
can be written 
\begin{equation*}
  \tilde
{\bf K}(t,r,\theta,r_2,\theta_2) =\sum_{k\in\Z} \Phi_k(\theta) 
\tilde{\bf K}_{k}(t,r,r_2)\overline{\Phi_k(\theta_2)},
\qquad
\Phi_k(\theta)=\frac{1}{\sqrt{2\pi}}
\left(\begin{matrix} e^{ik\theta}&0\\ 
0&e^{i(k+1)\theta}\end{matrix}\right)
\end{equation*}
where now \emph{for all $k\in \mathbb{Z}$}
\begin{equation*}
  \tilde{\bf K}_{k}(t, r,r_2)=
  \begin{pmatrix}
    m_{|k-\alpha|} & 0 \\
    0 &  m_{|k-\alpha+1|}
  \end{pmatrix}
\end{equation*}
and as before
\begin{equation*}
  m_{\nu}(t,r,r_{2})=
  \int_{0}^{\infty}e^{-it \rho}
  J_{\nu}(r \rho)J_{\nu}(r_{2}\rho)\rho d \rho.
\end{equation*}
Of course we have $\tilde{{\bf K}}_{k}={\bf K}_{k}$
for all $k\neq0$;
we modified only the term corresponding to $k=0$. This allows to apply to the nonsingular part of the propagator the well developed theory for the corresponding Schr\"odinger equation, as we shall see in the next subsection.

If $P_{>}$ and $P_{<}$ are the projections
on the harmonics with $k\ge1$ and $k\le-1$ respectively,
and $P_{0}$ is the projection on the component $k=0$
(see \eqref{def-pro}), we have
\begin{equation}\label{d-propag}
  e^{it\mathcal{D}_{A,\gamma}}f=
  e^{-it \sqrt{L}}P_{>}f+
  e^{it \sqrt{L}}P_{<}f+
  e^{it\mathcal{D}_{A,\gamma}}P_{0}f.
\end{equation}
Estimates \eqref{est:dis1}, \eqref{est:dis0} and \eqref{est:disq}
in Theorem \ref{thm:disper} are proved separately in the next
three subsections.

\subsection{Proof of \eqref{est:dis1}}

Since 
$e^{it\mathcal{D}_{A,\gamma}}P_{>}f
  =e^{-it \sqrt{L}}P_{>}f$,
in order to prove \eqref{est:dis1}
it will be sufficient to estimate $e^{-it \sqrt{L}}$;
the component $e^{it\mathcal{D}_{A,\gamma}}P_{<}$
is treated exactly in the same way.
We split
\begin{equation*}
  e^{-it \sqrt{L}}=\sum_{j\in \mathbb{Z}}U_{j}(t)
\end{equation*}
where
\begin{equation}\label{eq:uoft}
U_{j}(t)f=\int_{0}^\infty \int_{0}^{2\pi}\tilde{\bf K}^{(j)}(t,r,\theta,r_2,\theta_2) f_j(r_2,\theta_2)  d\theta_2 \;r_2dr_2,
\end{equation}
\begin{equation}\label{D-kernel-l}
  \tilde{\bf K}^{(j)}(t,r,\theta,r_2,\theta_2) =
  \begin{pmatrix}
    m^{(j)}_{|k-\alpha|} & 0 \\
    0 &  m^{(j)}_{|k-\alpha+1|},
  \end{pmatrix}
\end{equation}
\begin{equation*}
  m^{(j)}_{\nu}(t,r,r_{2})=
  \int_{0}^{\infty}e^{-it \rho}\varphi(2^{-j}\rho)
  J_{\nu}(r \rho)J_{\nu}(r_{2}\rho)\rho d \rho.
\end{equation*}
Cutting away the components $k\le 0$ with the projection
$P_{>}$, we see that
\begin{equation}\label{eq:equiv}
  U_{j}(t)P_{>}f=
  \varphi(2^{-j}\sqrt{L})
  e^{-it \sqrt{L}}P_{>}f=
  \varphi(2^{-j}|\mathcal{D}_{A,\gamma}|)
  e^{it \mathcal{D}_{A,\gamma}}P_{>}f
\end{equation}
for all $\varphi\in \mathcal{C}_{c}^{\infty}(\mathbb{R})$ and
all $j\in \mathbb{Z}$.
 
\begin{proposition}\label{prop:l-disper} 
  Let $\varphi\in\mathcal{C}_c^\infty([1,2])$.
  Define $U_{j}(t)$ as in 
  \eqref{eq:uoft}, \eqref{D-kernel-l}.
  Then there exists a constant $C$ such that for all $j\in\Z$,
  $f\in [L^1{(\R^2)}]^2$ we have
  \begin{equation}\label{est:dis1-1}
  \begin{split}
  \left\|U_{j}(t)f(x)\right\|_{[L^\infty(\R^2)]^2}\leq C2^{2j}(1+2^{j}|t|)^{-\frac12} \|f\|_{[L^1{(\R^2)]^2}}.
  \end{split}
  \end{equation}
\end{proposition}


Now, the proof of \eqref{est:dis1-1} is very similar to the one of the corresponding estimate for the half-wave propagator as developed in \cite{FZZ,GYZZ}; we include a proof in the appendix for the sake of completeness. 
Finally, recalling \eqref{eq:equiv}, and thanks to the fact that 
$\tilde \varphi=1$ on the support of $\varphi$,
we see that \eqref{est:dis1} is an immediate consequence of 
\eqref{est:dis1-1}.

\subsection{Proof of \eqref{est:dis0}} 

We now focus on the singular component of the propagator
$e^{it \mathcal{D}_{A,\gamma}}P_{0}$.
Recall that we have chosen a special value of $\gamma$ such that
$\sin\gamma=0$ or $\cos\gamma=0$, depending on the
value of $\alpha$, see \eqref{eq:convention}.

Since any element of the range of $P_{0}$ can be written as
\eqref{equ:f0} below, estimate \eqref{est:dis0} is an
immediate consequence of the following result:

\begin{proposition}\label{prop:l-disper0} 
  Let $\varphi\in\mathcal{C}_c^\infty([1,2])$, $\tilde{\varphi}\in\mathcal{C}_c^\infty([1/2,4])$ with $\tilde{\varphi}=1$ on the support of $\varphi$ and let
  for some $c\in \mathbb{C}$ and $(\phi_{0},\psi_{0})\in Y$
  \begin{equation}\label{equ:f0}
  f=
  c
  \begin{pmatrix}
    \cos \gamma \cdot K_{\alpha}(r) \\
    i\sin \gamma \cdot K_{1-\alpha}(r) e^{i\theta}
  \end{pmatrix}
  +\begin{pmatrix}
     \phi_{0}(r) \\
    e^{i\theta}\ \psi_{0}(r)
  \end{pmatrix}.
  \end{equation}
  Then there exists a constant $C$ such that for all 
  $j\in \mathbb{Z}$ one has
  \begin{equation}\label{est:dis0-1}
  \begin{split}
  \Big\|\big(1+&(2^j |x|)^{-\alpha}\big)^{-1}\varphi(2^{-j}|{\mathcal{D}}_{A,\gamma}|)e^{it\mathcal{D}_{A,\gamma}}f(x)\Big\|_{[L^\infty(\R^2)]^2}
  \\&\leq C2^{2j}(1+2^{j}|t|)^{-\frac12} \|\big(1+(2^j |x|)^{-\alpha}\big)\tilde{\varphi}(2^{-j}|{\mathcal{D}}_{A,\gamma}|) f\|_{[L^1{(\R^2)]^2}}.
  \end{split}
  \end{equation}
\end{proposition}



\begin{proof}
We start by writing
$f_{j}=\tilde\varphi(2^{-j}|{\mathcal{D}}_{A,\gamma}|)f$
and
\begin{equation*}
\varphi(2^{-j}|{\mathcal{D}}_{A,\gamma}|)e^{it\mathcal{D}_{A,\gamma}}
f=\int_{0}^\infty \int_{0}^{2\pi}
\Phi_0(\theta) {\bf K}_{0}^{(j)}(t,r,r_2)\overline{\Phi_0(\theta_2)}
 f_j(r_2,\theta_2)  d\theta_2 \;r_2dr_2
\end{equation*}
where

\begin{itemize}
\item if $\alpha\in(0,\frac12]$,
 \begin{equation}\label{equ:F1}
\begin{split}
  \Phi_0(\theta) {\bf K}_{0}^{(j)}(t,r,r_2)\overline{\Phi_0(\theta_2)}
=\left(\begin{matrix} F_{\alpha,j}& 0 \\0 & 
 e^{i(\theta-\theta_2)} E_{1-\alpha,j}\end{matrix}\right),
\end{split}
\end{equation}
\item if $\alpha\in(\frac12,1)$,
 \begin{equation}\label{equ:F2}
\begin{split}
  \Phi_0(\theta) {\bf K}_{0}^{(j)}(t,r,r_2)\overline{\Phi_0(\theta_2)}=\left(\begin{matrix} E_{\alpha,j}& 0 \\0 &  e^{i(\theta-\theta_2)} F_{1-\alpha,j}\end{matrix}\right),
\end{split}
\end{equation}
\end{itemize}
and for any $\alpha\in(0,1)$
\begin{equation}\label{equ:F-E-j}
\begin{split}
F_{\alpha,j}(t;r,r_2) =\frac1{2\pi}\int_0^\infty e^{-it\rho}J_{-\alpha}(r\rho)J_{-\alpha}(r_2\rho) \varphi(2^{-j}\rho)\,\rho d\rho,\\
E_{\alpha,j}(t;r,r_2) =\frac1{2\pi}\int_0^\infty e^{-it\rho}J_{\alpha}(r\rho)J_{\alpha}(r_2\rho) \varphi(2^{-j}\rho)\,\rho d\rho.
 \end{split}
\end{equation}

We shall need the following properties of Bessel functions.

\begin{lemma}\label{lem:bessel}
  Let $\nu_{-}=\max\{0,-\nu\}$.
  For all $x,\nu\in \mathbb{R}$ we have:
  \begin{equation}\label{eq:bess1}
    |J_{\nu}(x)|\le c_{\nu}|x|^{\nu}\bra{x}^{-\nu-\frac 12},
    \qquad
    |J_{\nu}(x)|\le c_{\nu}(1+|x|^{-\nu_{-}}),
  \end{equation}
  \begin{equation}\label{eq:bess2}
    |J'_{\nu}(x)|=
    |J_{\nu-1}(x)-\nu J_{\nu}(x)/x|
    \le c_{\nu}|x|^{\nu-1}\bra{x}^{-\nu+\frac 12}.
  \end{equation}
  Moreover we can write
  \begin{equation}\label{eq:bess3}
    J_{\nu}(x)=x^{-1/2}(e^{ix}a_{+}(x)+e^{-ix}a_{-}(x))
  \end{equation}
  for two functions $a_{\pm}$ depending on $\nu,x$
  and satisfying for all $N\ge1$ and $|x|\ge1$
  \begin{equation}\label{eq:bess4}
    |a_{\pm}(x)|\le c_{\nu,0},
    \qquad
    |\partial_{x}^{N}a_{\pm}(x)|\le c_{\nu,N}
    |x|^{-N-1}.
  \end{equation}
  Here $c_{\nu}$ and $c_{\nu,N}$ are various constants 
  depending only on $\nu$ and $\nu,N$ respectively.
\end{lemma}

\begin{proof}
  Since $J_{\nu}(z)=z^{\nu}\widetilde{J}_{\nu}(z)$
  where $\widetilde{J}_{\nu}$ is an even entire function,
  it is sufficient to prove the estimates for $x>0$.

  From the Taylor series of $J_{\nu}$ we have 
  $|J_{\nu}(x)|\le C|x|^{\nu}$ for $|x|\le1$, and the estimate
  $|J_{\nu}(x)|\le C|x|^{-\frac 12}$ for $|x|\ge1$
  and fixed $\nu\in \mathbb{R}$ is well known
  (it also follows from 
  \eqref{eq:bess3}--\eqref{eq:bess4} proved below).
  From these estimates we obtain \eqref{eq:bess1},
  while \eqref{eq:bess2} follows from \eqref{eq:bess1}
  and the identity $J'_{\nu}=J_{\nu-1}-\nu J_{\nu}/x$.

  Next, we recall the integral representation of the
  Macdonald function $K_{\nu}(z)$:
  \begin{equation}\label{eq:macd}
    K_{\nu}(z)=b_{\nu}z^{-1/2}e^{-z}
    \int_{0}^{\infty}
    e^{-t}t^{\nu-1/2}
    \left(1+\frac{t}{2z}\right)^{\nu-1/2}dt
  \end{equation}
  where
  $b_{\nu}=\sqrt{\frac \pi2}/\Gamma(\nu+\frac 12)$.
  The representation \eqref{eq:macd} is valid for
  all $z\not\in(-\infty,0]$ and $\Re \nu>-\frac 12$
  (see e.g.~\cite{EMO} page 19).
  By the standard connection formulas
  \begin{equation*}
    \textstyle
    K_{\nu}(z)=
    \frac{i\pi}{2}e^{i\nu\pi/2}H^{(1)}_{\nu}(iz)=
    -\frac{i\pi}{2}e^{-i\nu\pi/2}H^{(2)}_{\nu}(-iz)
  \end{equation*}
  we deduce the analogous representations, for some
  constants $b'_{\nu},b''_{\nu}$ and
  $\Re \nu>-1/2$,
  \begin{equation}\label{eq:hank1}
    H_{\nu}^{(1)}(z)=b'_{\nu}
    z^{-1/2}e^{iz}
    \int_{0}^{\infty}
    e^{-t}t^{\nu-1/2}
    \left(1+\frac{it}{2z}\right)^{\nu-1/2}dt,
  \end{equation}
  \begin{equation}\label{eq:hank2}
    H_{\nu}^{(2)}(z)=b''_{\nu}
    z^{-1/2}e^{-iz}
    \int_{0}^{\infty}
    e^{-t}t^{\nu-1/2}
    \left(1+\frac{t}{2iz}\right)^{\nu-1/2}dt.
  \end{equation}
  Both \eqref{eq:hank1}--\eqref{eq:hank2} are valid at
  least for $\Re z>0$, and we shall use them for
  $z=x\in(0,\infty)$ and $\nu>-1/2$.
  Consider the function on $[1,\infty)$
  \begin{equation*}
    a(x)=
    b'_{\nu}\int_{0}^{\infty}
    e^{-t}t^{\nu-1/2}
    \left(1+\frac{it}{2x}\right)^{\nu-1/2}dt.
  \end{equation*}
  Since 
  $|\left(1+\frac{it}{2x}\right)^{\nu-1/2}|=
    (1+\frac{t^{2}}{4x^{2}})^{\frac \nu2-\frac 14}\le
    c_{\nu}(1+ \bra{t}^{\nu-1/2})$,
  we have obviously 
  \begin{equation}\label{eq:esta}
    |a(x)|\le c_{\nu}
    \quad\text{for}\quad x\ge1.
  \end{equation}
  Next we have
  \begin{equation*}
    \left|
    \partial_{x}\left(1+\frac{it}{2x}\right)^{\nu-1/2}
    \right|=
    \left|
    (\nu-1/2)
    \left(1+\frac{it}{2x}\right)^{\nu-3/2}\frac{it}{2x^{2}}
    \right|
    \le
    c_{\nu}
    (1+\bra{t}^{\nu-3/2})t x^{-2}
  \end{equation*}
  and more generally, for some polynomial $P_{N}(t)$,
  \begin{equation*}
    \left|
    \partial_{x}^{N}\left(1+\frac{it}{2x}\right)^{\nu-1/2}
    \right|\le
    c_{\nu,N}P_{N}(t)x^{-N-1},
    \qquad
    x\ge1.
  \end{equation*}
  This implies, by differentiating under the integral sign,
  \begin{equation}\label{eq:estan}
    |\partial_{x}^{N}a(x)|\le
    c_{\nu,N}x^{-N-1},
    \qquad
    x\ge1.
  \end{equation}
  Recalling \eqref{eq:hank1} we have thus proved,
  for $x\ge1$, $\nu>-1/2$,
  \begin{equation*}
    H^{(1)}_{\nu}(x)=x^{-1/2}e^{ix}a(x)
    \quad\text{with $a(x)$ satisfying }\quad 
    \eqref{eq:esta}, \eqref{eq:estan}.
  \end{equation*}
  In a similar way we prove, for $x\ge1$, $\nu>-1/2$,
  \begin{equation*}
    H^{(2)}_{\nu}(x)=x^{-1/2}e^{-ix}b(x)
    \quad\text{with $b(x)$ satisfying }\quad 
    \eqref{eq:esta}, \eqref{eq:estan}.
  \end{equation*}
  Since $2J_{\nu}=H^{(1)}_{\nu}+H^{(2)}_{\nu}$,
  we have proved \eqref{eq:bess3}--\eqref{eq:bess4}
  for all $\nu>-1/2$, $x\ge1$.
  Recalling the identity
  \begin{equation*}
    J_{\nu-1}(x)=\frac{2\nu}{x}J_{\nu}(x)-J_{\nu+1}(x)
  \end{equation*}
  we obtain that \eqref{eq:bess3}--\eqref{eq:bess4} are
  valid for $\nu\in(-3/2,-1/2]$, and by induction
  we conclude the proof for all $\nu\le-1/2$.
\end{proof}

We apply Lemma \ref{lem:bessel} to prove the following

\begin{lemma}\label{lem:hankelb}
  Let $\nu\ge-1/2$, $\nu_{-}=\max\{0,-\nu\}$,
  $\phi\in C_{c}^{\infty}([1,2])$.
  Then the integral
  \begin{equation}\label{eq:intI}
    I_{\nu}(t;r_{1},r_{2})=
    \int_{0}^{\infty}e^{it \rho}
      J_{\nu}(r_{1}\rho)J_{\nu}(r_{2}\rho) \phi(\rho)\rho d\rho
  \end{equation}
  satisfies the following estimate for all
  $r_{1},r_{2},t>0$:
  \begin{equation}\label{eq:intIest}
    |I_{\nu}(t;r_{1},r_{2})|\le
    c_{\nu}\bra{t}^{-1/2}(1+r_{1}^{-\nu_{-}})(1+r_{2}^{-\nu_{-}}).
  \end{equation}
\end{lemma}

\begin{proof}
  Note that from
  \eqref{eq:bess1} we get, for $r>0$ and $\rho\in[1,2]$,
  \begin{equation}\label{eq:Jrr}
    |J_{\nu}(r \rho)|\le C(1+r^{-\nu_{-}})
  \end{equation}
  and (denoting by a prime the derivative $\partial_{\rho}$)
  \begin{equation}\label{eq:Jprr}
    |[J_{\nu}(r \rho)]'|=
    r|J'_{\nu}(r \rho)|\le C
    r^{\nu}\bra{r \rho}^{-\nu+1/2}.
  \end{equation}
  Note also that for $a_{\pm}$ satisfying \eqref{eq:bess4}
  we have, still for $r>0$ and $\rho\in[1,2]$,
  \begin{equation}\label{eq:arr}
    |a_{\pm}(r \rho)'|\le C r(r \rho)^{-2}\le Cr^{-1}
  \end{equation}
  In particular, \eqref{eq:Jrr} implies
  \begin{equation*}
    |I_{\nu}|\le C(1+r_{1}^{-\nu_{-}})(1+r_{2}^{-\nu_{-}})
    \quad\text{for all}\quad r_{1},r_{2},t>0
  \end{equation*}
  so that \eqref{eq:intIest} is true if $t\le 1$.
  It remains to consider $t\ge1$ only.

  \underline{Case $r_{1},r_{2}\le 1$}. 
  If $r\le 1$ \eqref{eq:Jprr} reduces to
  \begin{equation}\label{eq:Jprra}
    |J_{\nu}(r \rho)'|\le C r^{\nu}.
  \end{equation}
  Now, integration by parts in $I$ gives
  \begin{equation*}
    I_{\nu}=\frac{i}{t}
    \int_{0}^{\infty}
      (\rho \phi(\rho)J_{\nu}(r_{1}\rho) J_{\nu}(r_{2}\rho))'
      e^{it \rho}d\rho.
  \end{equation*}
  Expanding the derivative, we obtain: a term
  $(\rho \phi(\rho))'J_{\nu}(r_{1}\rho) J_{\nu}(r_{2}\rho)$
  which can be estimated as before by \eqref{eq:Jrr};
  a second term of the form
  \begin{equation*}
    \rho \phi(\rho)J_{\nu}(r_{1}\rho)' J_{\nu}(r_{2}\rho)
  \end{equation*}
  which we estimate using \eqref{eq:Jrr} and
  \eqref{eq:Jprra}; and a third symmetric term.
  Summing up,
  \begin{equation*}
    |(\rho \phi(\rho)J_{\nu}(r_{1}\rho) J_{\nu}(r_{2}\rho))'|
    \le C(1+r_{1}^{-\nu_{-}})(1+r_{2}^{-\nu_{-}})
  \end{equation*}
  which implies
  $|I|\le C t^{-1}(1+r_{1}^{-\nu_{-}})(1+r_{2}^{-\nu_{-}})$,
  and \eqref{eq:intIest} follows.

  \underline{Case $r_{1}\le 1$, $r_{2}\ge 1$} (or conversely).
  Using \eqref{eq:bess3}, we are reduced to estimate the
  two integrals
  \begin{equation*}
    I_{\pm}=\int_{0}^{\infty}
    \rho \phi(\rho)J_{\nu}(r_{1}\rho)(r_{2}\rho)^{-1/2}
    e^{i(t\pm r_{2})\rho}a_{\pm}(r_{2}\rho)d \rho.
  \end{equation*}
  If $\frac 12r_{2}\le t\le 2 r_{2}$ we have by \eqref{eq:Jrr}
  \begin{equation*}
    |I_{\pm}|\le C
    (1+r_{1}^{-\nu_{-}})r_{2}^{-1/2}\le
    C(1+r_{1}^{-\nu_{-}})t^{-1/2}
  \end{equation*}
  which implies \eqref{eq:intIest}.
  In the remaining cases we integrate by parts:
  \begin{equation*}
    I_{\pm}=\frac{ir_{2}^{-1/2}}{t\pm r_{2}}
    \int_{0}^{\infty}
    (\rho^{1/2}\phi(\rho)J_{\nu}(r_{1}\rho)a_{\pm}(r_{2}\rho))'
    e^{i \rho(t\pm r_{2})}d\rho.
  \end{equation*}
  When the derivative falls on $\rho^{1/2}\phi(\rho)$,
  we use \eqref{eq:Jrr} and $|a_{\pm}|\le C$;
  when it falls on $J_{\nu}$, we use \eqref{eq:Jprra}
  and $|a_{\pm}|\le C$; finally, when it falls on $a_{\pm}$,
  we use \eqref{eq:Jrr} and \eqref{eq:Jprra},
  which implies $|a_{\pm}(r_{2}\rho)|\le Cr_{2}^{-1}\le C$.
  Summing up we obtan
  \begin{equation*}
    |(\rho^{1/2}\phi(\rho)J_{\nu}(r_{1}\rho)a_{\pm}(r_{2}\rho))'|
    \le Cr_{1}^{-\nu_{-}}
  \end{equation*}
  which gives
  \begin{equation*}
    |I_{\pm}|\le C \frac{r_{1}^{-\nu_{-}}r_{2}^{-1/2}}
      {|t\pm r_{2}|}\le
      C\frac{r_{1}^{-\nu_{-}}}{|t\pm r_{2}|}.
  \end{equation*}
  Recall that either $t\ge 2r_{2}$, so that
  $|t\pm r_{2}|\ge Ct$, or $t\le \frac 12r_{2}$, so that
  $|t\pm r_{2}|\ge Cr_{2}\ge Ct$, and in both cases we get
  $|I_{\pm}|\le C t^{-1}r_{1}^{-\nu_{-}}$ which is stronger
  than \eqref{eq:intIest}.

  \underline{Case $r_{1},r_{2}\ge1$}.
  In this final case we use the representation \eqref{eq:bess3}
  for both Bessel functions. We get four terms, 
  with all possible combinations of signs:
  \begin{equation*}
    I_{++},I_{+-},I_{-+},I_{--}=
    (r_{1}r_{2})^{-1/2}
    \int_{0}^{\infty}\phi(\rho) 
    a_{\pm}(r_{1}\rho)a_{\pm}(r_{2}\rho)
    e^{i \rho(t\pm r_{1}\pm r_{2})}d \rho.
  \end{equation*}
  Since $|a_{\pm}|\le C$ we have trivially
  $|I_{\pm,\pm}|\le C(r_{1}r_{2})^{-1/2}$. This is already
  sufficient when
  \begin{equation*}
    t\le 3\max\{r_{1},r_{2}\}
    \implies
    (r_{1}r_{2})^{-1/2}\le C t^{-1/2}
    \implies
    \eqref{eq:intIest}.
  \end{equation*}
  When $t\ge 3\max\{r_{1},r_{2}\}$, we have
  $t\pm r_{1}\pm r_{2}\ge \frac t3$. Then we can
  integrate by parts:
  \begin{equation*}
    I_{\pm,\pm}=
    \frac{i(r_{1}r_{2})^{-1/2}}{t\pm r_{1}\pm r_{2}}
    \int_{0}^{\infty}
    (\phi a_{\pm}a_{\pm})'e^{i \rho(t\pm r_{1}\pm r_{2})}d\rho.
  \end{equation*}
  Recalling \eqref{eq:bess4} and \eqref{eq:arr}, we see that 
  $|(\phi a_{\pm}a_{\pm})'|\le C$. We obtain
  \begin{equation*}
    |I_{\pm,\pm}|\le
    C
    \frac{(r_{1}r_{2})^{-1/2}}{t\pm r_{1}\pm r_{2}}\le
    C t^{-1},
  \end{equation*}
  and the proof is concluded.
\end{proof}

\begin{lemma}\label{lemTF} Let $F_{\alpha,j}(t;r_1,r_2) $  and $E_{\alpha,j}(t;r_1,r_2) $ be given by \eqref{equ:F-E-j}. Then there exists a constant $C_{\alpha}$ such that
 \begin{equation}\label{est:F-E-j}
\begin{split}
\big|E_{\alpha,j}(t;r_1,r_2)\big|&\leq C_{\alpha} 2^{2j}
  (1+2^j|t|)^{-1/2},\\
\big|F_{\alpha,j}(t;r_1,r_2)\big|&\leq C_{\alpha} 2^{2j}
  (1+2^j|t|)^{-1/2}\big(1+(2^j r_{1})^{-\alpha}\big)
  \big(1+(2^j r_2)^{-\alpha}\big).
\end{split}
\end{equation}
\end{lemma}

\begin{proof}
By scaling, we have
\begin{equation*}
  E_{\alpha,j}(t;r_1,r_2)=2^{2j}E_{\alpha,0}(2^jt;2^jr_1, 2^jr_2)
\end{equation*}
\begin{equation*}
  F_{\alpha,j}(t;r_1,r_2)=2^{2j}F_{\alpha,0}(2^jt;2^jr_1, 2^jr_2).
\end{equation*}
Setting $t'=2^{j}t$, $r_{1}'=2^{j}r_{1}$, $r_{2}'=2^{j}r_{2}$,
estimates \eqref{est:F-E-j} follow from
\begin{equation*}
  |E_{\alpha,0}(t,r_{1},r_{2})|\le
  C \bra{t}^{-1/2},
\end{equation*}
\begin{equation*}
  |F_{\alpha,0}(t,r_{1},r_{2})|\le
  C \bra{t}^{-1/2}(1+r_{1}^{-\alpha})(1+r_{2}^{-\alpha}).
\end{equation*}
Now $E_{\alpha,0}=I_{\alpha}$ and
$F_{\alpha,0}=I_{-\alpha}$ where $I_{\nu}$ is
the integral \eqref{eq:intI}. Thus both estimates
are special cases of \eqref{eq:intIest}.
\end{proof}

We are in position to conclude the proof of Proposition \ref{prop:l-disper0}.
Consider the operators
\begin{equation}\label{T_j}
\begin{split}
T^F_jf(t,r)&=\int_0^\infty F_{\alpha,j}(t;r,r_2) f(r_2)\, r_2 dr_2,\\  T^E_jf(t,r)&=\int_0^\infty E_{\alpha,j}(t;r,r_2) f(r_2)\, r_2 dr_2, 
\end{split}
\end{equation}
where $f$ is a scalar radial function on $\mathbb{R}^{2}$.
By \eqref{est:F-E-j} we obtain immediately
\begin{equation}\label{eq:TEj}
\Big\|T^E_jf(x)\Big\|_{L^\infty(\R^2)}
\leq C2^{2j}(1+2^j|t|)^{-1/2} \|f\|_{L^1{(\R^2)}},
\end{equation}
\begin{equation}\label{eq:TFj}
\Big\|\big(1+(2^j r)^{-\alpha}\big)^{-1}T^F_jf(x)\Big\|_{L^\infty(\R^2)}\leq C2^{2j}(1+2^j|t|)^{-1/2} \|\big(1+(2^j r)^{-\alpha}\big)f\|_{L^1{(\R^2)}}.
\end{equation}
Recalling \eqref{equ:F1}--\eqref{equ:F-E-j} we deduce
\eqref{est:dis0-1} and thus the proof is concluded.
\end{proof}

\subsection{Proof of \eqref{est:disq}} We adapt the proof of Proposition \ref{prop:l-disper0}.
We start by writing

\begin{equation*}
\varphi(2^{-j}|{\mathcal{D}}_{A,\gamma}|)e^{it\mathcal{D}_{A,\gamma}}
f=\int_{0}^\infty \int_{0}^{2\pi}
\Phi_0(\theta) {\bf K}_{0}^{(j)}(t,r,r_2)\overline{\Phi_0(\theta_2)}
 f_j(r_2,\theta_2)  d\theta_2 \;r_2dr_2
\end{equation*}
where $f_{j}=\tilde\varphi(2^{-j}|{\mathcal{D}}_{A,\gamma}|)f$ and the kernel is given by \eqref{equ:F1}, \eqref{equ:F2} and \eqref{equ:F-E-j}. The main issue is the term \eqref{equ:F-E-j} with Bessel functions of negative order.
 By a scaling argument, it suffices to prove \eqref{est:disq} in the case $j=0$, which is a consequence of the following

\begin{lemma}\label{lem:est-qq'}
  Let $-1/2<\nu<0$. Let $T_\nu$ be the operator defined as
  \begin{equation}\label{Tnu-operator}
\begin{split}
(T_{\nu}g)(t,r_1)=\int_0^\infty  I_{\nu}(t;r_{1},r_{2}) g(r_2)\, r_2 dr_2
 \end{split}
\end{equation}
where $  I_{\nu}(t;r_{1},r_{2})$ is given in \eqref{eq:intI}. Then 
 for $2\leq q<-\frac2\nu$ the following estimate holds
  \begin{equation}\label{est:q-q'}
  \|T_{\nu}g\|_{L^q_{r_1 dr_1}}\le
    c_{\nu}(1+|t|)^{-\frac12(1-\frac 2q)}\|g\|_{L^{q'}_{r_2 dr_2}}.
  \end{equation}
\end{lemma}

\begin{proof}
Let $\chi\in \mathcal{C}_c^\infty ([0,+\infty)$ be defined as 
\begin{equation}
\chi(r)=
\begin{cases}1,\quad r\in [0, \frac12],\\
0, \quad r\in [1,+\infty)
\end{cases}
\end{equation}
and let us set $\chi^c=1-\chi$. Then we decompose the kernel $  I_{\nu}(t;r_{1},r_{2})$ into four terms as follows:
  \begin{equation}
\begin{split}
 I_{\nu}(t;r_{1},r_{2})=&\chi(r_1)I_{\nu}(t;r_{1},r_{2})\chi(r_2)+\chi^c(r_1)I_{\nu}(t;r_{1},r_{2})\chi(r_2)\\
&+\chi(r_1)I_{\nu}(t;r_{1},r_{2})\chi^c(r_2)+\chi^c(r_1)I_{\nu}(t;r_{1},r_{2})\chi^c(r_2).
 \end{split}
\end{equation}
This yields a corresponding decomposition for the operator $T_{\nu}=T^1_{\nu}+T^2_{\nu}+T^3_{\nu}+T^4_{\nu}$. We thus estimate separately the norms $\|T^j_{\nu}g\|_{L^q_{r_1 dr_1}}$ for $j=1,2,3,4$.

For what concerns $T^1_\nu$, by Lemma \ref{lem:hankelb}, we have that 
  \begin{equation}
\begin{split}
|\chi(r_1)I_{\nu}(t;r_{1},r_{2})\chi(r_2)|\leq (1+|t|)^{-1}\chi(r_1)r_1^{\nu}\chi(r_2)r_2^{\nu};
 \end{split}
\end{equation}
therefore, as long as $2\leq q<-\frac2\nu$, we can write
  \begin{equation}\label{est:q-q'1}
  \|T^1_{\nu}g\|_{L^q_{r_1 dr_1}}\le
    c_{\nu}(1+|t|)^{-1}\Big(\int_0^1 r^{\nu q} r dr\Big)^{2/q}\|g\|_{L^{q'}_{r_2 dr_2}}\le
    c_{\nu}(1+|t|)^{-1}\|g\|_{L^{q'}_{r_2 dr_2}}.
  \end{equation}
  
The terms $T^2_{\nu}$ and $T^3_{\nu}$ are similar, so we only deal with $T^3_{\nu}$. We need a refinement of the argument developed in the proof of Lemma \ref{lem:hankelb} in the Case $r_{1}\le 1$, $r_{2}\ge 1$.  
We need to estimate the two integrals
  \begin{equation*}
    I_{\pm}=\int_{0}^{\infty}
    \rho \phi(\rho)J_{\nu}(r_{1}\rho)(r_{2}\rho)^{-1/2}
    e^{i(t\pm r_{2})\rho}a_{\pm}(r_{2}\rho)d \rho.
  \end{equation*}
 If $|t|\leq 1$, thanks to an integration by parts we obtain
  \begin{equation*}
   | I_{\pm}|=\Big|\frac{ir_{2}^{-1/2}}{\pm r_{2}}
    \int_{0}^{\infty}
    (\rho^{1/2}\phi(\rho)J_{\nu}(r_{1}\rho)a_{\pm}(r_{2}\rho)e^{it\rho})'
    e^{\pm i \rho r_{2}}d\rho\Big|\leq Cr_1^{\nu}r_2^{-\frac32} .
  \end{equation*}
 Hence if $|t|\leq 1$, we have
      \begin{equation}
  \|T^3_{\nu}g\|_{L^q_{r_1 dr_1}}\le
    c_{\nu}\Big(\int_0^1 r_1^{\nu q} r_1 dr_1\Big)^{1/q}\Big(\int_{\frac12}^{+\infty} r_2^{-\frac{3q}2} r_2 dr_2\Big)^{1/q}\|g\|_{L^{q'}_{r_2 dr_2}}\le
    c_{\nu}\|g\|_{L^{q'}_{r_2 dr_2}}.
  \end{equation}
It remains to consider the region $|t|\geq 1$. If $|t\pm r_2|\leq 1$, we have by \eqref{eq:Jrr}
  \begin{equation*}
    |I_{\pm}|\le C
    (1+r_{1}^{-\nu_{-}})r_{2}^{-1/2}.
  \end{equation*}
If $|t\pm r_2|\geq 1$, we integrate by parts to obtain
 \begin{equation*}
    |I_{\pm}|\le C \frac{r_{1}^{-\nu_{-}}r_{2}^{-1/2}}
      {|t\pm r_{2}|}.
  \end{equation*}
   Therefore, we can write
      \begin{equation}\label{est:q-q'2}
    \begin{split}
  &\|T^3_{\nu}g\|^q_{L^q_{r_1 dr_1}}\\
   &\leq\int_0^1\Big(\int_{|t\pm r_2|\leq 1} r_{1}^{\nu}r_{2}^{-1/2}| g(r_2)| r_2\, dr_2 +\int_{1\leq |t\pm r_2|\leq |t|/2} r_{1}^{\nu}r_{2}^{-1/2}|t\pm r_2|^{-1}| g(r_2)| r_2\, dr_2 \\
   &\quad+\int_{|t\pm r_2|\geq |t|/2} r_{1}^{\nu} r_{2}^{-1/2}|t\pm r_2|^{-1}| g(r_2)| r_2\, dr_2\Big)^q r_1\, dr_1
    \end{split}
  \end{equation}
  where in all integrals we can add the condition $r_{2}\ge1$ if
  necessary.
  Since $|t|\geq 1$ and $|t\pm r_2|\leq 1/10$, one has that $r_2\sim |t|$. Therefore, on the one hand, we have
      \begin{equation}
    \begin{split}
&\int_{|t\pm r_2|\leq 1} r_{2}^{-1/2}| g(r_2)| r_2\, dr_2\\
&\quad\leq \Big(\int_{|t\pm r_2|\leq 1} r_{2}^{-q/2} r_2\, dr_2\Big)^{\frac1q} \|g\|_{L^{q'}_{r_2 dr_2}}\leq C|t|^{\frac1q-\frac12}\|g\|_{L^{q'}_{r_2 dr_2}}.
    \end{split}
  \end{equation}
  On the other hand, since $|t\pm r_2|\leq |t|/2$ implies $r_2\sim |t|$, we have
      \begin{equation}
    \begin{split}
&\int_{1\leq|t\pm r_2|\leq |t|/2} r_{2}^{-1/2}(1+|t\pm r_2|)^{-1}| g(r_2)| r_2\, dr_2\\
&\lesssim |t|^{-1/2}\int_{1\leq|t\pm r_2|\leq |t|/2} (1+|t\pm r_2|)^{-1}| g(r_2)| r_2\, dr_2\\
&\quad\lesssim |t|^{-1/2+1/q} \Big(\int_{ |t|/2}^{\frac{3|t|}2} (1+|t\pm r_2|)^{-q} \, dr_2\Big)^{\frac1q} \|g\|_{L^{q'}_{r_2 dr_2}}\leq C|t|^{\frac1q-\frac12}\|g\|_{L^{q'}_{r_2 dr_2}}.
    \end{split}
  \end{equation}
  Finally  we estimate 
  \begin{equation}
    \begin{split}
&\int_{|t\pm r_2|\geq |t|/2} r_{2}^{-1/2}|t\pm r_2|^{-1}| g(r_2)| r_2\, dr_2\\
&\leq C \Big(\int_{|t\pm r_2|\geq |t|/2}   r_{2}^{-q/2}(1+|t\pm r_2|)^{-q} \, r_2dr_2\Big)^{\frac1q} \|g\|_{L^{q'}_{r_2 dr_2}}\\
&\leq C|t|^{\frac1q-\frac12} \Big(\int_{|t\pm r_2|\geq |t|/2}   r_{2}^{-q/2}(1+|t\pm r_2|)^{-\frac q2-1} \, r_2dr_2\Big)^{\frac1q} \|g\|_{L^{q'}_{r_2 dr_2}}.
    \end{split}
  \end{equation}
  The last integral is uniformly bounded, indeed, recalling that
  $r_{2}\ge1$, we have for $q\ge2$
  \begin{equation*}
    \int_{|t\pm r_2|\geq |t|/2}r_{2}^{-q/2}
    (1+|t\pm r_2|)^{-\frac q2-1} \, r_2dr_2 \le
    \int_{-\infty}^{+\infty}
    (1+|t\pm r_2|)^{-\frac q2-1}dr_2=
    \int(1+|r_2|)^{-\frac q2-1}dr_2=C.
  \end{equation*}
  Thus we obtain for $q\ge2$
  \begin{equation*}
    \int_{|t\pm r_2|\geq |t|/2} r_{2}^{-1/2}
    |t\pm r_2|^{-1}| g(r_2)| r_2\, dr_2\le
    C|t|^{\frac 1q-\frac 12}\|g\|_{L^{q'}_{r_{2}dr_{2}}}.
  \end{equation*}
  In conclusion, we have shown that,   for $2\le q<-\frac{2}{\nu}$
      \begin{equation}\label{est:q-q'2}
    \begin{split}
  &\|T^3_{\nu}g\|_{L^q_{r_1 dr_1}}\leq C |t|^{\frac1q-\frac{1}2}\int_0^1r_1^{\nu q} r_1\, dr_1 \|g\|_{L^{q'}_{r_2 dr_2}}\leq C |t|^{\frac1q-\frac{1}2}\|g\|_{L^{q'}_{r_2 dr_2}}.
    \end{split}
  \end{equation}

We finally deal with $T^4_{\nu}$. By Lemma \ref{lem:hankelb}, we have
 \begin{equation}
  \|T^4_{\nu}g\|_{L^\infty_{r_1 dr_1}}\le
    c_{\nu}(1+|t|)^{-\frac12}\|g\|_{L^{1}_{r_2 dr_2}}.
  \end{equation}
  Since the Hankel transform is unitary on $L^2_{rdr}$, we get
     \begin{equation}
  \|T^4_{\nu}g\|_{L^2_{r_1 dr_1}}\le   \|\mathcal{H}_\nu e^{it\rho}\phi(\rho) \mathcal{H}_\nu e^{it\rho} g\|_{L^2_{r_1 dr_1}}\leq C \|g\|_{L^{2}_{r_2 dr_2}}.
  \end{equation}
  By interpolation, we obtain 
      \begin{equation}\label{est:q-q'4}
    \begin{split}
  &\|T^4_{\nu}g\|_{L^q_{r_1 dr_1}}\leq C (1+|t|)^{-\frac12(1-\frac2q)}\|g\|_{L^{q'}_{r_2 dr_2}}.
    \end{split}
  \end{equation}
  Collecting the estimates on the terms $T^j_\nu$, yields \eqref{est:q-q'} and the proof is concluded.
\end{proof}

\section{Proof of the Strichartz estimates}\label{sec:stri}

This section is devoted to the proof of Theorems \ref{thm-stri1}, \ref{thm-stri3} and \ref{thm-stri2}.
The proof of Theorem \ref{thm-stri1} relies on the abstract Keel--Tao argument \cite{KT} and on the localized dispersive estimates  \eqref{est:dis1}, and happens to be quite straightforward. On the other hand, the proofs of
Theorems \ref{thm-stri2} and \ref{thm-stri3}, which are inspired by \cite{CYZ}, are much more delicate, as they require a more careful analysis of Bessel functions.

\subsection{Proof of  Theorem \ref{thm-stri1}} 

We use the following variant of the Keel--Tao argument,
see \cite{Zhang} for a proof:

\begin{proposition}\label{prop:semi}
Let $(X,\mathcal{M},\mu)$ be a $\sigma$-finite measured space and
$U: \mathbb{R}\rightarrow B(L^2(X,\mathcal{M},\mu))$ be a weakly
measurable map satisfying, for some constants $C$, $\kappa\geq0$,
$\sigma, h>0$,
\begin{equation}\label{md}
\begin{split}
\|U(t)\|_{L^2\rightarrow L^2}&\leq C,\quad t\in \mathbb{R},\\
\|U(t)U(s)^*f\|_{L^\infty}&\leq
Ch^{-\kappa}(h+|t-s|)^{-\sigma}\|f\|_{L^1}.
\end{split}
\end{equation}
Then for every pair $q,p\in[1,\infty]$ such that $(p,q,\sigma)\neq
(2,\infty,1)$ and
\begin{equation*}
\frac{1}{p}+\frac{\sigma}{q}\leq\frac\sigma 2,\quad p\ge2,
\end{equation*}
there exists a constant $\tilde{C}$ only depending on $C$, $\sigma$,
$q$ and $p$ such that
\begin{equation*}
\Big(\int_{\mathbb{R}}\|U(t) u_0\|_{L^q}^p dt\Big)^{\frac1p}\leq \tilde{C}
\Lambda(h)\|u_0\|_{L^2}
\end{equation*}
where $\Lambda(h)=h^{-(\kappa+\sigma)(\frac12-\frac1q)+\frac1p}$.
\end{proposition}


To prove estimate \eqref{stri-D1} we write again as in \eqref{d-propag}
 \begin{equation*}
e^{it\mathcal{D}_{A,\gamma}}P_{>}= e^{-it \sqrt{L}}P_{>}
\end{equation*}
using the modified propagator defined in \eqref{D-propa'}.
Then \eqref{stri-D1} is a consequence of 
\begin{equation}\label{stri-D1-1}
\|e^{-it \sqrt{L}}f\|_{L^p_t(\R; L^q_x(\R^2))}\leq
C\| f\|_{\dot{H}^{s}_A}.
\end{equation}
with $p,q, s$ as in Theorem \ref{thm-stri1}. 
Using a Paley--Littlewood frequency decomposition
(see \eqref{LP-dp}) we define
$f_j=\varphi_j(\sqrt{L})f$ for all $j\in \mathbb{Z}$.
By the square-function
estimates \eqref{square1} and the Minkowski inequality, we obtain
\begin{equation}\label{LP}
\|e^{-it\sqrt{L}}f\|_{L^p(\R;L^q(\R^2))}\lesssim
\Big(\sum_{j\in\Z}\|e^{-it\sqrt{L}}f_j\|^2_{L^p(\R;L^q(\R^2))}\Big)^{\frac12}.
\end{equation}
Moreover, if $\tilde{\varphi}$ are cutoffs such that  $\tilde{\varphi}=1$ on the support of $\varphi$, we can write
$$e^{-it \sqrt{L} }f_j=
U_{j}(t)\tilde{f}_{j},
\qquad
U_{j}(t):=\varphi_j(\sqrt{L})e^{-it \sqrt{L} },
\qquad
\tilde{f}_j=\tilde{\varphi}_j(\sqrt{L})f.$$
By the spectral theorem, letting $g=\tilde{f}_j$ we have
\begin{equation*}
\|U_j(t)g\|_{L^2(\R^2)}\leq C\|g\|_{L^2(\R^2)}.
\end{equation*}
By the dispersive estimate \eqref{est:dis1} we have
\begin{equation*}
\|U_j(t)U_j^*(s)g\|_{L^\infty (\R^2)}=\|U_j(t-s)g\|_{L^\infty (\R^2)}\leq C 2^{\frac32 j}\big(2^{-j}+|t-s|\big)^{-\frac12}\|g\|_{L^1(\R^2)}.
\end{equation*}
Then assumptions \eqref{md} are valid
for $U_{j}(t)$ with $\alpha=3/2$, $\sigma=1/2$ and
$h=2^{-j}$, and Proposition \ref{prop:semi} gives
\begin{equation*}
\|U_j(t)g\|_{L^p(I;L^q(\R^2))}\leq C
2^{j[2(\frac12-\frac1p)-\frac1q]} \|g\|_{L^2(\R^2)}
\end{equation*}
which implies \eqref{stri-D1-1}
\begin{equation}\label{LP}
\|e^{-it\sqrt{L}}f\|_{L^p(\R;L^q(\R^2))}\lesssim
\Big(\sum_{j\in\Z}2^{2js} \|\tilde{\varphi}_j(\sqrt{L})f\|^2_{L^2(\R^2)}\Big)^{\frac12}
\end{equation}
 since $s=2(\frac12-\frac1q)-\frac1p$.

\subsection{Proof of  Theorem \ref{thm-stri2}: first part} 

We first prove the partial estimate
\begin{equation}\label{stri-D2}
  \|W(|x|) e^{it\mathcal{D}_{A,\gamma}}P_{0}f\|
  _{[L^p_t(\R; L^q_x(\R^2))]^2}
  \leq C \|P_{0} f\|_{[{H}^{s}_A]^2},
  \quad s=1-\frac1p-\frac2q
\end{equation}
for all $(p,q)$ satisfying \eqref{pqrange0}.
This estimate will follow
by complex interpolation between the two estimates
\begin{equation}\label{est:2-2}
\|W(|x|)e^{it\mathcal{D}_{A,\gamma}}P_{0}f\|
  _{[L^\infty_t(\R; L^2_x(\R^2))]^2}\leq
   C \|P_{0} f\|_{[L^2]^2}
\end{equation}
and 
\begin{equation}\label{est:inf-2}
\|W(|x|)e^{it\mathcal{D}_{A,\gamma}}P_{0}f\|_{[L^p_t(\R; L^\infty_x(\R^2))]^2}\leq
C
\|P_{0} f\|_{[H^{1-\frac1p}_A]^2}, \quad \forall p>2.
\end{equation}
We recall that the $k=0$ component of $f$ (the singular component) takes the form
 \begin{equation}\label{equ:f0b}
 P_0 f=
  c_1
  \begin{pmatrix}
    \cos \gamma \cdot K_{\alpha}(r) \\
    i\sin \gamma \cdot K_{1-\alpha}(r) e^{i\theta}
  \end{pmatrix}
  +c_2\begin{pmatrix}
     \phi_{0}(r) \\
    e^{i\theta}\ \psi_{0}(r)
  \end{pmatrix}:=\begin{pmatrix}
     g_{0}(r) \\
    e^{i\theta} g_1(r)
  \end{pmatrix}.
  \end{equation}
We need to estimate the following terms:

$\bullet$ if $\alpha\in(0,\frac12]$,
\begin{equation*}
e^{it\mathcal{D}_{A,\gamma}}
P_0 f=\int_{0}^\infty \int_{0}^{2\pi}\left(\begin{matrix} F_{\alpha}& 0 \\0 &  e^{i(\theta-\theta_2)} E_{1-\alpha}\end{matrix}\right)  \begin{pmatrix}
     g_{0}(r_2) \\
    e^{i\theta_2} g_1(r_2)
  \end{pmatrix}  d\theta_2 \;r_2dr_2;
\end{equation*}

$\bullet$ if $\alpha\in(\frac12,1)$
\begin{equation*}
e^{it\mathcal{D}_{A,\gamma}}
P_0 f=\int_{0}^\infty \int_{0}^{2\pi}\left(\begin{matrix} E_{\alpha}& 0 \\0 &  e^{i(\theta-\theta_2)} F_{1-\alpha}\end{matrix}\right)  \begin{pmatrix}
     g_{0}(r_2) \\
    e^{i\theta_2} g_1(r_2)
  \end{pmatrix}  d\theta_2 \;r_2dr_2,
\end{equation*}
where $F_{\alpha}$ and $E_{\alpha}$ 
(see \eqref{equ:F-E-j}) are defined by 
\begin{equation}\label{equ:F-E}
\begin{split}
F_{\alpha}=F_{\alpha}(t;r_1,r_2) =\frac1{2\pi}\int_0^\infty e^{it\rho}J_{-\alpha}(r_1\rho)J_{-\alpha}(r_2\rho) \,\rho d\rho,\\
E_{\alpha}=E_{\alpha}(t;r_1,r_2) =\frac1{2\pi}\int_0^\infty e^{it\rho}J_{\alpha}(r_1\rho)J_{\alpha}(r_2\rho) \,\rho d\rho.
 \end{split}
\end{equation}
Let us define
\begin{equation}\label{Tnuoperator}
\begin{split}
(T_{\nu}g)(t,r)=\int_0^\infty K_{\nu}^0(t; r, r_2) g(r_2)\, r_2 dr_2.
 \end{split}
\end{equation}
where
\begin{equation}\label{Knukernel}
\begin{split}
K_{\nu}^0(t; r, r_2)=\frac1{2\pi}\int_0^\infty e^{it\rho}J_{\nu}(r\rho)J_{\nu}(r_2\rho)\,\rho d\rho, \quad \nu\in[-1/2, 0)\cup(0,1].
 \end{split}
\end{equation}
Then, estimates \eqref{est:2-2} and \eqref{est:inf-2} 
are consequences of the following result:

\begin{lemma}\label{lemwe} 
  Let $p>2$ and $0<\epsilon\ll1$.
  We have
\begin{equation}\label{est:2-2'}
\|\omega_{\nu}(r)(T_{\nu}g)(t,r)\|_{L^\infty(\R;L^2_{rdr})}
\leq \|g\|_{L^2},
\end{equation}
and 
\begin{equation}\label{est:inf-2'}
\|\omega_{\nu}(r)(T_{\nu}g)(t,r)\|_{L^p(\R;L^\infty_{rdr})}
\leq C\|g\|_{H^{1-\frac1p}_A},
\end{equation}
where 
\begin{equation}
\omega_{\nu}(r)=
\begin{cases}
    1 &
    \ \text{if}\ \ 0< \nu\leq 1\\
    (1+r^{\nu-\epsilon})^{-1}&
    \ \text{if}\  -1/2\leq \nu<0.
\end{cases}
\end{equation}
\end{lemma}

\begin{proof}
The proof is inspired by \cite{CYZ}; we include the details for the sake of completeness.  
The operator $T_{\nu}$ can be written via the Hankel
transform:
\begin{equation}
   T_{\nu}g=\mathcal{H}_{\nu}e^{it\rho} \mathcal{H}_{\nu}g.
\end{equation}
Since $|\omega_{\nu}(r)|\leq 1$ and $\mathcal{H}_{\nu}$
is unitary on $L^{2}(rdr)$,
the proof of \eqref{est:2-2'} is trivial.
Let us focus on \eqref{est:inf-2'}.


Denote by $R,N\in 2^{\mathbb{Z}}$ dyadic numbers.
By a dyadic decomposition,  we can write
\begin{eqnarray*}
\|\omega_{\nu}(r)T_{\nu}f(t,r)\|_{L^p_tL^\infty_{dr}}^2&\leq&\left\|\omega_{\nu}(r)\sum_{N\in 2^\Z}\hank\left[e^{it{\rho}} \varphi(\frac{\rho}N)\hank f(\rho) \right](r)\right\|_{L^p_tL^\infty_{dr}(\R^+)\,}^2
\\
&\leq&
\left\|\sup_{R\in2^{\Z}} \left\|\omega_{\nu}(r)\sum_{N\in 2^\Z} \hank\left[e^{it{\rho}} \varphi(\frac{\rho}N)\hank f(\rho) \right]\right\|_{L^\infty_{dr}([R,2R])}\right\|_{L^p_t}^2,
\end{eqnarray*}
where we replace the measure $rdr$ in \eqref{est:inf-2'} by $dr$ since we focus on the $L^\infty$-norm.
By using the fact that $\ell^2 \hookrightarrow\ell^\infty$ and the Minkowski  inequality, we further obtain 
\begin{eqnarray*}
\|\omega_{\nu}(r)T_{\nu}f(t,r)\|_{L^p_tL^\infty_{dr}}^2
&\leq&
\left\| \left(\sum_{R\in2^{\Z}}\left\|\omega_{\nu}(r)\sum_{N\in 2^\Z} \hank\left[e^{it{\rho}} \varphi(\frac{\rho}N)\hank f(\rho) \right](r)\right\|^2_{L^\infty_{dr}\,([R,2R])}\right)^{1/2}\right\|_{L^p_t}^2
\\
&\leq&
 \sum_{R\in2^{\Z}}\left\|\omega_{\nu}(r)\sum_{N\in 2^\Z} \hank\left[e^{it{\rho}} \varphi(\frac{\rho}N)\hank f(\rho) \right]\right\|^2_{L^p_tL^\infty_{dr}\,([R,2R])}
\\
&\leq&
 \sum_{R\in2^{\Z}}\left(\omega_{\nu}(R)\sum_{N\in 2^\Z}\left\|    \hank\left[e^{it{\rho}} \varphi(\frac{\rho}N)\hank f(\rho) \right]\right\|_{L^p_tL^\infty_{dr}\,([R,2R])}\right)^2.
\end{eqnarray*}
Notice that in the last inequality we have used the triangle inequality instead of Littlewood--Paley square function inequality, which fails at $L^\infty_{dr}$. 
By using a scaling argument, we can further estimate with 
\begin{equation}\label{strifin}
\leq
 \sum_{R\in2^{\Z}}\biggl(\omega_{\nu}(R)\sum_{N\in 2^\Z}N^{2-\frac1p}\left\| \hank\left[e^{it{\rho}} \varphi(\rho)\hank f(N\rho) \right]\right\|_{L^p_tL^\infty_{dr}\,([NR,2NR])}\biggr)^2
 \end{equation}
Now, we need the following result, whose proof will be given at the end of this section.

\begin{proposition}\label{LRE}

Let $p\geq2$, $\varphi\in\mathcal{C}_c^\infty(\R)$ be supported in $I:=[1,2]$, $R>0$ be a dyadic number,  and $\nu\in[-\frac12, 1]$. Then, for any $0<\varepsilon\ll 1$ the following estimate holds
\begin{equation}\label{stri-L}
\bigl\|\mathcal{H}_{\nu}\big[e^{ it\rho} g(\rho)\big](r)
\bigr\|_{L^p_tL^\infty_{dr}([R/2,R])} 
\lesssim \|g\|_{L^2_{\rho d\rho}(I)}\times
\begin{cases}
R^{\nu-\frac{\varepsilon} 2}&\ \text{if}\  R\lesssim1\\
R^{\frac1p-\frac12}&\ \text{if}\  R\gg1
\end{cases}
\end{equation}
provided $g(\rho)$ is supported in $[1,2]$.
\end{proposition}

Applying this result with
$g(\rho)=\varphi(\rho)\mathcal{H}_{\nu}f(N \rho)$,
we can continue the estimate:
\begin{equation*} 
  \eqref{strifin}\leq 
  \sum_{R\in2^{\Z}}
  \biggl(\sum_{N\in 2^\Z}N^{1-\frac1p} \omega_{\nu}(R)Q(NR)
  \left\| \varphi(\frac\rho{N})\hank f(\rho) \right\|
  _{L^2_{\rho d\rho}}
  \biggr)^2
\end{equation*}
where
\begin{equation}
Q(NR)=
\begin{cases}
  (NR)^{\nu-\frac{\varepsilon} 2}& \ \text{if}\  NR\lesssim 1,\\ 
  (NR)^{\frac1p-\frac12}& \ \text{if}\ NR\gg1.
\end{cases}
\end{equation}
Since $p>2$, we have
\begin{equation}\label{q-cond}
\frac1p-\frac12<0.
\end{equation}
We now consider two cases.
\bigskip

$\bullet$ When $0<\nu\leq 1$, 
we take $0<\varepsilon<\nu$ so that
\begin{equation}\label{ST}
\mathcal{Q}_1:=\sup_{R} \sum_{N\in2^\Z} Q(NR) <\infty,\quad \mathcal{Q}_2:=\sup_{N} \sum_{R\in2^\Z} Q(NR) <\infty.
\end{equation}
Setting now
\begin{equation}
A_{N,\nu}=N^{1 -\frac 1p}
\|(\hank f)(\rho) \varphi(\rho/N)\|_{L^2_{\rho d\rho}},
\end{equation} 
by the Schur test we obtain
\begin{equation*}
  \Bigl\|\sum_{N}Q(NR)A_{N,\nu}\Bigr\|_{\ell^{2}_{R}}
  \le
  (\mathcal{Q}_{1}\mathcal{Q}_{2})^{1/2}
  \|A_{N,\nu}\|_{\ell^{2}_{N}}.
\end{equation*}
Summing up we have proved
\begin{equation*}
  \|\omega_{\nu}(r)(T_{\nu}f)(t,r)\|_{L^p(\R;L^\infty_{dr})}^{2}
  \le
  C_{\nu}\sum_{N\in2^\Z}|A_{N,\nu}|^2
  =C_{\nu}\|f\|_{\dot H^{1-\frac1p}_A}^2 
\end{equation*}
which implies \eqref{est:inf-2'} for $0<\nu\leq 1$.
\medskip

$\bullet$
For $-\frac12\leq\nu<0$ we modify the previous
argument as follows.
If $NR \lesssim 1$ we have
\begin{equation*}
  \omega_{\nu}(R)Q(NR)=
  \frac{(NR)^{\nu-\frac \epsilon2}}{1+R^{\nu-\epsilon}}=
  \frac{(NR)^{\nu-\epsilon}}
    {N^{\nu-\epsilon}+(NR)^{\nu-\epsilon}}
  (NR)^{\frac \epsilon2}N^{\nu-\epsilon}
  \le (NR)^{\frac \epsilon2}N^{\nu-\epsilon}.
\end{equation*}
If $NR\ge 1$ we have
\begin{equation*}
  \omega_{\nu}(R)Q(NR)=
  \frac{(NR)^{\frac 1p-\frac 12}}{1+R^{\nu-\epsilon}}=
  \frac{N^{\nu-\epsilon}}{N^{\nu-\epsilon}+(NR)^{\nu-\epsilon}}
  (NR)^{\frac 1p-\frac 12}\le
  (NR)^{\frac 1p-\frac 12}N^{\nu-\epsilon}.
\end{equation*}
Summing up, taking $\epsilon>0$
so small that $1-\frac1p+(\nu-\epsilon)\geq 0$, we obtain
\begin{equation*}
  N^{1-\frac1p}\omega_{\nu}(R)Q(NR)\leq 
  N^{1-\frac 1p}N^{\nu-\epsilon}\widetilde{Q}(NR)\leq
 \bra{N}^{1-\frac 1p}\widetilde{Q}(NR)
\end{equation*}
where $\bra{N}=(1+N^{2})^{1/2}$ and
\begin{equation}
\tilde{Q}(NR)=
\begin{cases}
  (NR)^{\frac{\epsilon}2}& \ \text{if}\ NR\lesssim 1 \\
  (NR)^{\frac1p-\frac12}&\ \text{if}\  NR\gg1.
\end{cases}
\end{equation}
This gives
\begin{equation*}
  \|\omega_{\nu}(r)T_{\nu}f(t,r)\|_{L^p_tL^\infty_{dr}}^2
  \leq
  \sum_{R\in2^{\mathbb{Z}}}
  \biggl(\sum_{N\in 2^\mathbb{Z}}
  \bra{N}^{1-\frac1p}\tilde{Q}(NR)
    \left\| \varphi(\frac\rho{N})\hank f(\rho) \right\|
    _{L^2_{\rho d\rho}}\biggr)^{2}.
\end{equation*}
We can now repeat the previous argument with
$Q(NR)$ replaced by $\tilde{Q}(NR)$ and $A_{N,\nu}$
replaced by
\begin{equation*}
  \tilde{A}_{N,\nu}=
  \bra{N}^{1-\frac1p}
    \left\| \varphi(\rho/N)\hank f(\rho) \right\|
    _{L^2_{\rho d\rho}}
\end{equation*}
and we obtain
\begin{equation*}
  \|\omega_{\nu}(r)T_{\nu}f(t,r)\|_{L^p_tL^\infty_{dr}}^2
  \le C_{\nu}
  \sum_{N\in 2^{\mathbb{Z}}}
  |\tilde A_{N,\nu}|^{2}=
  C_{\nu}\|f\|_{ H^{1-\frac1p}_A}^2
\end{equation*}
and this concludes the proof.
\end{proof}

It remains to prove Proposition \ref{LRE}.

\begin{proof}[Proof of Proposition \ref{LRE}]
  We examine the two cases of estimate \eqref{stri-L}.
  We shall write
  \begin{equation*}
    Z=\int_0^\infty e^{ it\rho} J_{\nu}(r\rho)g(\rho)\, \rho d\rho,
    \qquad
    I=[R/2,R].
  \end{equation*}
 \medskip
 
$\bullet$ 
For $R\lesssim 1$ we must prove
\begin{equation}\label{eq:Rm1}
  \|Z\|_{L^p_tL^{\infty}_{dr}(I)}
  \lesssim R^{\nu-\frac{\epsilon} 2}
  \|g\|_{L^2_{\rho d\rho}(I)}.
\end{equation}
Taking  $0<\varepsilon\ll 1$, by the embedding 
$H^{\frac{1+2\varepsilon}2}(I)\hookrightarrow
L^\infty_{dr}(I)$ and interpolation, we have 
\begin{equation}\label{eq:interm}
  \|Z\| _{L^p_tL^{\infty}(I)}
  \lesssim
  \|Z\|_{L^p_t H^{\frac 12+\varepsilon}(I)}
  \lesssim
  \|Z\|^{\frac 12-\varepsilon}
  _{L^p_t L^{2}(I)}
  \|Z\|^{\frac 12+\varepsilon}
  _{L^p_t \dot H^{1}(I)}.
\end{equation}
The $L^p_t L^{2}(I)$ factor can be estimated 
by the Minkowski and the Hausdorff--Young inequality (in $t$)
as follows:
\begin{equation*}
  \|Z\|_{L^p_t L^{2}_{dr}(I)}\lesssim
  \|J_{\nu}(r \rho)g(\rho)\|
  _{L^{2}_{dr}(I)L^{p'}_{\rho}}.
\end{equation*}
Recall that $g(\rho)$ is supported in $[1,2]$ while
$r\in I=[R,2R]\subset[0,K]$ for some fixed $K$, since
$R \lesssim 1$. Using the standard inequaity
\begin{equation*}
  \textstyle
  |J_{\nu}(r)|\le C_{K}r^{\nu},
  \qquad
  r\in(0,K],
  \qquad
  \nu\in [-\frac 12,\frac 12]
\end{equation*}
we can continue the previous estimate as follows
\begin{equation*}
  \lesssim \|r^{\nu}\|_{L^{2}_{dr}([R,2R])}
  \|\rho^{\nu} g(\rho)\|_{L^{p'}_{\rho}([1,2])}
  \lesssim R^{\nu+\frac 12}\|g\|_{L^{2}_{\rho d \rho}},
\end{equation*}
and this proves
\begin{equation}\label{eq:Jest1}
  \|Z\|_{L^p_t L^{2}_{dr}(I)}\lesssim
  R^{\nu+\frac 12}\|g\|_{L^{2}_{\rho d \rho}}.
\end{equation}
To estimate the $L^{p}_{t}\dot H^{1}(I)$ factor we proceed
in a similar way and we arrive at
\begin{equation*}
  \|Z\|_{L^{p}_{t}\dot H^{1}(I)}\lesssim
  \|J'_{\nu}(r \rho)\rho g(\rho)\|
  _{L^{2}_{dr}(I)L^{p'}_\rho}.
\end{equation*}
Now using the obvious inequality
\begin{equation*}
  \textstyle
  |J'_{\nu}(r)|\le C_{K}r^{\nu-1},
  \qquad
  r\in(0,K],
  \qquad
  \nu\in [-\frac 12,\frac 12]
\end{equation*}
we obtain as before
\begin{equation}\label{eq:Jest2}
  \|Z\|_{L^{p}_{t}\dot H^{1}(I)}\lesssim
  R^{\nu-\frac 12}|g\|_{L^{2}_{\rho d \rho}}.
\end{equation}
Inserting these estimates in \eqref{eq:interm} we finally
obtain \eqref{eq:Rm1}.

\bigskip
$\bullet$
For $R\gg1$ we must prove
\begin{equation}\label{eq:Rp1}
  \|Z\|_{L^p_tL^{\infty}_{dr}(I)}
  \lesssim R^{\frac 1p-\frac 12}
  \|g\|_{L^2_{\rho d\rho}(I)}.
\end{equation}
By the Sobolev embedding
$W^{1,p}(I)\hookrightarrow L^{\infty}(I)$ for $p>2$,
we have
\begin{equation*}
  \|Z\|_{L^p_tL^{\infty}_{dr}(I)}\lesssim
  \|Z\|_{L^{p}_{t}L^{p}_{dr}(I)}+
  \|\partial_{r}Z\|_{L^{p}_{t}L^{p}_{dr}(I)}.
\end{equation*}
By the Minkowski and the Hausdorff-Young
inequality (in $t$), we get
\begin{equation*}
  \|Z\|_{L^p_tL^{p}_{dr}(I)}
  \lesssim
  \|J_{\nu}(r\rho)g(\rho)\|_{L^{p}_{dr}(I)L^{p'}_{\rho}}\lesssim
  \|r^{-\frac 12}\|_{L^{p}_{dr}(I)}
  \|g(\rho)\|_{L^{p'}}\lesssim
  R^{\frac 1p-\frac 12}\|g\|_{L^{2}_{\rho d \rho}},
\end{equation*}
where we used the standard estimate
\begin{equation*}
  \textstyle
  |J_{\nu}(r)|\le Cr^{-1/2},
  \qquad
  r\ge1,
  \quad
  \nu\in[-\frac 12,\frac 12].
\end{equation*}
In a similar way we have 
\begin{equation*}
  \|\partial_{r}Z\|_{L^{p}_{t}L^{p}_{dr}(I)}\lesssim
  \|J'_{\nu}(r\rho)\rho g(\rho)\|_{L^{p}_{dr}(I)L^{p'}_{\rho}}
  \lesssim
  \|r^{-\frac 12}\|_{L^{p}_{dr}(I)}
  \|\rho g(\rho)\|_{L^{p'}}\lesssim
  R^{\frac 1p-\frac 12}\|g\|_{L^{2}_{\rho d \rho}}
\end{equation*}
where we used the estimate
\begin{equation*}
  \textstyle
  |J'_{\nu}(r)|\le Cr^{-1/2},
  \qquad
  r\ge1,
  \quad
  \nu\in[-\frac 12,\frac 12]
\end{equation*}
and this concludes the proof of \eqref{eq:Rp1} and of
Proposition \ref{LRE}.
\end{proof}

To conclude the proof of Theorem \ref{thm-stri2}, it is now
sufficient to use complex interpolation between estimate
\eqref{stri-D3} (which will be proved in the next subsection) with $q<q(\alpha)$ sufficiently close to
$q(\alpha)$ and estimate \eqref{stri-D2} with $q$ sufficiently high.
We omit the straightforward details.

\subsection{Proof of Theorem \ref{thm-stri3}} 

The proof is very similar to the one of \eqref{stri-D2},
thus we shall just sketch it and highlight the differences. We must estimate the operator $T_{\nu}$ defined in \eqref{Tnuoperator}, \eqref{Knukernel} but instead of \eqref{est:2-2'}, \eqref{est:inf-2'}, our goal is now the weightless estimate
\begin{equation}\label{eq:weightless}
  \|T_{\nu}g(t,r)\|_{L^p_tL^q_{rdr}}
  \leq C\|f\|_{\dot H^{s}_A}.
\end{equation}
Proceeding as above (see \eqref{strifin}),
we arrive at the following bound:
\begin{equation}\label{strifin2}
\|T_{\nu}f(t,r)\|_{L^p_tL^q_{rdr}}^2\leq
 \sum_{R\in2^{\Z}}\left(\sum_{N\in 2^\Z}N^{2(1-\frac1q)-\frac1p}\left\| \hank\left[e^{it{\rho}} \varphi(\rho)\hank f(N\rho) \right]\right\|_{L^p_tL^q_{rdr}\,([NR,2NR])}\right)^2.
\end{equation}
Now we need the following analog of Proposition \ref{LRE}:

\begin{proposition}\label{LRE2}
  Let $p\geq2$, $2\leq q<\infty$, 
  $R>0$ a dyadic number,  and $\nu\in[-\frac12,1]$. 
  Then the following estimate holds:
  \begin{equation}\label{stri-L2}
    \|\mathcal{H}_{\nu}[e^{ it\rho}g(\rho)\big](r)\|
    _{L^p_tL^q_{rdr}([R/2,R])} 
    \lesssim
    \|g(\rho)\|_{L^2_{\rho d\rho}}
    \times
    \begin{cases}
      R^{\nu+\frac2q} &
      \ \text{if}\ R\lesssim1\\
      R^{\frac1q+\frac1{\min\{p, q\}}-\frac12} &
      \ \text{if}\ R\gg1
    \end{cases}
  \end{equation}
  provided $g(\rho)$ is supported in $[1,2]$.
\end{proposition}

\begin{proof}
  We write as before
  \begin{equation*}
    Z=\int_0^\infty e^{ it\rho} J_{\nu}(r\rho)g(\rho)xd\rho,
    \qquad
    I=[R/2,R].
  \end{equation*}

  \bigskip
  $\bullet$
  For $R\gg1$, the proof follows closely the proof 
  of Proposition \ref{LRE}. 
  If $q\geq p$, the same argument of Proposition \ref{LRE} 
  for large $R$ works also in this case; we omit the details.
  If $p>q$, by the Minkowski inequality, it suffices to prove
  the bound
  \begin{equation*}
    \|Z\|_{L^q_{dr}(I)L^p_t(\R)}\lesssim
    R^{\frac{1}q-\frac12}\|g(\rho)\|_{L^{p'}_{d\rho}}
  \end{equation*}
  \begin{equation*}
  \begin{split}
  \Big\|\int_0^\infty e^{it\rho}
  J_{\nu}(r\rho)g(\rho)\varphi(\rho)d\rho\Big\|
  _{L^q_{dr}([R/2,R])L^p_t(\R)}\lesssim
  R^{\frac{1}q-\frac12}\Big\|g(\rho)\varphi(\rho)\Big\|_{L^{p'}_{d\rho}(I)},
  \end{split}
  \end{equation*}
  which follows as \eqref{eq:Rp1}.

  $\bullet$
  For $R\lesssim 1$ we use the Sobolev embedding $\dot{H}^{\frac12-\frac1q}\hookrightarrow L^q$, $q \in[2,+\infty)$ to estimate 
  \begin{equation*}
    \|\mathcal{H}_{\nu}\big[e^{ it\rho}g(\rho)\big](r)\|
    _{L^p_tL^q_{rdr}(I)} \lesssim
    R^{1/q}\|Z\|_{L^p_tL^q_{rdr}(I)} \lesssim
    R^{1/q}\|Z\|_{L^{p}_{t}\dot H^{\frac 12-\frac 1q}(I)}
  \end{equation*}
  \begin{equation*}
    \lesssim
    R^{1/q}
    \|Z\|^{\frac 12+\frac 1q}
    _{L^{p}_{t}L^{2}(I)}
    \|Z\|^{\frac 12-\frac 1q}
    _{L^{p}_{t}\dot H^{1}(I)}
    \lesssim
    R^{\nu+\frac 2q}\|g(\rho)\|_{L^{2}_{\rho d\rho}(I)}
  \end{equation*}
  where the last inequalityis proved as before
  relying on \eqref{eq:Jest1}, \eqref{eq:Jest2}
\end{proof}

Using this result, we can continue the estimate as follows:
\begin{equation*} 
  \eqref{strifin2}\leq {\sum_{R\in2^{\Z}}\left(\sum_{N\in 2^\Z}N^{1-\frac1p-\frac2q}\left\| \varphi(\rho/N)\hank f(\rho) \right\|_{L^2_{\rho d\rho}}
\right)^2}
\end{equation*}
where
\begin{equation}
  Q(NR)=
  \begin{cases}
    (NR)^{\nu+\frac2q} &
    \ \text{if}\ 
    NR\lesssim 1\\ 
    (NR)^{\frac1q+\frac1{\min\{p,q\}}-\frac12} &
    \ \text{if}\ NR\gg1.
\end{cases}
\end{equation}
Now, in order to make the Schur test argument work as above, we need the conditions
\begin{equation}\label{condfinal}
{\frac1q+\frac1{\min\{p,q\}}}-\frac12<0,\qquad \nu+\frac2q>0
\end{equation}
which are equivalent to $p>2$ and $q>0$ if $\nu>0$ while $q<-\frac2\nu$ if $\nu\in [-1/2,0)$, which is ensured by the assumption $q<q(\alpha)$. We remark that there is no $q$ satisfying the above condition when $\alpha=1/2$. We omit the rest of the details (we mention that this proof now closely follows the one of Theorem 1.2 in \cite{cacserzha}). This concludes the proof of Theorem \ref{thm-stri3} provided that either $p\leq q$ or $p>q>4$. 

To conclude with, we now need to consider the additional range $2\leq q\leq 4$ when $q<p$. We first prove the Strichartz estimates when $2\leq q<4$ and $\frac2p+\frac1q\leq \frac12$ by using Keel-Tao's argument. Then, by interpolation with \eqref{eq:weightless}, we extend $\frac2p+\frac1q\leq \frac12$ to $\frac1p+\frac1q < \frac12$.

We shall need another variant of Keel-Tao's argument
(compare with Proposition \ref{prop:semi} above):

\begin{proposition}\label{prop:semi-1}
Let $(X,\mathcal{M},\mu)$ be a $\sigma$-finite measured space and
$U: \mathbb{R}\rightarrow B(L^2(X,\mathcal{M},\mu))$ be a weakly
measurable map satisfying, for some constants $C$, $\kappa\geq0$,
$\sigma, h>0$,
\begin{equation}\label{md-1}
\begin{split}
\|U(t)\|_{L^2\rightarrow L^2}&\leq C,\quad t\in \mathbb{R},\\
\|U(t)U(s)^*f\|_{L^{q_0}}&\leq
Ch^{-\kappa(1-\frac2{q_0})}(h+|t-s|)^{-\sigma(1-\frac2{q_0})}\|f\|_{L^{q_0'}},\quad 2\leq q_0\leq +\infty.
\end{split}
\end{equation}
Then for every pair $q,p\in[1,\infty]$ such that $(p, q,\sigma)\neq
(2,\infty,1)$ and
\begin{equation*}
\frac{1}{p}+\frac{\sigma}{q}\leq\frac\sigma 2,\quad 2\leq q\leq q_0, p\geq 2.
\end{equation*}
there exists a constant $\tilde{C}$ depending only on $C$, $\sigma$,
$q$ and $p$ such that
\begin{equation*}
\Big(\int_{\mathbb{R}}\|U(t) u_0\|_{L^q}^p dt\Big)^{\frac1p}\leq \tilde{C}
\Lambda(h)\|u_0\|_{L^2}
\end{equation*}
where $\Lambda(h)=h^{-(\kappa+\sigma)(\frac12-\frac1q)+\frac1p}$.
\end{proposition}

\begin{proof} 
By interpolation of the bilinear form of \eqref{md-1}, we have
\begin{equation*}
\begin{split}
\langle U(s)^*f(s), U(t)^*g(t) \rangle&\leq
Ch^{-\kappa(1-\frac2q)}(h+|t-s|)^{-\sigma(1-\frac2q)}\|f\|_{L^{q'}}\|g\|_{L^{q'}},\quad 2\leq q\leq q_0.
\end{split}
\end{equation*}
Therefore we see by H\"older's and Young's inequalities for
$\frac1p+\frac\sigma q<\frac\sigma 2$
\begin{equation*}
\begin{split}
\Big|\iint\langle U(s)^*f(s),& U(t)^*g(t) \rangle
dsdt\Big|\\&\lesssim
h^{-\kappa(1-\frac2q)}\iint(h+|t-s|)^{-\sigma(1-\frac2q)}\|f(t)\|_{L^{q'}}\|g(s)\|_{L^{q'}}dtds\\&
\lesssim
h^{-\kappa(1-\frac2q)}h^{-\sigma(1-\frac2q)+\frac2p}\|f\|_{L^{p'}_tL^{q'}}\|g\|_{L^{p'}_tL^{q'}}
\end{split}
\end{equation*}
and this concludes the proof.
\end{proof}

Now, in order to remove the frequency localization, we need the Littlewood-Paley theory for the radial component. Square function estimates are a standard consequence of Gaussian upper bounds for the corresponding heat kernel. Unfortunately, we cannot prove the Gaussian upper bounds for $e^{-t \mathcal{D}_{A,\gamma}^2} P_0$ due to the presence of Bessel functions of negative order (and we actually suspect that such bounds fail). Therefore we must prove the multiplier estimates in $L^{p}$ for the Littlewood-Paley operator directly. We shall not prove the estimates for all $1<p<\infty$, but only for $q'(\alpha)<p<q(\alpha)$: this will be enough for our purpose of proving Strichartz estimates with $2\leq q<4$ since $(\frac43,4)\subset (q'(\alpha),q(\alpha))$.
\begin{proposition}[Square function inequality]
  \label{prop:squarefun0} 
  Let $\varphi\in C_c^\infty$, with support in $[1/2,1]$
  and such that $0\leq\varphi\leq 1$ and
  \begin{equation}\label{LP-dp}
    \sum_{j\in\Z}\varphi(2^{-j}\lambda)=1
    \qquad\text{where}\qquad 
    \varphi_j(\lambda):=\varphi(2^{-j}\lambda).
  \end{equation}
  Then for all $q'(\alpha)<p<q(\alpha)$
  there exist constants $c_p$ and $C_p$ depending on $p$ such that
  \begin{equation}\label{square1-0}
  c_p\|P_0 f\|_{[L^p(\R^2)]^2}\leq
  \Big\|\Big(\sum_{j\in\Z}|\varphi_j(|\mathcal{D}_{A,\gamma}|)P_0f|^2\Big)^{\frac12}\Big\|_{[L^p(\R^2)]^2}\leq
  C_p\|P_0f\|_{[L^p(\R^2)]^2}.
  \end{equation}
\end{proposition}

\begin{proof}  
In order to prove the square function estimates \eqref{square1-0}, by using the Rademacher functions and 
the argument of Stein \cite[Appendix D]{Stein}, it suffices to show 
(similarly to \eqref{equ:F1} and \eqref{equ:F2})
that the Littlewood-Paley operator 
$\varphi(|\mathcal{D}_{A,\gamma}|)P_0$
is bounded on $L^p(\R^2)$.
We may write explictly
\begin{equation*}
  \varphi(|\mathcal{D}_{A,\gamma}|)P_0 f=
  \int_{0}^{\infty}K(r_{1},r_{2})f(r_{2})r_{2}dr_{2}
\end{equation*}
where the kernel $K$ is defined as follows:
\begin{itemize}
\item if $\alpha\in(0,\frac12]$,
 \begin{equation}\label{def:LP1}
K(r_1,r_2)=
\left(\begin{matrix} \mathrm{LP}_{-\alpha}& 0 \\0 & 
 e^{i(\theta_1-\theta_2)} \mathrm{LP}_{1-\alpha}\end{matrix}\right),
\end{equation} 
\item if $\alpha\in(\frac12,1)$,
 \begin{equation}\label{def:LP2}
K(r_1,r_2)=\
\left(\begin{matrix} \mathrm{LP}_{\alpha}& 0 \\0 & 
 e^{i(\theta_1-\theta_2)} \mathrm{LP}_{\alpha-1}\end{matrix}\right),
\end{equation} 
\end{itemize}
and
\begin{equation}\label{def:LP}
\begin{split}
\mathrm{LP}_{\nu}=\int_0^\infty \varphi(\rho)J_{\nu}(r_1\rho)J_{\nu}(r_2\rho)\,\rho d\rho, \quad \nu\in[-1/2, 0)\cup(0, 1/2].
 \end{split}
\end{equation} 
Compare with the proof in \cite{Zhang} where the multiplier estimates on $L^p$ with $1<p<\infty$ are deduced from the Gaussian upper bounds for the heat kernel. The same argument is used in the Appendix of the present paper to prove the square function estimates for operator $L$. As mentioned above, due to  the presence of Bessel functions of negative order, we shall obtain $L^{p}$ boundedness only in the range $q'(\alpha)<p<q(\alpha)$. In the following argument, since the other cases can be proved simlarly, we only consider the case 
$\alpha\in (0,1/2]$ and the kernel $\mathrm{LP}_{-\alpha}$.
We write $\nu=-\alpha<0$.

We analyze directly the kernel as in Lemma \ref{lem:est-qq'}.
Define $\chi\in \mathcal{C}_c^\infty ([0,+\infty)$ as 
\begin{equation}
\chi(r)=
\begin{cases}1,\quad r\in [0, \frac12],\\
0, \quad r\in [1,+\infty)
\end{cases}
\end{equation}
and write $\chi^c=1-\chi$. We decompose the kernel $  K(r_{1},r_{2})$ into four terms as follows:
  \begin{equation}
\begin{split}
 K(r_{1},r_{2})=&\chi(r_1)K(r_{1},r_{2})\chi(r_2)+\chi^c(r_1)K(r_{1},r_{2})\chi(r_2)\\
&+\chi(r_1)K(r_{1},r_{2})\chi^c(r_2)+\chi^c(r_1)K(r_{1},r_{2})\chi^c(r_2).
 \end{split}
\end{equation}
This yields a corresponding decomposition for the operator $\mathrm{LP}_{\nu}=\mathrm{LP}^1_{\nu}+\mathrm{LP}^2_{\nu}+\mathrm{LP}^3_{\nu}+\mathrm{LP}^4_{\nu}$. We thus estimate separately the norms $\|\mathrm{LP}^j_{\nu}g\|_{L^p_{r_1 dr_1}}$ for $j=1,2,3,4$.

For $\mathrm{LP}^1_\nu$, we have that 
  \begin{equation}
\begin{split}
|\chi(r_1)K_{\nu}(r_{1},r_{2})\chi(r_2)|\leq \chi(r_1)r_1^{\nu}\chi(r_2)r_2^{\nu};
 \end{split}
\end{equation}
therefore, as long as $1<p, p'<-\frac2\nu$ 
we can write
  \begin{equation}\label{est:q-q'1}
  \|\mathrm{LP}^1_{\nu}g\|_{L^p_{r_1 dr_1}}\le
    c_{\nu}\Big(\int_0^1 r^{\nu p} r dr\Big)^{1/p}\Big(\int_0^1 r^{\nu p'} r dr\Big)^{1/p'}\|g\|_{L^{p}_{r_2 dr_2}}\le
    c_{\nu}\|g\|_{L^{p}_{r_2 dr_2}}.
  \end{equation}
  
The terms $\mathrm{LP}^2_{\nu}$ and $\mathrm{LP}^3_{\nu}$ are similar to the second and third terms of Lemma \ref{lem:est-qq'} with $t=0$, so we can use the same argument for $|t|\leq 1$ of Lemma \ref{lem:est-qq'}.
Focusing on $\mathrm{LP}^{3}_{\nu}$,
we need to estimate the two integrals
  \begin{equation*}
    I_{\pm}=\int_{0}^{\infty}
    \rho \phi(\rho)J_{\nu}(r_{1}\rho)(r_{2}\rho)^{-1/2}
    e^{\pm i r_{2}\rho}a_{\pm}(r_{2}\rho)d \rho.
  \end{equation*}
Thanks to \eqref{eq:bess2}
and 
\begin{equation}
|J_{\nu}''(x)|=\big|J_{\nu-1}'(x)+\frac{\nu J_{\nu}(x)}{x^2}-\frac{\nu J'_{\nu}(x)}x\big|\leq C_\nu |x|^{\nu-2}
\quad \text{if}\, |x|\leq 1,
\end{equation}
and then by integration by parts twice we obtain
  \begin{equation*}
   | I_{\pm}|=\Big|\frac{ir_{2}^{-1/2}}{\pm r^2_{2}}
    \int_{0}^{\infty}
    (\rho^{1/2}\phi(\rho)J_{\nu}(r_{1}\rho)a_{\pm}(r_{2}\rho))'
    e^{\pm i \rho r_{2}}d\rho\Big|\leq 
    Cr_1^{\nu}r_2^{-\frac52} .
  \end{equation*}
 Hence, as long as $1<p<-\frac2\nu$,  we have
      \begin{equation}
  \|\mathrm{LP}^3_{\nu}g\|_{L^p_{r_1 dr_1}}\le
    c_{\nu}\Big(\int_0^1 r_1^{\nu p} r_1 dr_1\Big)^{1/p}\Big(\int_{\frac12}^{+\infty} r_2^{-\frac{5p'}2} r_2 dr_2\Big)^{1/p'}\|g\|_{L^{p}_{r_2 dr_2}}\le
    c_{\nu}\|g\|_{L^{p}_{r_2 dr_2}}.
\end{equation}
The computation for the term $\mathrm{LP}^2_{\nu}$
is similar and gives the condition
$1<p'<-\frac2\nu$. 

For the term $\mathrm{LP}^4_{\nu}$, 
we need to estimate the two integrals
\begin{equation*}
  I_{\pm}(r_1,r_2)=\int_{0}^{\infty}
  \rho \phi(\rho)(r_{1}\rho)^{-1/2}(r_{2}\rho)^{-1/2}
  e^{ i(r_1\pm r_{2})\rho}a_{\pm}(r_{1}\rho) a_{\pm}(r_{2}\rho)d \rho.
\end{equation*}
Integrating by parts, for $r_1,r_2\geq 1/2$ and any $N\geq 0$, we obtain
  \begin{equation*}
   |  I_{\pm}(r_1,r_2)|\leq C_N (r_1r_2)^{-\frac12} (1+|r_1-r_2|)^{-N}.
  \end{equation*}
 To prove $\mathrm{LP}^4_{\nu}$ is bounded on $L^p_{rdr}$, by Schur test lemma, it suffices to show
   \begin{equation*}
 \sup_{r_2\in [\frac12,+\infty)} 
 \int_{\frac12}^\infty 
 |  I_{\pm}(r_1,r_2)| r_1 dr_1\leq C,\quad  \sup_{r_1\in [\frac12,+\infty)} \int_{\frac12}^\infty |  I_{\pm}(r_1,r_2)| r_2 dr_2\leq C.
\end{equation*}
By symmetry, we consider only the first estimate.
Since $r_{1}^{1/2}\lesssim |r_{1}-r_{2}|^{1/2}+r_{2}^{1/2}$,
we can write
\begin{equation*}
  |I_{\pm}(r_1,r_2)|r_1\le
  Cr_{2}^{-1/2}(1+|r_{1}-r_{2}|)^{-N+1/2}+
  C(1+|r_{1}-r_{2}|)^{-N}
\end{equation*}
and recalling that $r_{2}\ge1/2$
\begin{equation*}
  \int_{\frac12}^\infty 
  |  I_{\pm}(r_1,r_2)| r_1 dr_1\le
  C\int_{\mathbb{R}}(1+|r_{1}-r_{2}|)^{-N'}dr_{1}\le C
\end{equation*}
  Therefore, $\mathrm{LP}^4_{\nu}$ is bounded on $L^p_{rdr}$ for $1<p<\infty$. Summing up, we have proved that the Littlewood-Paley operator is bounded on $L^p_{rdr}$ for $q'(\alpha)=\frac2{2-\alpha}<p<\frac 2{\alpha}=q(\alpha)$
  when $\alpha\in (0,1/2]$.
\end{proof}

As for the proof of Theorem \ref{thm-stri1}, we now use Proposition \ref{prop:semi-1} and the square function estimates given in Proposition \ref{prop:squarefun0} to deduce
the Strichartz estimates of Theorem \ref{thm-stri3} when $2\leq q<4$ and $\frac2p+\frac1q\leq \frac12$.
The remaining estimates can be obtained by interpolating with the estimates for the range $\frac1p+\frac1q< \frac12$ with $q>4$ obtained in the first step. This concludes the proof of
Theorem \ref{thm-stri3}

\subsection{Sharpness on the restriction $q<q(\alpha)$ in \eqref{stri} }

Corollary \ref{cor-stri} is an obvious consequence of the previous
estimates. In this subsection, we construct a counterexample to show that the restriction $q<q(\alpha)$ is in fact necessary. 

\begin{proposition}[Counterexample] 
  Let $(p,q)$ be as in \eqref{pqrange1} and let $q\geq q(\alpha)$.
  Then the Strichartz estimates \eqref{stri} fail, in the sense that there exist an initial condition $f\in D(\mathcal{D}_{A,\gamma})$ such that
  \begin{equation}\label{counter}
  \|e^{it\mathcal{D}_{A,\gamma}}f\|_{L^p(\R;L^q(\R^2))}=\infty
  \quad\text{for any}\quad q\geq q(\alpha).
  \end{equation}
\end{proposition}

\begin{proof} 
  The counterexample is inspired by \cite{ZZ}.  Without loss of generality, we assume $\alpha\in (0, \frac12]$ and choose initial data of the form $f=(\mathcal{H}_{-\alpha}\chi, 0)$, where $\chi\in\CC_c^\infty([1,2])$ takes value in $[0,1]$. Obviously,  $f\in \big[\dot H^s\big]^2$ due to the compact support of  $\chi$ and to the unitarity of the Hankel transform on 
  $L^2(rdr)\simeq L^{2}(\mathbb{R}^{2})$. 
  We prove that \eqref{counter} holds for this choice of $f$. Recalling \eqref{equ:F-E}, we must prove that the quantity
  \begin{equation*}
    Z=\int_0^\infty 
    J_{-\alpha}(r\rho)e^{it\rho}\chi(\rho)\rho d\rho
  \end{equation*}
  satisfies
  \begin{equation}\label{est:aim}
    \|Z\|_{L^p(\R;L^q(\R^2))} =\infty,
    \qquad q\geq q(\alpha).
  \end{equation}
  From the series expansion of $J_{-\alpha}(r)$ at 0,
  we know that
  \begin{equation}\label{Bessel1}
    J_{-\alpha}(r)=C_{\alpha}r^{-\alpha}+S_{\alpha}(r)
  \end{equation}
  where
  \begin{equation}\label{Bessel3}
    |S_{\alpha}(r)|\leq C_{\alpha} r^{1-\alpha},
    \qquad
    r\in(0,2].
  \end{equation}
  Then for any $0<\epsilon<1$ we can estimate
  \begin{equation*}
    \|Z\|_{L^p(\R;L^q(\R^2))}\ge
    \|Z\|_{L^p_{t}([0,1/2];L^q_{rdr}[\epsilon,1]}\ge
    C_{\alpha}\|P\|_{L^p_{t}([0,1/4];L^q_{rdr}[\epsilon,1]}
    -
    \|Q\|_{L^p_{t}([0,1/2];L^q_{rdr}[\epsilon,1]}
  \end{equation*}
  where
  \begin{equation*}
    P= \int_{0}^{\infty}(r \rho)^{-\alpha}e^{it \rho}\chi(\rho)
      \rho d \rho,
    \qquad
    Q=\int_{0}^{\infty}S_{\alpha}(\rho)e^{it \rho}\chi(\rho)
      \rho d \rho.
  \end{equation*}
  Now, on one hand by \eqref{Bessel3} we have
  \begin{equation*}
    \|Q\|_{L^p_{t}([0,1/2];L^q_{rdr}[\epsilon,1]}
    \lesssim
    \left\|
      \int_{0}^{\infty}(r \rho)^{1-\alpha}\chi(\rho)\rho d \rho
    \right\|_{L^p_{t}([0,1/2];L^q_{rdr}[\epsilon,1]}
    \lesssim
    1-\epsilon^{1-\alpha+\frac 2q}\le1.
  \end{equation*}
  On the other hand, we have
  \begin{equation*}
    \|P\|_{L^p([0,1/4];L^q_{rdr}[\epsilon,1])}=
    \left(\int_0^{\frac14}\left( \int_{\epsilon}^1 \left|\int_0^\infty (r\rho)^{-\alpha}e^{
    it\rho}\chi(\rho)\rho d\rho\right|^q r dr\right)^{p/q}dt\right)^{1/p}
  \end{equation*}
  and by the assumption $q\geq q(\alpha)=\frac2\alpha$
  \begin{equation*}
    \gtrsim
    \left(\int_0^{\frac14}
    \left|\int_0^\infty \rho^{-\alpha}
    e^{it\rho}\chi(\rho)\rho d\rho\right|^{p}dt\right)^{1/p}
    \times 
    \begin{cases}
      \epsilon^{-\alpha+\frac 2q} &
      \text{if}\quad q\alpha>2\\
      \ln\epsilon &
      \text{if}\quad q\alpha=2
    \end{cases}
    \gtrsim C
    \begin{cases}
      \epsilon^{-\alpha+\frac 2q} &
      \text{if}\quad q\alpha>2\\
      \ln\epsilon &
      \text{if}\quad q\alpha=2.
    \end{cases}
  \end{equation*}
  In the last inequality, we have used the fact that 
  $\cos(\rho t)\geq 1/2$ for $t\in [0, 1/4]$ and 
  $\rho\in [1,2]$, so that
  \begin{equation}
  \begin{split}
  \left|\int_0^\infty \rho^{-\alpha}e^{
  it\rho}\chi(\rho)\rho d\rho\right|\geq \frac12\int_0^\infty \rho^{-\alpha}\chi(\rho)\rho d\rho\geq c.
  \end{split}
  \end{equation}
  Summing up, if $q>\frac2\alpha$ we obtain 
  \begin{equation*}
    \Big\|\int_0^\infty 
      J_{-\alpha}(r\rho)e^{it\rho}\chi(\rho)\rho d\rho\Big\|
      _{L^p(\R;L^q(\R^2))} \geq
      c\epsilon^{-\alpha+\frac 2q} -C 
      \to +\infty \qquad \text{as}\quad \epsilon\to 0,
  \end{equation*}
  while if $q\alpha=2$ we have
  \begin{equation*}
    \Big\|\int_0^\infty 
      J_{-\alpha}(r\rho)e^{it\rho}\chi(\rho)\rho d\rho
    \Big\|_{L^p(\R;L^q(\R^2))}
    \ge
    c\ln\epsilon-C\to +\infty \qquad \text{as}\quad \epsilon\to 0,
  \end{equation*}
  and this implies \eqref{est:aim}.
  
 \end{proof}

\appendix
\section{Proof of estimate \eqref{est:dis1-1}.}

We include a proof of estimate \eqref{est:dis1-1} for the sake of completeness. The proof requires a number of technical tools, such as Bernstein and Littlewood--Paley inequalities, which will be established via heat kernel estimates. For the convenience of the Reader we divide the argument in several subsections.

\subsection{The modified Schr\"odinger propagator} 

Recall the AB Schr\"{o}dinger operators $H$ and 
$L$ constructed at the beginning of Section \ref{sec:proofdisp}.
We first consider the propagator $e^{itL}$, which we
represent in integral form as
\begin{equation*}
  e^{it L}f=
  \int_{0}^\infty \int_{0}^{2\pi}
  {\tilde{\bf K}}_{S}(t,r,\theta,r_2,\theta_2) 
  f(r_2,\theta_2)  d\theta_2 \;r_2dr_2,
\end{equation*}
where we use polar coordinates
$x=r_1(\cos\theta_1,\sin\theta_1)$ and
$y=r_2(\cos\theta_2,\sin\theta_2)$.
We compute ${\tilde{\bf K}}_{S}$ explicitly:

\begin{proposition}\label{prop:kerenel-S} 
  Let $\alpha\in(0,1)$ and $L$ the operator defined in
  \eqref{eq:defL}. Then the integral kernel of the
  propagator $e^{itL}$ can be written as
\begin{equation}\label{bf-KS1}
  e^{itL}(x,y)=
  \tilde{\bf K}_S(t,r_1,\theta_1,r_2,\theta_2)= 
  \begin{pmatrix} 
    G_{\alpha}+D_{\alpha}& 0 \\
    0 & G_{\alpha}+D_{\alpha}\end{pmatrix}
\end{equation}
where, writing
$x=r_1(\cos\theta_1,\sin\theta_1)$ and
$y=r_2(\cos\theta_2,\sin\theta_2)$,
\begin{equation}\label{G-term}
  G_\alpha(t;r_1,\theta_1,r_2,\theta_2) =
  \frac{e^{-\frac{|x-y|^2}{4it}} }{it}
  A_{\alpha}(\theta_1,\theta_2),
\end{equation}
\begin{equation}\label{D-term}
  D_{\alpha}(t;r_1,\theta_1,r_2,\theta_2)
  =-\frac{e^{-\frac{r_1^2+r_2^2}{4it}} }{4\pi^2it}
  \int_0^\infty e^{-\frac{r_1r_2}{2it}\cosh s} 
  \Big(B_{\alpha}(s)+
  C_{\alpha}(s,\theta_1-\theta_2+\pi)\Big) ds,
\end{equation}
and
\begin{align} 
        \label{equ:A} 
  & A_{\alpha}(\theta_1,\theta_2)= 
  \frac{e^{i\alpha(\theta_1-\theta_2)}}{4\pi^2}
  \big(\mathbbm{1}_{[0,\pi]}(|\theta_1-\theta_2|)
  +e^{-i2\pi\alpha}\mathbbm{1}_{[\pi,2\pi]}
  (\theta_1-\theta_2)\\&\qquad\qquad\qquad\qquad\qquad\qquad+e^{i2\pi\alpha}\mathbbm{1}_{(-2\pi,-\pi]}
  (\theta_1-\theta_2)\big),\nonumber
\\ 
        \label{equ:BC} 
  & B_{\alpha}(s)= 
   \sin(|\alpha|\pi)e^{-|\alpha| s},
\\
    \notag
  & C_{\alpha}(s,\theta)=
  \sin(\alpha\pi)
  \frac{(e^{-s}-\cos(\theta))
      \sinh(\alpha s)-i\sin(\theta)\cosh(\alpha s)}
    {\cosh(s)-\cos(\theta)}.
\end{align}
\end{proposition}

\begin{proof} 
  The proof is in the spirit of Cheeger--Taylor 
  \cite{CT1,CT2} and is adapted from \cite{FZZ,GYZZ}; 
  we include it here for the sake of completeness. 
  We expand $\tilde{\bf K}_S$ in spherical harmonics:
  \begin{equation}\label{equ:ktxyschr}
    \tilde{\bf K}_S(t; r_1,\theta_1,r_2,\theta_2)=
    \frac1{2\pi} 
    \sum_{k\in\Z} \Phi_k(\theta_{1})
    \tilde{\bf K}_{S,k}(t; r_1,r_2)
    \overline{\Phi_k(\theta_{2})}
  \end{equation}
  where $\Phi_{k}$ is given by \eqref{eq:defPhik}.
  With the notations of Section \ref{sec:proofdisp},
  in particular \eqref{eq:Lka}, we see that
  $\tilde{\bf K}_{S,k}$ is the kernel of
  \begin{equation*}
    e^{it L_{k-\alpha}}=
    \widetilde{\mathcal{P}_{k}}^{-1}
      e^{it \rho^{2}}
    \widetilde{\mathcal{P}_{k}}
  \end{equation*}
  that is to say
  \begin{equation*}
    \tilde{\bf K}_{S,k}=
    \begin{pmatrix}
      S_{|k-\alpha|} & 0 \\
      0 &  S_{|k-\alpha+1|}
    \end{pmatrix}
  \end{equation*}
  where $S_{\nu}$ is defined by
  \begin{equation}\label{S-nu}
    S_{\nu}(t,r_1,r_2)=
    \int_0^\infty e^{it\rho^2}
    J_{\nu}(r_1\rho)J_{\nu}(r_2\rho) \,\rho d\rho.
  \end{equation}
  We now use the Weber identity and analytic continuation 
  like in \cite[Proposition 8.7, and Page 161, (8.88)]{Taylor}
  to write 
  \begin{equation}\label{equ:knukdef12sch}
    S_{\nu}(-t,r_{1},r_{2})=
    \lim_{\epsilon\searrow0}\int_0^\infty 
    e^{-(\epsilon+it)\rho^2}J_{\nu}(r_1\rho)
    J_{\nu}(r_2\rho) \,\rho d\rho=
    \lim_{\epsilon\searrow0}
    \frac{e^{-\frac{r_1^2+r_2^2}{4(\epsilon+it)}}}
      {2(\epsilon+it)} 
      I_\nu\Big(\frac{r_1r_2}{2(\epsilon+it)}\Big)
  \end{equation}
  where $I_\nu$ is the modified Bessel function,
  which can be defined e.g.~via the integral representation
  \cite{Watson}
  \begin{equation}\label{m-bessel}
    I_\nu(z)=
    \frac1{\pi}\int_0^\pi e^{z\cos \tau} 
    \cos(\nu \tau) d\tau-\frac{\sin(\nu\pi)}{\pi}
    \int_0^\infty e^{-z\cosh \tau} e^{-\tau\nu} d\tau.
  \end{equation}
  To compute \eqref{bf-KS1}, we must prove that
  \begin{equation*}
    \frac1{2\pi} \sum_{k\in\Z}\Phi_{k}(\theta_{1})
    \begin{pmatrix} 
      S_{|k-\alpha|}(t,r_1, r_2)& 0 \\
      0 & S_{|k-\alpha+1|} (t,r_1,r_2)
    \end{pmatrix}
    \overline{\Phi_{k}(\theta_{2})}=
    \begin{pmatrix} 
      G_{\alpha}+D_{\alpha}& 0 
      \\0 & G_{\alpha}+D_{\alpha}
    \end{pmatrix}
  \end{equation*}
  that is,
  \begin{equation*}
    \lim_{\epsilon\searrow0}
    \frac{e^{-\frac{r_1^2+r_2^2}{4(\epsilon+it)}}}{4\pi(\epsilon+it)} 
    \sum_{k\in\Z} e^{ik(\theta_1-\theta_2)}  
    I_{\nu}\Big(\frac{r_1r_2}{2(\epsilon+it)}\Big)=
    G_{\alpha}+D_{\alpha},\qquad 
    \nu=|k-\alpha|.
  \end{equation*}
  Applying the following Lemma, we conclude the proof
  of the Proposition.
\end{proof}

\begin{lemma}\label{lem:sumk}
  Let $\alpha\in(0,1)$, $S_{\nu}$ as in \eqref{S-nu},
  $G_{\alpha},D_{\alpha}$ as in \eqref{G-term}, \eqref{D-term}.
  Define
  \begin{equation*}
    S(-t,r_1,r_2,\alpha):=
    \sum_{k\in\Z} e^{ik(\theta_1-\theta_2)}  
    S_{|k-\alpha|}(-t,r_1,r_2).
  \end{equation*}
  Then we have
  \begin{equation}\label{equ:Skernel2}
    S(-t,r_1,r_2,\alpha)=
    G_\alpha(t;r_1,\theta_1,r_2,\theta_2)+
    D_{\alpha}(t;r_1,\theta_1,r_2,\theta_2).
  \end{equation}
\end{lemma}

\begin{proof}
  The two terms $G_{\alpha}$, $D_{\alpha}$ correspond to the
  two terms in the integral representation \eqref{m-bessel}.
  We deal with $G_{\alpha}$ first.
  Setting $z=\frac{r_1r_2}{2(\epsilon+it)}$, we must
  compute
  \begin{equation*}
    \frac{1}{2\pi^{2}}
    \sum_{k\in\Z} e^{ik(\theta_1-\theta_2)}
    \int_0^\pi e^{z\cos s} \cos(|k-\alpha| s) ds.
  \end{equation*}
  By the Poisson summation formula we have
  \begin{equation*}
  \begin{split}
  &\sum_{ k\in\Z} \cos(|k-\alpha|s) e^{ik(\theta_1-\theta_2)}=\sum_{k\in\Z} \frac{e^{i(k-\alpha)s}+e^{-i(k-\alpha)s}}2 e^{ik(\theta_1-\theta_2)}\\
  &=\frac12\sum_{j\in\Z}\big(e^{-i\alpha s}\delta(\theta_1-\theta_2+s+2\pi j)+e^{i\alpha s}\delta(\theta_1-\theta_2-s+2\pi j)\big).
  \end{split}
  \end{equation*} 
  Then the previous sum can be written
  \begin{equation*}
    =
    \frac{1}{4\pi^{2}}
    \sum_{j\in \mathbb{Z}}
    \int_{0}^{\pi}e^{z\cos s}
    \big[e^{-is\alpha}\delta(\theta_1-\theta_2+2j\pi+s)
       +e^{is\alpha}\delta(\theta_1-\theta_2+2j\pi-s)\big]ds
  \end{equation*}
  \begin{equation*}
    =\frac1{2\pi^2}
    \sum_{\{j\in\Z: 0\leq |\theta_1-\theta_2+2j\pi|\leq \pi\}} 
    e^{z\cos(\theta_1-\theta_2+2j\pi)}
    e^{i(\theta_1-\theta_2+2j\pi)\alpha}
  \end{equation*}
  \begin{equation*}
    =\frac1{2\pi^2}
    \begin{cases}
      e^{z\cos(\theta_1-\theta_2)} 
      e^{i(\theta_1-\theta_2)\alpha} &
      \text{if}\quad |\theta_1-\theta_2|< \pi,\\
      e^{z\cos(\theta_1-\theta_2)}
      e^{i(\theta_1-\theta_2-2\pi)\alpha}&
      \text{if}\quad \pi<\theta_1-\theta_2<2\pi,\\
      e^{z\cos(\theta_1-\theta_2)}
      e^{i(\theta_1-\theta_2+2\pi)\alpha}&
      \text{if}\quad -2\pi<\theta_1-\theta_2<-\pi,\\
      e^{-z}\big(e^{i\pi\alpha}+e^{-i\pi\alpha}\big)&
      \text{if}\quad |\theta_1-\theta_2|=\pi,
  \end{cases}
  \end{equation*}
  which gives $G_{\alpha}$.
  To obtain the term $D_\alpha$, we consider the second term 
  arising from \eqref{m-bessel}
  \begin{equation}\label{aim-D}
    \sum_{k\in\Z} 
    e^{ik(\theta_1-\theta_2)}
    \frac{\sin(|k-\alpha|\pi)}{\pi}
    \int_0^\infty e^{-z\cosh s} e^{-s\tilde{\nu}_k} ds.
  \end{equation}
  We note that
  \begin{equation*}
    \sin(\pi|k-\alpha|)=-\begin{cases}
    \cos(k\pi)\sin(\alpha\pi),
     &k\geq1;\\
    \sin(\alpha\pi),
    &k=0;\\
    -\cos(k\pi)\sin(\alpha\pi),
    &k\leq -1.
  \end{cases}
  \end{equation*}
  Hence we can write
  \begin{equation}\label{}
  \begin{split}
  &\sum_{k\in\Z}\sin(\pi|k-\alpha|)e^{-s|k-\alpha|}e^{ik(\theta_1-\theta_2)}\\
  =&-\sin(\alpha\pi)\sum_{k\geq 1} \frac{e^{ik\pi}+e^{-ik\pi}}2 e^{-s(k-\alpha)}e^{ik(\theta_1-\theta_2)}-\sin(\alpha\pi)e^{-\alpha s}\\&+\sin(\alpha\pi)\sum_{k\leq-1} \frac{e^{ik\pi}+e^{-ik\pi}}2 e^{s(k-\alpha)}e^{ik(\theta_1-\theta_2)}\\
  =&-\sin(\alpha\pi)e^{-\alpha s}-\frac{\sin(\alpha\pi)}2\Big(e^{s\alpha}\sum_{k\geq1} e^{-ks} \big(e^{ik(\theta_1-\theta_2+\pi)} +e^{ik(\theta_1-\theta_2-\pi)}\big)\\&\qquad\qquad\qquad- e^{-s\alpha} \sum_{k\geq 1} e^{-ks} \big(e^{ik(\theta_1-\theta_2+\pi)} +e^{ik(\theta_1-\theta_2-\pi)}\big)\Big).
  \end{split}
  \end{equation}
  Using the identity
  \begin{equation}
    \sum_{k=1}^\infty e^{ikz}=
    \frac{e^{iz}}{1-e^{iz}},\qquad \mathrm{Im} z>0,
  \end{equation}
  and $\alpha\in(1,1)$, we conclude
  \begin{align*}
  &-\sum_{k\in\Z}\sin(\pi|k-\alpha|)e^{-s|k-\alpha|}e^{ik(\theta_1-\theta_2)}\\\nonumber
  =&\sin(\alpha\pi)e^{-\alpha s}+\frac{\sin(\alpha\pi)}2\Big(\frac{e^{-(1+\alpha)s-i(\theta_1-\theta_2+\pi)}}
  {1-e^{-s-i(\theta_1-\theta_2+\pi)}}+\frac{e^{-(1+\alpha)s-i(\theta_1-\theta_2-\pi)}}{1-e^{-s-i(\theta_1-\theta_2-\pi)}}\\\nonumber
  &\qquad\qquad\qquad-\frac{e^{-(1-\alpha)s+i(\theta_1-\theta_2+\pi)}}{1-e^{-s+i(\theta_1-\theta_2+\pi)}}-
  \frac{e^{-(1-\alpha)s+i(\theta_1-\theta_2-\pi)}}{1-e^{-s+i(\theta_1-\theta_2-\pi)}}\Big)\\\nonumber
  =&\sin(|\alpha|\pi)e^{-|\alpha| s}+\sin(\alpha\pi)\frac{(e^{- s}-\cos(\theta_1-\theta_2+\pi))\sinh(\alpha s)
  -i\sin(\theta_1-\theta_2+\pi)\cosh(\alpha s)}{\cosh( s)-\cos(\theta_1-\theta_2+\pi)}\\
  =&B_{\alpha}(s)+C_{\alpha}(s,\theta_1-\theta_2+\pi).
  \end{align*}
  Plugging this into \eqref{aim-D} and letting 
  $\epsilon\to 0^+$, we obtain 
  $D_\alpha(t;r_1,\theta_1,r_2,\theta_2)$ 
  and \eqref{equ:Skernel2} is proved.
\end{proof}

\begin{proposition}\label{prop:dis-S}
  The kernel $\tilde{\bf K}_S$ of the propagator
  $e^{itL}$ satisfies the pointwise estimate
  \begin{equation}\label{est:dis-S1}
    |\tilde{\bf K}_S(t,r_1,\theta_1,r_2,\theta_2)|\le
    C|t|^{-1}
  \end{equation}
\end{proposition}

\begin{proof}
  By Proposition \ref{prop:kerenel-S}, 
  we must only prove that 
  $|G_{\alpha}|+|D_{\alpha}|\lesssim|t|^{-1}$.
  The bound for $G_{\alpha}$ follows immediately from
  \eqref{G-term}.
  To control $D_\alpha$, from \eqref{D-term}, we must prove
  \begin{equation}
    \int_0^\infty  e^{-|\alpha|s} ds \lesssim 1,
  \end{equation}
  \begin{equation}
  \int_0^\infty  
  \Big|\frac{(e^{- s}-\cos(\theta_1-\theta_2+\pi))
    \sinh(\alpha s) }{\cosh( s)-\cos(\theta_1-\theta_2+\pi)}
    \Big| ds \lesssim 1,
  \end{equation}
  and
  \begin{equation}
  \int_0^\infty  \Big|\frac{\sin(\theta_1-\theta_2+\pi)\cosh(\alpha s)}{\cosh( s)-\cos(\theta_1-\theta_2+\pi)}\Big| ds \lesssim 1,
  \end{equation}
  all of which are proved in \cite[(3.20)-(3.21)]{FZZ}.
\end{proof}

\subsection{The modified heat propagator} 

We next compute the heat kernel of $L$, that is the
integral kernel $\tilde{\bf K}_H$ of $e^{-tL}$.
We have the following result, where as usual we use
polar coordinates
$x=r_1(\cos\theta_1,\sin\theta_1)$, 
$y=r_2(\cos\theta_2,\sin\theta_2)$.
We omit the proof since the
proof is the same as for Proposition \ref{prop:kerenel-S},
with minor modifications.

\begin{proposition}\label{prop:kerenel-H} 
  Let $\alpha\in(0,1)$. Then 
  the integral kernel $\tilde{\bf K}_H$ of $e^{-tL}$
  can be written
   \begin{equation}\label{bf-KH1}
    e^{-tL}(x,y)=
    \tilde{\bf K}_H(t,r_1,\theta_1,r_2,\theta_2)= \left(\begin{matrix} G^H_{\alpha}+D^H_{\alpha}& 0 \\0 & G^H_{\alpha}+D^H_{\alpha}\end{matrix}\right)
  \end{equation}
  where
 \begin{equation}\label{G-H-term}
    G^H_\alpha(t;r_1,\theta_1,r_2,\theta_2) =
    \frac{e^{-\frac{|x-y|^2}{4t}} }{t}
    A_{\alpha}(\theta_1,\theta_2),
  \end{equation}
\begin{equation}\label{D-H-term}
  D^H_{\alpha}(t;r_1,\theta_1,r_2,\theta_2)
  =-\frac{1}{4\pi^2}\frac{e^{-\frac{r_1^2+r_2^2}{4t}} }{t}  
  \int_0^\infty e^{-\frac{r_1r_2}{2t}\cosh s} 
  \Big(B_{\alpha}(s)+
  C_{\alpha}(s,\theta_1-\theta_2+\pi)\Big) ds,
\end{equation}
ans $A_{\alpha}$, $B_{\alpha}$ and $C_{\alpha}$ 
are defined by \eqref{equ:A},\eqref{equ:BC} respectively.
\end{proposition}

Similarly as in Proposition \ref{prop:dis-S}, 
we obtain the following pointwise bound:

\begin{proposition}[Heat kernel estimates]
  The heat kernel $e^{-tL}$ satisfies the estimate
  \begin{equation}\label{est:heat1}
    |e^{-tL}(x,y)|=
    \big|  \tilde{\bf K}_H(t,r_1,\theta_1,r_2,\theta_2)\big|\leq 
     C t^{-1}e^{-\frac{|x-y|^2}{4t}},\quad t>0.
  \end{equation}
\end{proposition}

\subsection{Bernstein inequalities and square function estimates }\label{sec:BS}

We prove next the Bernstein and square function 
inequalities associated with the heat kernel $e^{-tL}$.

\begin{proposition}[Bernstein inequalities]
  Let $\varphi(\lambda)$ be a $C^\infty_c$ function on 
  $\R$  with support in $[\frac{1}{2},2]$.
  Then for all $f\in L^q(\mathbb{R}^2)$, $j\in\mathbb{Z}$
  we have
  \begin{equation}\label{est:Bern1}
    \|\varphi(2^{-j}\sqrt{L})f\|_{[L^p(\mathbb{R}^2)]^2}
    \lesssim
    2^{2j\big(\frac{1}{q}-\frac{1}{p}\big)}
    \|f\|_{[L^q(\mathbb{R}^2)]^2},
    \qquad 1\leq q\leq p\leq\infty.
  \end{equation}
\end{proposition}

\begin{proof}
  Let  $\psi(x)=\varphi(\sqrt{x})$ and $\psi_e(x):=\psi(x)e^{2x}$. 
  Then $\psi_e$  is $C^\infty_c$ with support in 
  $[\frac{1}{4},4]$ and its Fourier transform $\hat{\psi}_e$ 
  is of Schwartz class. We can write
  \begin{equation*}
    \varphi(\sqrt{x})=\psi(x)= e^{-2x}\psi_{e}(x)=
    e^{-2x}\int_{\R} e^{i x \cdot\xi} \hat{\psi}_e(\xi)\, d\xi
    =e^{-x}\int_{\R} e^{-x(1-i\xi)} \hat{\psi}_e(\xi)\, d\xi.
  \end{equation*}
  By the functional calculus we get
  \begin{equation*}
    \varphi(\sqrt{L})=\psi(L)=
    e^{-2^{j}L}\int_{\R} e^{-(1-i\xi)L} \hat{\psi}_e(\xi)\, d\xi
  \end{equation*}
  and
  \begin{equation*}
    \varphi(2^{-j}\sqrt{L})= \psi(2^{-2j}L)=
    e^{-2^jL}\int_{\R} e^{-(1-i\xi)2^{-2j}L} 
    \hat{\psi}_e(\xi)\, d\xi.
  \end{equation*}
  Applying \eqref{est:heat1} with $t=2^{-2j}$, we obtain
  for all $N>1$
  \begin{align*}
    \Big|\varphi(2^{-j}\sqrt{L})(x,y)\Big| &
    \lesssim 
    2^{4j}\int_{\R^2}e^{-\frac{2^{2j}|x-z|^2}{C}}
      e^{-\frac{2^{2j}|y-z|^2}{C}}\, dz \int_{\R} \hat{\psi}_e(\xi)\, d\xi\\
    & \lesssim 2^{2j}\int_{\R^2}e^{-\frac{|2^jx-z|^2}{C}}
    e^{-\frac{|2^jy-z|^2}{C}}\, dz \lesssim 2^{2j}(1+2^j|x-y|)^{-N},
  \end{align*}
  where we used the inequality 
  $|x-z|^2+|y-z|^2\geq \frac12|x-y|^2$ with $x,y\in\R^2$. 
  Then by Young's inequality we deduce \eqref{est:Bern1}.
\end{proof}

\begin{proposition}[Square function inequality]\label{prop:squarefun} 
  Let $\varphi\in C_c^\infty$, with support in $[1/2,1]$
  and such that $0\leq\varphi\leq 1$ and
  \begin{equation}\label{LP-dp}
    \sum_{j\in\Z}\varphi(2^{-j}\lambda)=1
    \qquad\text{where}\qquad 
    \varphi_j(\lambda):=\varphi(2^{-j}\lambda).
  \end{equation}
  Then for all $1<p<\infty$
  there exist constants $c_p$ and $C_p$ depending on $p$ such that
  \begin{equation}\label{square1}
  c_p\|f\|_{[L^p(\R^2)]^2}\leq
  \Big\|\Big(\sum_{j\in\Z}|\varphi_j(\sqrt{L})f|^2\Big)^{\frac12}\Big\|_{[L^p(\R^2)]^2}\leq
  C_p\|f\|_{[L^p(\R^2)]^2}.
  \end{equation}
\end{proposition}

\begin{proof}
  Estimate \eqref{square1} is a consequence of the Gaussian
  heat kernel upper bound proved in \eqref{est:heat1};
  see \cite{Alex} for a general approach to such estimates.
\end{proof}

\subsection{Proof of \eqref{est:dis1-1}}

We prove the decay estimate \eqref{est:dis1-1}.
We distinguish the two cases $2^jt\leq 1$ and $2^jt\geq 1$.

\underline{{\bf Case 1: $2^jt\leq 1$.}} 
It is sufficient to prove that
\begin{equation}\label{est:dis1-1'}
  \Big\|\varphi(2^{-j}\sqrt{L})
  e^{it\sqrt{L}}f(x)\Big\|_{[L^\infty(\R^2)]^2}\leq 
  C2^{2j} \|f\|_{[L^1{(\R^2)]^2}}.
\end{equation}
Indeed, by the Bernstein inequality \eqref{est:Bern1} and the boundedness of $e^{it\sqrt{L}}$ on $L^{2}$, we obtain
\begin{equation}
\begin{split}
  \Big\|\varphi(2^{-j}\sqrt{L})e^{it\sqrt{L}}f(x)&
    \Big\|_{[L^\infty(\R^2)]^2}\leq 
    C 2^{j}\Big\|e^{it\sqrt{L}}\varphi(2^{-j}
      \sqrt{L})f(x)\Big\|_{[L^2(\R^2)]^2}\\
  &\leq C 2^{j}\Big\|\varphi(2^{-j}\sqrt{L})f(x)
    \Big\|_{[L^2(\R^2)]^2}
  \leq C2^{2j} \|f\|_{[L^1{(\R^2)]^2}}.
\end{split}
\end{equation}

\underline{ {\bf Case 2: $2^jt\geq 1$.}}
This is slightly more complicated. We first recall the following Proposition about the subordination formula from \cite{MS, DPR}.

\begin{proposition}\label{prop:sub0}
  If $\beta(\lambda)\in C_c^\infty(\mathbb{R})$ is supported in $[\frac{1}{2},2]$, then for all $j\in\Z$ and
  $t, x>0$ with $2^jt\geq 1$, we can write
  \begin{equation}\label{key}
    \beta(2^{-j}\lambda)e^{it\lambda}=
    \beta(2^{-j}\lambda)\rho\big(\frac{t\lambda^2}{2^j}, 2^jt\big)
    +\beta(2^{-j}\lambda)\big(2^jt\big)^{\frac12}
    \int_0^\infty \chi(s,2^jt)
    e^{\frac{i2^jt}{4s}}e^{i2^{-j}ts\lambda^2}\,ds,
  \end{equation}
  where $\rho(s,\tau)\in\mathcal{S}(\mathbb{R}\times\mathbb{R})$ is a Schwartz function and 
  $\chi\in C^\infty(\mathbb{R}\times\mathbb{R})$ with $\text{supp}\,\chi(\cdot,\tau)\subseteq[\frac{1}{16},4]$ such that
  \begin{equation}\label{est:chi}
  \sup_{\tau\in\R}\big|\partial_s^\alpha\partial_\tau^\beta \chi(s,\tau)\big|\lesssim_{\alpha,\beta}(1+|s|)^{-\alpha},\quad \forall \alpha,\beta\geq0.
  \end{equation}
\end{proposition}

As a direct consequence, we obtain:

\begin{proposition}\label{prop:sub1}
  If $\varphi(\lambda)\in C_c^\infty(\mathbb{R})$ is supported 
  in $[\frac{1}{2},2]$, then for all $j\in\Z$
  and $t,x>0$ with $2^jt\geq 1$ we can write
  \begin{equation}\label{key}
    \begin{split}
    &\int_0^\infty \varphi(2^{-j}\lambda)e^{it\lambda} J_{\nu}(\lambda r_1)J_{\nu}(\lambda r_2)\, \lambda d\lambda\\&=\int_0^\infty  \varphi(2^{-j}\lambda)\rho\big(\frac{t\lambda^2}{2^j}, 2^jt\big) J_{\nu}(\lambda r_1)J_{\nu}(\lambda r_2)\, \lambda d\lambda
    \\
    &+\int_0^\infty \varphi(2^{-j}\lambda)\big(2^jt\big)^{\frac12}\int_0^\infty \chi(s,2^jt)e^{\frac{i2^jt}{4s}}e^{i2^{-j}ts\lambda^2}\,ds J_{\nu}(\lambda r_1)J_{\nu}(\lambda r_2)\, \lambda d\lambda,
    \end{split}
  \end{equation}
  where $\rho(s,\tau)\in\mathcal{S}(\mathbb{R}\times\mathbb{R})$ 
  is a Schwartz function
  and $\chi\in C^\infty(\mathbb{R}\times\mathbb{R})$ 
  with $\text{supp}\,\chi(\cdot,\tau)\subseteq[\frac{1}{16},4]$,
  such that
  \begin{equation}\label{est:chi}
    \sup_{\tau\in\R}\big|\partial_s^\alpha\partial_\tau^\beta \chi(s,\tau)\big|\lesssim_{\alpha,\beta}(1+|s|)^{-\alpha},\quad \forall \alpha,\beta\geq0.
  \end{equation}
\end{proposition}

Then, by the spectral theory for the non-negative self-adjoint operator $L$, we have the representation of the microlocalized
half-wave propagator, if $2^jt\geq 1$:
\begin{equation}\label{key-operator}
  \varphi(2^{-j}\sqrt{L})e^{it\sqrt{L}}
  =\varphi(2^{-j}\sqrt{L})\rho\big(\frac{tL}{2^j}, 2^jt\big)
  +\varphi(2^{-j}\sqrt{L})\big(2^jt\big)^{\frac12}
  \int_0^\infty \chi(s,2^jt)
  e^{\frac{i2^jt}{4s}}e^{i2^{-j}tsL}\,ds.
\end{equation}
Since $\rho(s,\tau)\in\mathcal{S}(\mathbb{R}\times\mathbb{R})$ is a Schwartz function, by a similar argument as above we get
\begin{equation}
  \Big\|\varphi(2^{-j}\sqrt{L})\rho
  \big(\frac{tL}{2^j}, 2^jt\big)f(x)
  \Big\|_{[L^\infty(\R^2)]^2}\leq 
  C2^{2j}(1+2^j|t|)^{-N} \|f\|_{[L^1{(\R^2)]^2}}.
\end{equation}
By using \eqref{est:dis-S1} from Proposition \ref{prop:dis-S}, 
we have
\begin{equation}\label{est:dis1-1''}
  \begin{split}
  &\Big\|\varphi(2^{-j}\sqrt{L})
    \big(2^jt\big)^{\frac12}\int_0^\infty \chi(s,2^jt)
    e^{\frac{i2^jt}{4s}}e^{i2^{-j}tsL}\,ds 
    f(x)\Big\|_{[L^\infty(\R^2)]^2}\\
  &\leq C \big(2^jt\big)^{\frac12} 
  \int_0^\infty \chi(s,2^jt)\big(2^{-j}|ts|\big)^{-1}\,ds 
  \|f\|_{[L^1{(\R^2)]^2}}\\
  &\leq C2^{\frac32j}|t|^{-1/2} \|f\|_{[L^1{(\R^2)]^2}}\leq 
  C2^{2j}(1+2^j|t|)^{-\frac12} \|f\|_{[L^1{(\R^2)]^2}}.
  \end{split}
\end{equation}
Putting together \eqref{est:dis1-1'} and \eqref{est:dis1-1''} 
we complete the proof of \eqref{est:dis1-1}. \vspace{0.2cm}

{\bf Conflicts of Interest Statement:}
The authors declare that there are no conflicts of interest relevant to the content of this manuscript. The research was conducted without any commercial or financial relationships that could be construed as a potential conflict of interest.

{\bf Data Availability Statement:}
The data supporting the findings of this study are available from the corresponding author upon reasonable request.

\end{document}